\documentclass[12pt]{article}
\usepackage{amsmath,amsfonts,amsthm,amscd,amssymb,latexsym,comment}
\usepackage{graphicx,psfrag,subfigure,changebar,oldgerm,MnSymbol,mathtools}
\DeclarePairedDelimiter{\ceil}{\lceil}{\rceil}
\DeclarePairedDelimiter{\floor}{\lfloor}{\rfloor}
\usepackage[shortlabels]{enumitem}
\usepackage{xcolor}
\usepackage{etoolbox}
\usepackage{tikz}
\usetikzlibrary{arrows.meta}
\usetikzlibrary{backgrounds}
\usetikzlibrary{patterns,decorations.markings,calc}
\newtoggle{brief}

\toggletrue{brief}
\togglefalse{brief}
\title{One-relator quotients of Partially Commutative Groups
}
\author{ {Andrew J. Duncan}
 \and 
Arye Juh\'{a}sz}
\def\D{\Delta }
\def\d{\delta }
\def\b{\beta }
\def\y{\zeta }
\def\i{\iota }

\def\l{\lambda }
\def\L{\Lambda }
\def\r{\rho }
\def\x{\chi }

\def\e{\varepsilon }
\def\G{\Gamma }
\def\g{\gamma }
\def\a{\alpha }
\def\t{\tau }

\def\s{\sigma }
\def\S{\Sigma }

\def\pd{\partial}

\newtheorem{theorem}{Theorem}[section]
\newtheorem{lemma}[theorem]{Lemma}
\newtheorem{corollary}[theorem]{Corollary}
\newtheorem{prop}[theorem]{Proposition}

\theoremstyle{definition}
\newtheorem{defn}[theorem]{Definition}
\newtheorem{exam}[theorem]{Example}
\newenvironment{example}{\begin{exam} \rm}{\end{exam}}
\numberwithin{equation}{section}
\newcommand{\bs}{\backslash}
\newcommand{\ZZ}{\ensuremath{\mathbb{Z}}}

\newcommand{\GG}{\ensuremath{\mathbb{G}}}
\newcommand{\FF}{\ensuremath{\mathbb{F}}}
\renewcommand{\AA}{\ensuremath{\mathbb{A}}}

\newcommand{\supp}{\operatorname{supp}}
\newcommand{\lk}{\operatorname{lk}}
\newcommand{\st}{\operatorname{st}}
\newcommand{\hnn}{\operatorname{HNN}}

\newcommand{\maln}{\operatorname{Maln}}
\newcommand{\cD}{\mathcal{D}}

\newcommand{\cR}{\mathcal{R}}
\newcommand{\cS}{\mathcal{S}}
\newcommand{\la}{\langle}
\newcommand{\ra}{\rangle}
\newcommand{\maps}{\rightarrow}
\newcommand{\sep}{\operatorname{Sep}}
\newcommand{\Li}{L^{\ref{it:composed1}}}
\newcommand{\Lii}{L^{\ref{it:composed2}}}
\newcommand{\li}{l^{\ref{it:composed1}}}
\newcommand{\lii}{l^{\ref{it:composed2}}}

\newcommand{\be}{\begin{enumerate}}
\newcommand{\ee}{\end{enumerate}}
\newcommand{\biz}{\begin{itemize}}
\newcommand{\eiz}{\end{itemize}}
\newenvironment{ajd1}{\noindent\color{blue}}{}
\newcommand{\ajd}[1]{\begin{ajd1} #1 \end{ajd1}}
\newenvironment{arye1}{\noindent\color{blue}}{}

\begin{document}
\maketitle
\begin{abstract}
  We generalise a key result of one-relator group theory, namely
  Magnus's Freiheitssatz, to partially commutative groups, under
  sufficiently strong conditions on the relator. The main theorem shows that under
  our conditions, on an element $r$ of a partially commutative group $\GG$,
  certain Magnus subgroups embed in the quotient $G=\GG/N(r)$; that if $r=s^n$ has root $s$ in $\GG$ then
   the order of $s$ in $G$ is $n$, and under slightly stronger conditions that the
  word problem of $G$ is decidable. We also give conditions under which the question of which Magnus
  subgroups of $\GG$ embed in $G$ reduces to the same question in the minimal parabolic subgroup of $\GG$ containing $r$.
  In many cases this allows us to characterise Magnus subgroups which embed in $G$, via a condition on $r$ and the
  commutation graph of $\GG$, and to find further examples of quotients $G$ where the word and conjugacy problems are decidable.
  We give evidence that situations in which our main theorem applies are not uncommon, by proving that for cycle graphs
  with a chord $\G$, almost all cyclically reduced elements of the partially commutative group $\GG(\G)$ satisfy the conditions of the theorem.
\end{abstract}
{\let\thefootnote\relax}
{\let\thefootnote\relax}
\section{Introduction}
\label{sec:intro}
% \brief{
%\ajd{Magnus subgroups}~\\
Partially commutative groups have been extensively studied
in several different guises and are variously known as semi-free groups; 
right-angled Artin groups; trace groups;
graph groups or even locally free groups.  
Let $\G$ be a finite, undirected, simple graph. Let $A=V(\G)$ be the set of vertices
of $\G$ and let $\FF(A)$ be the free group on $A$.
For elements $g,h$ of a group we denote the commutator $g^{-1}h^{-1}gh$
of $g$ and $h$ by $[g,h]$. Let
\[
R=\{[x_i,x_j]\in \FF(A)\mid x_i,x_j\in A \textrm{ and there is an edge from } x_i \textrm{ to } x_j
\textrm{ in } \G\}.
\]
We define
the {\em free partially commutative group with (commutation) graph } $\G$ to be the 
group $\GG=\GG(\G)$ with free presentation
$\left< A\mid R\right>.$
 We shall refer to
finitely generated free partially commutative groups as {\em partially
  commutative} groups. (Strictly speaking we have defined the class of \emph{finitely generated} partially commutative groups.) 
 The class of partially commutative groups contains
finitely generated free, and free Abelian groups; has provided several crucial examples in the theory 
of finitely presented groups and has applications both in mathematics and computer science. For an introduction and survey of the literature see \cite{Ch} or \cite{EKR}.  It emerges, from the work of Sageev, Haglund, Wise, Agol and others, that many 
well-known families of groups  virtually embed into partially commutative groups: among these 
are Coxeter groups, certain one-relator groups with torsion, limit groups, and fundamental groups of closed $3$-manifold groups 
(see for example \cite{W} for details and references). It is therefore natural to consider related classes, such as 
their one-relator quotients as we do here. %, as initiated by Antolin and Kar \cite{AK}. 

There are generalisations of one-relator group theory in several directions; for instance  
to two-relator groups, to one-relator quotients of surface groups \cite{HS, ADL} and  
very successfully to one-relator quotients of free products of groups, see e.g. \cite{%Br1,
Br2,%DH,
H1,H2} and see \cite{FR} for details and fuller references. 
A key result 
of one-relator group theory %, generalising the theory of free groups 
is Magnus's Freiheitssatz which states that if $\FF$ is the  free group of rank $n$,
%. This theorem,
%proved by Magnus in 1929, is that if $r$ is a cyclically reduced element of the free group
and $r$ is an element of  $\FF$, involving  every generator of $\FF$,  
 then the subgroup generated by $n-1$ of the generators embeds in the one-relator group $\FF/N(r)$. 
(We  use $N(r)$ to denote the normal closure of $r$ in a given group.)  
  The Freiheitssatz has several immediate
consequences: for instance if $r=s^m$, where $s$ is not a proper power in
 $\FF$ then $s$ has order precisely $m$ in $\FF/N$;  if  $m=1$ then $\FF/N$ is torsion-free;  
 if $m>1$ then any element of finite order in $\FF/N$ is conjugate to a power of $s$; and 
%(see for example \cite{FR} for further details).
%or example, it follows that 
 one-relator
groups have solvable word problem (see for example \cite{FR}). 
Turning to one-relator quotients of partially commutative groups, 
Antolin and Kar \cite{AK} proved that a Freiheitssatz holds for one-relator quotients of 
starred partially commutative groups: $\GG(\G)$ is \emph{starred} if $\G$ 
has no full subgraph isomorphic to the four cycle or the path graph on four vertices.
 More generally they 
prove versions of the  Freiheitssatz for one-relator quotients of starred partially commutative products and show that 
 the word problem is solvable, both for one-relator quotients of starred partially commutative groups and 
of starred partially commutative products of polycyclic groups.
 %As in the cases above, from the Freiheitssatz several consequences flow, such  as  the 
%solvability of the word problem, and the classification of torsion elements. 
 %Given this 

Here, with appropriate restrictions on the relator, 
%we apply small cancellation theory for HNN-extensions 
%to one-relator quotients of  partially commutative groups and   obtain results 
%by restricting the form of the relator.
%With appropriate restrictions
we show that 
a Freiheitssatz holds for one-relator quotients of any partially commutative group and that, with
 slightly stronger restrictions, the word problem is solvable. 
%
% straightforward analysis of  the structure of the group as an amalgamated free product also
% yields a Freiheitssatz. 
A summary of the situations in which a Freiheitssatz is known to hold, combining our results and those of \cite{AK}, is given in the
Examples and final summary of this section.  

In order to state our main results we need some definitions. 
%We use $N(r)$ to denote the normal closure of an element $r$ (in a group $G$ which is always apparent from the context). 
If 
$w$ is an element of the free monoid $(A\cup A^{-1})^*$ then we denote by $\supp(w)$ the set of
elements $x\in A$ such that $x$ or $x^{-1}$ occurs in $w$. For an element $g\in \GG$ we define
$\supp(g)$ to be equal to $\supp(w)$, where $w$ is an element of $(A\cup A^{-1})^*$ 
of minimal length amongst all those elements representing $g$. For more detail (including the 
fact that $\supp(g)$ is well-defined) see Section \ref{sec:prelim}. An element $g$ of $\GG$ is
said to be \emph{cyclically minimal} if it is represented by a word $w \in (A\cup A^{-1})^*$ 
of length no greater than any other word representing an element of the conjugacy class of $g$.
 (See Section \ref{sec:prelim}, Lemma \ref{lem:cycgeod}.) 

 A subset $Y$ of $A$ is called a \emph{clique}
 if the full subgraph $\G_Y$ of $\G$ generated by $Y$ is a complete graph and \emph{independent} if $\G_Y$ 
is a null graph.  
%If $Y$ is a subset of $A$ then  $Y$ generates an Abelian subgroup of $\GG$ if and only if the 
%full subgraph of $\G$ generated by $Y$ is a complete graph. On the other hand $Y$ generates a free subgroup
%if and only if $Y$ generates a null graph: see Section  \ref{sec:prelim} for details.  
For $x\in A$ define the \emph{link} of $x$ to 
be $\lk(x)=\{y\in A| d(x,y)=1\}$ (edges of the graph have length $1$) 
and the \emph{star} of $x$ to be $\st(x)=\lk(x)\cup \{x\}$. The definitions 
of $t$-thick and $t$-root appear in Section \ref{sec:tthick}, Definition \ref{def:tthick} and Section \ref{sec:redn}, 
Definition \ref{def:troot}, respectively. 
\begin{comment}
For a group $G$ given by a finite presentation $\la X\,|\,R\ra$ the $n$\emph{-fold dependence problem} DP$(n)$ asks, 
given an $n+1$-tuple $(w_0,w_1,\ldots, w_n)$ of $\FF(X)^{n+1}$, whether or not there exists an element $(u_1,\ldots, u_n)\in \FF(X)^n$
such that
\[w_0\prod_{i=1}^n u_i^{-1}w_iu_i=_G 1.\]
If there is an algorithm to solve DP$(n)$ (for all elements of $\FF(X)^{n+1}$) then the problem is said to be solvable in $G$.
(Solvability of DP$(n)$ is independent of the choice of  finite presentation for $G$.)  
The word problem is DP$(0)$ and the
conjugacy problem is DP$(1)$. If DP$(n)$ is solvable for all $n\ge 0$ then we say that DP$(\infty)$ is solvable in $G$. For further details see \cite{pride}.  
\end{comment}

To illustrate some of the possibilities we give some examples. 
\begin{example}\label{ex:abcase}
%In the notation above, with $s$ a cyclically minimal element of $\GG$ and $r=s^n$ for some positive integer $n$, assume $\G_s$ is a simplex or a null graph.
\be[label=\arabic*.,ref=\arabic*]
\item\label{it:abcase1}  If %$\G_s=\G$ then 
$\G$ is a complete or null graph then $\GG=\GG(\G)$ is free Abelian or free, respectively. In this case, 
if $s\in \GG$ is not a proper power, $t\in \supp(s)$ and $r=s^n$, for some positive integer $n$,    
then $\la A\bs \{t\}\ra$ embeds in $G=\GG/N(r)$ and $s$ has order $n$ in $G$; 
using standard results from the theory of  finitely generated Abelian groups, or of one-relator groups, as appropriate.
\item\label{it:abcase2new}
  Let $A=\{a,b,c\}$, $R=\{[a,b],[b,c]\}$ and
  %let $R_F=\{[a,c],[b,c],[x,c]\}$. 
  %Choose $R$ equal to $R_A$ or $R_F$ and set
  $\GG=\la A|R\ra\cong \FF_2\times \ZZ$. Let $s=abc$, let $N$ be the normal closure of $s^n$ in $\GG$, for some $n>0$,
  and let $G=\GG/N$.  
  Here $b\in \supp(s)$ and $r=(ac)^nb^n$, so $[(ac)^n,a]=1$ in $G$, but $[(ac)^n,a]\neq 1$ in $\GG$ (as $C_\GG(a)=\la a,b\ra$).
  Hence $\la A\bs \{b\}\ra=\la a,c\ra$ does not embed in $G$.
  \begin{comment}
\item\label{it:abcase4} In parts \ref{it:abcase2} and \ref{it:abcase3} the full subgraph on $\supp(s)$ is a complete graph. 
In the dual case, where $\supp(s)$ induces a null subgraph,  failure of 
embedding appears to be less common.  
 For instance, let $A=\{a,b,c,x\}$ and $R=\{[a,c],[b,c],[x,a]\}$ and let $\G$ be the commutation
graph of the group $\GG$ with presentation $\la A\,|\,R\ra$. 
Let $r=s=a^{-1}b^{-1}ab$, and $N=N(r)$, so $G=\GG/N$ is the partially commutative group with presentation 
$\la A\,|\,R\cup\{[a,b]\}\ra$. Let $\G_N$ be the commutation graph of $G$.  
In  this case $X=A\bs(\supp(s)\cup \lk(s))=\{x\}$ and so $x\in X$ and 
$x$ commutes with the element $a$ of $\supp(s)$. However $A_a=A\bs\{a\}=\{b,c,x\}$ and the full subgraphs 
of $\G$ and $\G_N$ on the vertices $A_a$ are identical. Hence the subgroup $\GG_a=\la A_a \ra$ of $\GG$ embeds as 
a subgroup of $G$.
\end{comment}
\ee
\end{example}

\begin{theorem}\label{thm:main}
Let $\GG=\la A|R\ra$ and 
let $s\in \GG$ 
be a cyclically minimal element such that, for some $t\in \supp(s)$, 
\be[label=\arabic*.,ref=\arabic*]
%\item
%$[v,t]\neq 1$ 
\item \label{it:main_hyp1}
$\lk(t)$ is a clique (or empty),
\item \label{it:main_hyp2}
  $s$ is $t$-thick,
\item \label{it:main_hyp3} $s\notin \la \st(t) \ra$  and 
\item \label{it:main_hyp4}
$s$ is a $t$-root.
\ee
 %%%
Then, for  $n\ge 3$, 
\be[label=(\alph*)]
\item\label{it:main1}
$\la A\bs\{t\}\ra$ embeds in $\GG/N$,
 where $N$ denotes the normal closure of $s^n$ in $\GG$, and 
\item\label{it:main2}
  the order of $sN$ in $\GG/N$ is $n$.
\ee
Moreover if $n\ge 4$ then the word problem is solvable in $\GG/N$.   %\ajd{This is what I wish for. At the moment, we have, at best, $DP(0)$.}
\end{theorem}

%In the case where $\lk(t)$ is empty we have $\GG=\la A\bs\{t\}\ra*\la t\ra$ and so we may invoke 
%standard results of one-relator products of groups. 
%Theorem \ref{thm:main} applies only to the case where $\supp(s)$ is not a simplex.
Given a cyclically reduced element $r$ of the free group $\FF(A)$ on $A$, a \emph{Magnus subgroup} is a subgroup generated by a subset
$B$ of $A$ such that $r\notin\la B\ra$.  
As given above, Magnus's Freiheitssatz for one-relator groups applies when $r$ involves every generator of the
free group in question so every proper subset of $A$ generates a Magnus subgroup; which embeds in the quotient by the normal closure of $r$.
In the free group setting the general case reduces to this special case: indeed, if $r$ is an arbitrary
element of $\FF(A)$ and $Y=\supp(r)\subseteq A$ with $X=A\bs Y$ non-empty
then we have $\FF(A)/N(r)\cong [\FF(Y)/M(r)]*\FF(X)$, where $N(r)$ and $M(r)$ are the  normal closures of $r$ in $\FF(A)$ and $\FF(Y)$,
respectively. From the Freiheitssatz as stated above, if $Y'$ is a proper subset of $Y$ then
the subgroup $\la Y'\ra$ of $\FF(Y)$ embeds in $\FF(Y)/M(r)$, so $\la X\cup Y'\ra$ embeds in $\FF(A)/N(r)$.
Thus the general Freiheitssatz for Magnus subgroups of  $\FF(A)$ reduces to the  special case of Magnus subgroups for $\FF(Y)$.  

To obtain an analogous reduction for one-relator quotients of partially commutative groups we  first note that
it follows from work of B.~Baumslag and S.J.~Pride \cite{BP}, together with the fact, proved by Diekert and Mushcoll \cite{DM}, that
equations are decidable over partially commutative groups; 
that if $\GG=\GG_1*\cdots *\GG_k$ is a free product (so $\G$ has  $k$ connected components)
and $\supp(r)$ contains vertices in at least two components of $\G$, then each of the $\GG_i$ embeds in $G$. 
In general, we consider the case where 
% $\lk(s)$ forms a vertex cut-set for the graph $\G$ and $\supp(s)$ is a union of components of $\G_{A\bs\lk(s)}$. With such a restriction we obtain a decomposition of
the
group $\GG$ decomposes as a free product with amalgamation, and one factor is generated by $\supp(r)$.
% , one consequence of which is a lifting of the Freiheitssatz.
More precisely, we make the following definition.
\begin{defn}
  For a subset $B$ of $A$ define $\lk(B)=\cap_{b\in B}\lk(b)$ and for $g \in \GG$ define $\lk(g)=\lk(\supp(g))$. 
  A subset $Y$ of $A$ is called \emph{synchronised} if, 
  % satisfies the property that,
  for all vertices $v$ of $Y$, the star  $\st(v)$ of $v$ is a subset
  of $Y\cup \lk(Y)$.
\end{defn}
If $Y$ is synchronised then, writing
$X=A\bs\{Y\cup \lk(Y)\}$, it follows also that, for all vertices $v$ of $X$, the star $\st(v)$ is a subset
of $X\cup \lk(Y)=A\bs Y$. Hence, if $Y$ is synchronised, setting $\AA_0=\la Y\cup \lk(Y)\ra$, $\AA_1=\la X\cup \lk(Y)\ra$  and
$U=\la \lk(Y)\ra$, it follows that $\GG=\AA_0*_{U}\AA_1$. 

\begin{theorem}\label{thm:amalgam} 
Let $\GG=\la A|R\ra$ and  
let $s\in \GG$
be a cyclically minimal element such that $\supp(s)$ is synchronised. %$\G_s$ has diameter one. 
%%%
Let %$Y=\supp(s)$ and, as above, let
$\AA_0=\la \supp(s)\cup \lk(s)\ra$, $\AA_1=\la A\bs \supp(s)\ra$   and
$U=\la \lk(s)\ra$; so $\GG=\AA_0*_{U}\AA_1$. 
Let $r=s^n$, for some $n\ge 1$, let $K=\la \supp(s)\ra$, 
denote by $N$ the normal closure of $r$ in $\GG$ and by $M$ the normal closure of $r$ in $K$, and  let $G=\GG/N$. %  and let $G_s=K/M$. % and let $t\in \supp(s)$.
Then the following hold.
\be[label=(\alph*)]
\item \label{it:amalgam1}
 If $Y'$ is a subset of $\supp(s)$ such that the subgroup $\la Y'\ra$ of $K$ embeds in $K/M$ then  
 $\la Y'\cup (A\bs\supp(s))\ra$ embeds in $G$; 
 \item \label{it:amalgam2}
   if $s$ has order $n$ in $K/M$ then $s$ has order $n$ in $G$,
 \item \label{it:amalgam3} if the word problem is solvable in $K/M$ then the word problem is solvable in $G$ and
 \item \label{it:amalgam4}  if the conjugacy problem is solvable in $K/M$ then the conjugacy problem is solvable in $G$.
\ee
\end{theorem}
In terms of Magnus subgroups of $\GG$ (defined exactly as for free groups) with the notation, and under the hypotheses, of the theorem,
a Magnus subgroup
generated by a subset $B$ of $A$ embeds in $G$ if and only if the Magnus subgroup $\la B\cap \supp(s)\ra$ of $K$ embeds in $K/M$.

In cases where $\supp(s)$ is a %simplex, using  
% the notion of synchronised
 clique  %(see Section \ref{sec:simplex}, Definition \ref{defn:compress})  
we obtain a  complete characterisation of the conditions 
under which all Magnus subgroups of $\GG$ embed in $G$. %a Freiheitssatz holds.
\begin{corollary}\label{cor:Abfrei} 
Let $\GG=\la A|R\ra$ and  
let $s\in \GG$
be a cyclically minimal element such that $\supp(s)$ is a clique. %$\G_s$ has diameter one. 
 %%%
Let $r=s^n$, for some $n\ge 1$, 
denote by $N$ the normal closure of $r$ in $\GG$ and  let $G=\GG/N$. % and let $t\in \supp(s)$.
Then the following are equivalent. 
\be
\item 
$\GG_t=\la A\bs\{t\}\ra$ embeds in $G$, for all $t\in \supp(s)$. 
%if and only if %, for all 
 %$\st(t)=(\lk(s)\cup \supp(s))$. 
\item  $\supp(s)$ is synchronised. 
\ee
%<<<<<<< Updated upstream
Moreover, if $\supp(s)$ is synchronised then $s$ has order $n$ in $G$, and the word and conjugacy problem
are decidable in $G$. 
%=======
%Moreover, if $\supp(s)$ is synchronised then $s$ has order $n$ in $G$, and the word and conjugacy problems are solvable in $G$. 
%>>>>>>> Stashed changes
%\ajd{... and if $n=1$ then $G$ is torsion free?}
\end{corollary}
\begin{proof}
Suppose $\supp(s)$ is a clique and not synchronised. %Write $X=A\bs(\supp(s)\cup \lk(s))$ and let $K=\la \supp(s)\ra$. 
%Then from Example \ref{ex:abcase}.\ref{it:nofri}, $\GG_t$ does not embed in $G$.
%%%%%%%%%%%%%%%here
% Generalising the final part of the previous example, for arbitrary $\GG$, assume $s$ is chosen so that 
%  $\G_s$ is complete (and $K$ is Abelian).
 This means  there exists $t\in \supp(s)$ and 
$x\in A\bs(\supp(s)\cup \lk(s))$ such that $[x,t]=1$.  %Then $\GG_t$ does not embed in $G$. To see
%this, note that
Since $x\notin \lk(s)$, there is $a\in \supp(s)$ such that $[x,a]\neq 1$. As $K=\la\supp(s)\ra$ is
 Abelian we have $r=s^n=t^ma^nw$, for some $w$ such that $\supp(w)=\supp(s)\bs\{a,t\}$. Then 
\[x(a^nw)=_G xt^{-m}=t^{-m}x=_G (a^nw)x.\]
Hence $[x,a^nw]=_G1$ while, as $[x,a]\neq 1$,  $[x,a^nw]\neq 1$ in $\GG$. As $[x,a^nw]\in \la A\bs\{t\}\ra$, this implies that $\la A\bs\{t\}\ra$
does not embed in $G$.

The converse follows from Theorem \ref{thm:amalgam}, since when $K$ is Abelian $\la \supp(s)\bs\{t\}\ra$ embeds in $K/M$, for all
$t\in \supp(s)$, and $s$ has order $n$ in $K/M$; as in Example \ref{ex:abcase}.\ref{it:abcase1}.  As $K/M$ is a finitely
generated Abelian group, it has solvable conjugacy problem, so the remaining statement follows from
Theorem \ref{thm:amalgam} \ref{it:amalgam4}. 
\end{proof}
In the dual case, where  $\supp(s)$ is independent,  
we obtain a sufficient, but not necessary, condition for 
all Magnus subgroups of $\GG$ to embed in $G$. %a full Freiheitssatz to hold.  
\begin{corollary}\label{cor:Freefrei}
Let $\GG=\la A|R\ra$ and  
let $s\in \GG$
be a cyclically minimal element such that $\supp(s)$ is independent. %$\G_s$ has diameter one. 
 %%%
Let $r=s^n$, for some $n\ge 1$, 
denote by $N$ the normal closure of $r$ in $\GG$ and   let $G=\GG/N$. % and let $t\in \supp(s)$.
If $\supp(s)$ is synchronised  
then 
$\GG_t=\la A\bs\{t\}\ra$ embeds in $G$, for all $t\in \supp(s)$,  
$s$ has order $n$ in $G$ and the word problem is solvable in $G$. Moreover, if $n\ge 2$ then the conjugacy
problem is solvable in $G$. 
%\ajd{... and if $n=1$ then $G$ is torsion free?}
\end{corollary}
\begin{proof}
This follows from Theorem \ref{thm:amalgam}, since when $K$ is free $\la \supp(s)\bs\{t\}\ra$ embeds in $K/M$, for all
$t\in \supp(s)$ and $s$ has order $n$ in $K/M$; as in Example \ref{ex:abcase}.\ref{it:abcase1}. As $K/M$ is a one-relator
group it has solvable word problem, and if $n\ge 2$ then solvable conjugacy problem; so the final statement follows from
Theorem \ref{thm:amalgam} \ref{it:amalgam3} and \ref{it:amalgam4}. 
\end{proof}
%(In \cite{DKR3} a synchronised subset $B$ of $A$ which is either a clique or independent is called \emph{compressible}.   In this case $B$ maps
%to a single vertex of a ``compressed'' quotient of $\G$.)
\begin{example}\label{ex:thm_main} 
\be[label=\arabic*.,ref=\arabic*]
\item  Corollaries \ref{cor:Abfrei} and \ref{cor:Freefrei} generalise Example \ref{ex:abcase}.\ref{it:abcase1}.
  %the hypotheses and conclusions of Theorems \ref{thm:Abfrei} or Theorem \ref{thm:Freefrei} hold.
  In %Example \ref{ex:abcase}.\ref{it:abcase2} the hypotheses and conclusions of Theorem \ref{thm:Abfrei} hold, while in
  Example \ref{ex:abcase}.\ref{it:abcase2new}, $\supp(s)=A$ so, although it is synchronised,  
  Theorem \ref{thm:amalgam} gives no new information. In this example $\lk(b)$ is not a clique but $\lk(c)=\lk(a)=\{b\}$, a clique.
  However (taking $c=t$ in Theorem \ref{thm:main}) we have $s=abc$ and $ab\notin \maln(\la b\ra)$, so $s$ is not $c$-thick, and 
  Theorem \ref{thm:main} does not apply.
  %hypotheses and conclusions of Theorem \ref{thm:Abfrei} fail. Example  \ref{ex:abcase}.\ref{it:abcase4} illustrates Theorem \ref{thm:Freefrei}. 
  \item\label{it:p3}
Let $A=\{a,b,c,t\}$ and $R=\{[t,a],[a,b],[b,c]\}$ so $\GG=\la A\,|\,R\ra$ has commutation graph the path graph 
$P_4$ of length $3$ (and $4$ vertices). If $s=ct$ then $\lk(t)$ is a clique and $s$ satisfies the hypothesis of Theorem \ref{thm:main}. 
Setting $r=s^3$ it follows that $\la a,b,c\ra$ embeds in $\GG/N(r)$. Note that, since $G=\GG/N(s)$ is isomorphic to $\ZZ^3$, 
the subgroup $\la a,c\ra$ of $\GG$ does not embed in $G$; showing that  
the condition $n\ge 3$  in 
Theorem \ref{thm:main} cannot be entirely removed. %The  technique used to prove Theorem \ref{thm:main} applies when  
% $n\ge 3$ but breaks down in the case $n=2$.
\item \label{it:tree}
  Generalising the previous example, let $\G$ be a tree, let $t$ be a leaf of $\G$  and let $a$ be the vertex of $\G$ to
  which $t$ is adjacent; so $\lk(t)=\{a\}$. Let $w$ be any word in $\la A\bs \st(a)\ra$. Then $s=wt$ satisfies all the hypotheses
  of Theorem \ref{thm:main} so $\la A\bs\{t\}\ra$ embeds in $\GG/N(s^n)$, when $n\ge 3$. Moreover, if $w_i\in \la A\bs \st(a)\ra\cup \la a\ra$  and
  $\e_i\in \{\pm 1\}$ are chosen for $i=0,\ldots, m$, 
such that $w_i\notin \la a\ra$, for at least one $i$, and  $s=w_0t^{\e_1}w_1\cdots w_{m-1}t^{e_m}w_m$ is cyclically minimal and not a proper power in $\GG$, then again
Theorem \ref{thm:main} implies that  $\la A\bs\{t\}\ra$ embeds in $\GG/N(s^n)$, for $n\ge 3$.
\item \label{it:fourchord} 
Consider the graph $C_4'$ on the left of Figure \ref{fig:fourchord}. This is  a starred graph,  therefore  Theorem A of \cite{AK} applies. 
Every  subgraph of a connected starred graph contains a \emph{central} vertex; that is a vertex incident to all other vertices. 
In this example the set of central vertices is $B=\{a,c\}$. From \cite[Theorem A]{AK} it follows that if $r\notin \la a,b,c\ra$ then 
$\la a,b,c\ra$ embeds in $G=\GG/N(r)$. Moreover (\emph{loc. cit.}), if $r\notin \la a,c,d\ra$  then 
$\la a,c,d\ra$ embeds in $G$ and if $r\notin \la a,c\ra$  then 
$\la a,c\ra$ embeds in $G$. On the other hand, $\{a,c\}$ is a synchronised clique, so 
from our Corollary \ref{cor:Abfrei} it follows that if $r\in \la a,c\ra$ and $t\in \supp(r)$ then $\la A\bs \{t\}\ra$ embeds
in $G$. Again, $\{b,d\}$ is a synchronised and independent, so 
from  Corollary \ref{cor:Freefrei}, if $r\in \la b,d\ra$ and $t\in \supp(r)$ then $\la A\bs \{t\}\ra$ embeds
in $G$. The same applies when $\G=C_4$, as shown on the right hand side of Figure \ref{fig:fourchord}. In this case,
if $\supp(r)=\{b,d\}$ and $t=b$ or $d$,  then Corollary \ref{cor:Freefrei} implies that
$\la A\bs\{t\}\ra$ embeds in $\GG(C_4)/N(r)$.

As the set of central vertices of a starred graph always forms a synchronised clique, for starred graphs  
Corollary \ref{cor:Abfrei} always complements  \cite[Theorem A]{AK}, as in this example.
\begin{figure}
\begin{center}
\begin{tikzpicture}[scale=.75,arrowmark/.style 2 args={decoration={markings,mark=at position #1 with \arrow{#2}}}]%
  \tikzstyle{every node}=[circle, draw, fill=black, color=black,
  inner sep=0pt, minimum width=6pt]
\pgfmathsetmacro{\l}{1}
\draw (0,0) node[color=black] {};
\draw (3*\l,0) node {};
\draw (0,2*\l) node {};
\draw (3*\l,2*\l) node {};
%\draw (\l*1.5,\l*4) node {};
%
\begin{pgfonlayer}{background}
  \draw (0,0)  -- (\l*3,0);
  \draw (0,0)  -- (0,\l*2);
  \draw (0,\l*2) --(\l*3,\l*2);
  \draw (3*\l,0) --(\l*3,\l*2);
  \draw (0,0) -- (\l*3,\l*2);
  %\draw (0,\l*2) -- (\l*1.5,\l*4);
  %\draw (\l*3,\l*2) -- (\l*1.5,\l*4);
\end{pgfonlayer}
\draw (\l*3,2*\l) +(0.4,0) node[draw=none,fill=none,color=black] {$a$};
\draw (\l*3,0) +(0.4,0) node[draw=none,fill=none,color=black] {$b$};
\draw (0,0)  +(-0.4,0) node[draw=none,fill=none,color=black] {$c$};
\draw (0,2*\l) +(-0.4,0) node[draw=none,fill=none,color=black] {$d$};
%\draw (\l*1.5,4*\l) +(0,0.5)  node[draw=none,fill=none,color=black] {$t$}; 
\draw (\l*1.5,0) +(0,-1)  node[draw=none,fill=none,color=black] {$C_4'$}; 
%%%%%%%%%%%%%
  %%%%%%%%%%%%%
\begin{scope}[shift={(7*\l,0)}]
\draw (0,0) node[color=black] {};
\draw (3*\l,0) node {};
\draw (0,2*\l) node {};
\draw (3*\l,2*\l) node {};
%\draw (\l*1.5,\l*4) node {};
%
\begin{pgfonlayer}{background}
  \draw (0,0)  -- (\l*3,0);
  \draw (0,0)  -- (0,\l*2);
  \draw (0,\l*2) --(\l*3,\l*2);
  \draw (3*\l,0) --(\l*3,\l*2);
 
 % \draw (\l*3,\l*2) -- (\l*1.5,\l*4);
\end{pgfonlayer}
\draw (\l*3,2*\l) +(0.4,0) node[draw=none,fill=none,color=black] {$a$};
\draw (\l*3,0) +(0.4,0) node[draw=none,fill=none,color=black] {$b$};
\draw (0,0)  +(-0.4,0) node[draw=none,fill=none,color=black] {$c$};
\draw (0,2*\l) +(-0.4,0) node[draw=none,fill=none,color=black] {$d$};
%\draw (\l*1.5,4*\l) +(0,0.5)  node[draw=none,fill=none,color=black] {$t$};
\draw (\l*1.5,0) +(0,-1)  node[draw=none,fill=none,color=black] {$C_4$}; 
\end{scope}
\end{tikzpicture}
\end{center}
    \caption{Example \ref{ex:thm_main}.\ref{it:fourchord}}\label{fig:fourchord}
\end{figure}
%%%%%%%%%%%%%%%%%%%%%%%%%%%%%%%%%%%%%%%%%%%%
%%%%%%%%%%%%%%%%%%%%%%%%%%%%%%%%%%%%%%%%%%%%
%%%%%%%%%%%%%%%%%%%%%%%%%%%%%%%%%%%%%%%%%%%%
%%%%%%%%%%%%%%%%%%%%%%%%%%%%%%%%%%%%%%%%%%%%
%%%%%%%%%%%%%%%%%%%%%%%%%%%%%%%%%%%%%%%%%%%%
%%%%%%%%%%%%%%%%%%%%%%%%%%%%%%%%%%%%%%%%%%%%
%%%%%%%%%%%%%%%%%%%%%%%%%%%%%%%%%%%%%%%%%%%%
%%%%%%%%%%%%%%%%%%%%%%%%%%%%%%%%%%%%%%%%%%%%

\item \label{it:pentchord}
  The previous examples of a four cycle and a four cycle with a chord may be extended to $n$-cycles where $n\ge 5$.
Let $C_n'$ be the $n$-cycle with a chord, $n\ge 5$, as shown  on the left of Figure \ref{fig:pentchord}, 
let $\GG'=\GG(C_n')$ and let $s'=a_2a_{n-2}t\in \GG'$. Then $\GG'$, $t$ and $s'$ satisfy the conditions of
Theorem \ref{thm:main},  so 
setting $r'=s'^3$ the  subgroup $\la a_1, a_2,\ldots, a_{n-2},a_{n-1}\ra$ of $\GG'$ embeds in  
$G'=\GG'/N(r')$. 

Next let $C_n$ be $n$-cycle on the right of Figure \ref{fig:pentchord}, 
let $\GG=\GG(C_n)$, let $s=a_2a_{n-2}t\in \GG$ and $r=s^3$. In this case $\lk(t)$ is not a clique so Theorem \ref{thm:main} does 
not apply. There is a natural surjection $\pi'$ from $\GG$ to $\GG'$ and composing this with the canonical map $\rho'$ from $\GG'$ to $G'$ 
gives a surjection $\rho'\pi'$ of $\GG$ onto $G'$. Since $\pi'$ maps the 
subgroup $H=\la a_2,\ldots ,a_{n-2} \ra$ of $\GG$ isomorphically to the 
subgroup $H'=\la a_2,\ldots ,a_{n-2} \ra$ of $\GG'$ and, from the above, $\rho'$ embeds $H'$  into $G'$, it follows that $\rho'\pi'$ restricts 
to an embedding 
of $H$ into $G'$. On the other hand, denoting the canonical map from $\GG$ to $G=\GG/N(r)$ by $\rho$, and the canonical map 
of $G$ to $G'$ by $\pi$ we have $\pi\rho=\rho'\pi'$. Therefore $\rho$ restricts to an  embedding of $H$ into $G$. 
This gives a restricted Freiheitssatz in the case where $\lk(t)$ is not a clique.

\ee
\end{example}
%%%%%%%%%%%%%%%%%%%%%%%%%%%%%%%%%%%%%%%%%%%%
%%%%%%%%%%%%%%%%%%%%%%%%%%%%%%%%%%%%%%%%%%%%
%%%%%%%%%%%%%%%%%%%%%%%%%%%%%%%%%%%%%%%%%%%%
%%%%%%%%%%%%%%%%%%%%%%%%%%%%%%%%%%%%%%%%%%%%
%%%%%%%%%%%%%%%%%%%%%%%%%%%%%%%%%%%%%%%%%%%%
%%%%%%%%%%%%%%%%%%%%%%%%%%%%%%%%%%%%%%%%%%%%
%%%%%%%%%%%%%%%%%%%%%%%%%%%%%%%%%%%%%%%%%%%%
%%%%%%%%%%%%%%%%%%%%%%%%%%%%%%%%%%%%%%%%%%%%
\begin{figure}
\begin{center}
\begin{tikzpicture}[scale=.75,arrowmark/.style 2 args={decoration={markings,mark=at position #1 with \arrow{#2}}}]%
  \tikzstyle{every node}=[circle, draw, fill=black, color=black,
  inner sep=0pt, minimum width=6pt]
  \pgfmathsetmacro{\l}{3}
\foreach \a in {0,...,4}
    {
        \node at ({-30+\a*60}:\l){};
      }
\begin{pgfonlayer}{background}
 %\pgfmathsetmacro{\l}{3}
  \draw [black,thick,domain=-40:220] plot ({\l*cos(\x)}, {\l*sin(\x)});
  \draw [black,thick,dashed,domain=220:320] plot ({\l*cos(\x)}, {\l*sin(\x)});
  \draw [black,thick] (30:\l) -- (150:\l);
\end{pgfonlayer}
\node[draw=none,fill=none,color=black] at (90:\l+0.5){$t$};
\node[draw=none,fill=none,color=black] at (270:\l+1){$C'_n$};
\foreach \a in {0,...,1}
  {
    \pgfmathtruncatemacro{\s}{2-\a}
    \node[draw=none,fill=none,color=black] at ({-30+\a*60}:\l+0.5){$a_{\s}$};
  }
\foreach \a in {3,...,4}
  {
    \pgfmathtruncatemacro{\s}{\a-2}
    \node[draw=none,fill=none,color=black] at ({-30+\a*60}:\l+0.7){$a_{n-\s}$};
  }  
%%%%%%%%%%%%%
  %%%%%%%%%%%%%
  \begin{scope}[shift={(3.5*\l,0)}]
 \foreach \a in {0,...,4}
    {
        \node at ({-30+\a*60}:\l){};
      }
\begin{pgfonlayer}{background}
 %\pgfmathsetmacro{\l}{3}
  \draw [black,thick,domain=-40:220] plot ({\l*cos(\x)}, {\l*sin(\x)});
  \draw [black,thick,dashed,domain=220:320] plot ({\l*cos(\x)}, {\l*sin(\x)});
\end{pgfonlayer}
\node[draw=none,fill=none,color=black] at (90:\l+0.5){$t$};
\node[draw=none,fill=none,color=black] at (270:\l+1){$C_n$};
\foreach \a in {0,...,1}
  {
    \pgfmathtruncatemacro{\s}{2-\a}
    \node[draw=none,fill=none,color=black] at ({-30+\a*60}:\l+0.5){$a_{\s}$};
  }
\foreach \a in {3,...,4}
  {
    \pgfmathtruncatemacro{\s}{\a-2}
    \node[draw=none,fill=none,color=black] at ({-30+\a*60}:\l+0.7){$a_{n-\s}$};
  }  
\end{scope}

\end{tikzpicture}
\end{center}
    \caption{Example \ref{ex:thm_main}.\ref{it:pentchord}}\label{fig:pentchord}
  \end{figure}
%\begin{remark}
%\be
%\item
%In cases where the hypotheses of both Theorems \ref{thm:main} and Corollary \ref{cor:Freefrei} hold the latter 
%always gives stronger results. 
%\item 
%\ee    
%\end{remark}

To put  the conditions and conclusions of Theorem \ref{thm:main} in context we consider circumstances under which the Theorem 
fails. As in Example \ref{ex:thm_main}.\ref{it:pentchord},  we may salvage some form of Freiheitssatz in the case where 
$\lk(t)$ is not a clique; so we assume that, for some vertex $t$ of $\supp(s)$, the link of $t$ is a clique. To 
keep matters simple we  also assume that $s=wt$, where $\supp(w)=A\bs\{t\}$. In this case the conditions of 
Theorem \ref{thm:main} fail if either 
\be[label=(\roman*)]
\item 
$\supp(s)\subset \st(t)$, if and only if $\supp(w)\subset \lk(t)$; or 
\item 
$s$ is not a $t$-root; or 
\item $s$ is not $t$-thick.  
\ee
Since we have assumed that $\supp(s)=A$ the first of these possibilities, together with the assumption that 
$\lk(t)$ is a clique, implies that $\GG$ is a free Abelian group. In this case we defer to standard results for finitely generated
free Abelian groups; and no longer need our Theorems. Moreover, the form $s=wt$ that we have chosen implies that 
$s$ is not a $t$-root. (Indeed, from the definitions in Section \ref{sec:hnndiag}, this holds for every $s$ with prime $t$-length.) This leaves the case where $s$ is not $t$-thick. When $s=wt$ is not $t$-thick it  follows from Lemma \ref{lem:thick} that 
there is a element $b$ which belongs to both $Y=\{y\in A\,:\, \st(y)=A\}$ (the  set of central vertices of $\G$)  
and  to $\supp(w)\bs\lk(t)$. The centre $Z$ of $\GG$ is equal to $\la Y\ra$ 
and so, setting $A_0=A\bs Y$ and $\GG_0=\la A_0\ra$, we have $\GG=\GG_0\times Z$ and $w=w_0w_z$, where $w_0\in \GG_0$, $w_z\in Z$. 
As $s\notin \st(t)$, we have $t\in A_0$ and, setting $U_0=\la \lk(t)\bs Y\ra$, Lemma \ref{lem:thick} implies that $w_0$ is 
in $\maln_{\GG_0}(U_0)$ so $s_0=w_0t$ is $t$-thick. Moreover, 
from the definition in Section \ref{sec:redn}, it follows that $s_0$ is not a $t$-root and certainly $s_0$ is not in the star 
of $t$ (in the full subgraph of $\G$ generated by $A_0$). The hypotheses of Theorem \ref{thm:main} therefore hold for 
the element $s_0$ of $\GG_0$ and the letter $t$ of $\supp(s_0)$. Thus $\la A_0\bs \{t\}\ra$ embeds in $G_0=\GG_0/N$, where 
$N$ is the normal closure of $s_0^n$ in $\GG_0$, for some $n\ge 3$. In $\GG$ we have $r=s^n=w_z^n(w_0t)^n=w_z^ns_0^n$ and so
 $G=\GG/N(r)=G_0\times Z/W$, where $W=\la w_z\ra$. It follows that the subgroup $\la A_0\bs \{t\}\ra$ of $\GG$ embeds in 
$G$. Also, using Theorem \ref{thm:main} \ref{it:main2}, $sN(r)$ has order $n$ in $G$. 

From the results of \cite{AK}, the examples above and this analysis we conclude that a version of the Freiheitssatz holds 
%\\ \ajd{This statement needs tightening up}\\
for many one-relator quotients of partially commutative groups. A precise classification requires a more delicate analysis
of the ways in which $s$ may  fail to be $t$-thick or a $t$-root (or a proof which avoids some of
these conditions). To add some credence to our claim however, we show in Section
\ref{sec:almostall} that %for sufficiently large $n$,
 Theorem \ref{thm:main} applies for almost all cyclically reduced words 
 in $\GG(C_n')$ (as defined in Example  \ref{ex:thm_main}.\ref{it:pentchord}).
% \ajd{Failure of being a $t$-root would be dealt with by proving the result for $s_1\cdots s_n$ instead of $s^n$. I have not managed to make Lemma \ref{lem:perpospiece} work in this generality so far.}

 In Section \ref{sec:prelim} we review the parts of the 
theory of partially commutative groups necessary for the paper.  In Section \ref{sec:hnndiag}, we describe   
small cancellation theory over the HNN-presentation of $\GG$, in situations where the there is an element $t$ in $\supp(s)$ such that
$\lk(t)$ is a clique. 
 In  Section \ref{sec:main} we prove Theorem \ref{thm:main} and Theorem \ref{thm:amalgam}. Finally, in Section
\ref{sec:almostall} we show that Theorem \ref{thm:main} is almost always applicable in the situation of
Example \ref{ex:thm_main}.\ref{it:pentchord}. 
%In Section \ref{sec:simplex} we then 
%prove Theorems \ref{thm:Abfrei} and \ref{thm:Freefrei}, where $\supp(s)$ is a simplex or co-simplex.

%\ajd{Applications: WP, torsion elements, Lyndon's identity theorem ....}
\section{Preliminaries}\label{sec:prelim}
First we recall some of the notation and definitions of \cite{B} and  \cite{EKR}.
As above, if $w\in (A\cup A^{-1})^\ast$ then $\supp(w)$ is the set of
elements $x\in A$ such that $x$ or $x^{-1}$ occurs in $w$. (In \cite{EKR}
$\a(x)$ is used instead of $\supp(x)$.) 
 For $u,v \in (A\cup A^{-1})^*$ we
use $u=v$ to mean $u$ and $v$ are equal as elements of $\GG$. Equality
of words $u,v$ in $(A\cup A^{-1})^*$ is denoted by $u\equiv v$.
If $g\in \GG$ and $w\in (A\cup A^{-1})^*$ is a
word of minimal length representing $g$ then we say that $w$ is a {\em
  minimal} form of $g$ (or just that $w$ is {\em minimal}).

The Cancellation Lemma, \cite[Lemma 4]{B}, asserts that if $w$ is a non-minimal word in $(A\cup A^{-1})^*$ then
$w$ has a subword $xux^{-1}$, where $x\in A\cup A^{-1}$, $u\in (A\cup A^{-1})^*$ and $x$ commutes with every
letter occurring in $u$. 
The Transformation Lemma, \cite[Lemma 5.5.1]{B} (see also \cite[Lemma 2.3]{EKR}) asserts that, if $u$
and $v$ are minimal and $u=v$ then we can transform the word $u$ into
the word $v$ using only commutation relations from $R$ (that is,
without insertion or deletion of any subwords of the form
$x^{\e}x^{-\e}, x\in A$).
%\brief{
From the
Transformation Lemma it follows that if $u$
and $v$ are minimal and $u=v$ then $\supp(u)=\supp(v)$ and $|u|=|v|$
(where $|w|$ denotes the length of the word $w$). Therefore, for
any element $g\in \GG$ we may define $\supp(g)=\supp(w)$, and the {\em length} $l(g)$ of $g$ as
$l(g)=|w|$, where $w$ is a
minimal form of $g$. For a subset $S$ of $\GG$ we define  $\supp(S)=\cup_{s\in S}\supp(s)$. 

If $g,h\in \GG$ such that $l(gh)=l(g)+l(h)$ then we write $gh=g\cdot h$. 
(Written $g\circ h$ in \cite{EKR}.)
It follows that $gh=g\cdot h$ if and only if, for all minimal forms
$u$ and $v$ of $g$ and $h$, respectively, $uv$ is a minimal form for
$gh$. Clearly if $w\in (A\cup A^{-1})^*$ is minimal and $w\equiv uv$
then, in $\GG$, $w=u\cdot v$. If $k=g\cdot h$ then we say that $g$ is a
{\em left divisor} of $k$ (and $h$ is a {\em right divisor} of $k$).

We say that $h\in \GG$ is {\em cyclically minimal} 
if $l(h)\le l(h^g)$, for all $g\in \GG$.
\begin{comment}
  It follows from 
Lemma \ref{lem:cycgeod} that if $g$ is a cyclically minimal element of $\GG$ then all
minimal forms of $g$ are cyclically minimal forms. %elements of $(A\cup A^{-1})^*$. 

\end{comment}
% and
%the element
%$vu$ of $\GG$ a {\em cyclic} $\GG${\em -permutation} of $w$.
If $w\in (A\cup A^{-1})^*$ and $w\equiv uv$ then we
call  $vu\in (A\cup A^{-1})^*$ a {\em cyclic permutation} of $w$. 
\begin{lemma}[{\cite[Lemma 2.2]{DKR2}}]\label{lem:cycgeod}
Let $w\in (A\cup A^{-1})^*$ be a minimal form for an element $h$ of $\G$. Then the following are equivalent.
\be[(i)]
\item\label{cycgeod1} $h$ is cyclically minimal.
\item\label{cycgeod2} If $y\in A\cup A^{-1}$ is a left divisor of $h$ then $y^{-1}$ is not a right
  divisor of $h$.
\item\label{cycgeod3} All cyclic permutations of $w$ are minimal forms.
%\item\label{cycgeod4} $w$ is of minimal length in its conjugacy class.
\ee
Moreover, if $w$ is cyclically minimal then $\supp(w^g)\supseteq \supp(w)$,
for all $g\in \GG$.
\end{lemma}
If $w$ is a minimal form of a cyclically minimal element $g\in \GG$ % which is cyclically minimal
then we say that $w$ is a cyclically minimal form.
From, for example, Proposition 3.9 of \cite{DK}, if $g\in \GG$ then there exist $u,
w\in (A\cup A^{-1})^*$, with $w$ a cyclically minimal form, such that
$g=u^{-1}\cdot w \cdot u$. Thus $g$ is cyclically minimal if and only if
$u=1$. Observe that if $g$ is cyclically minimal then
$l(g^n)=nl(g)$. Therefore partially commutative groups are torsion free.

\begin{lemma}\label{lem:doublecoset}
Let $U$ be a subgroup of $\GG$ and let $D$ be a set of double coset representatives of $U$ in $\GG$. The following are 
equivalent.
\be[(i)]
\item\label{it:dc3} If $d\in D$ and $g\in UdU$ then $l(g)\ge l(d)$.
\item\label{it:dc4} If $d\in D$ then $d$ has no left or right divisor in $U$.
\ee 
There exists a unique set $D$ of double coset representatives of $U$ in $\GG$ such that  \ref{it:dc3} holds and 
%\ref{it:dc4} and 
this
set satisfies 
\be[label=(\roman*)]
\addtocounter{enumi}{2}
\item\label{it:dc1} $1\in D$ and
\item\label{it:dc2} if $d\in D$ then $d^{-1} \in D$. 
\ee 
In particular, if $UdU=Ud^{-1}U$, for some $d\in D$, then $d=1$. 
\end{lemma}
\begin{proof}
For each $g\in \GG$ we may choose a minimal form $d\in \GG$ such that $l(d)\le l(w)$, 
for all $w\in UgU$. Thus there exists a set $D$ of double coset representatives
having property \ref{it:dc3}. Suppose $D$ is such a set and $d\in D$. Then
$d$ satisfies \ref{it:dc4}, since it has minimal length amongst elements of $UdU$. 
 If $d\in D$, 
$g\in UdU$ and $l(g)=l(d)$ 
then $g=udv$, for some $u,v\in U$. In fact, as $d$ has no left or right divisor 
in $U$ we may choose such
$u$ and $v$ so that $g=u\cdot d \cdot v$.  Unless $u=v=1$ this implies that $l(g)>l(d)$, and  
if $u=v=1$ then we have $g=d$. Therefore $d$ is the unique element of minimal length in $UdU$; 
and $D$ is uniquely determined by condition \ref{it:dc3}. 
 On the other hand, if $D'$ is a set of double coset representatives satisfying \ref{it:dc4} a similar argument
shows that $D'$ satisfies \ref{it:dc3}.

Assume now that $D$ satisfies \ref{it:dc3}. Then \ref{it:dc1}, a special case of \ref{it:dc3}, also holds. If $d\in D$ and 
 $Ud^{-1}U$ contains an element $g$ such that $l(g)\le l(d^{-1})$ then $g^{-1}\in UdU$ and $l(g^{-1})\le l(d)$; so
from the above $g=d^{-1}$. It follows that if $d\in D$ then so is $d^{-1}$, and so \ref{it:dc2} holds. 
Moreover, since $d\neq d^{-1}$, we have $UdU\cap Ud^{-1}U=\emptyset$, for all $d\in D\bs \{1\}$.
\begin{comment} 
$UdU\cap Ud^{-1}U\neq \emptyset$ then there exist $u,v\in U$ such that
$d^{-1}=udv$. In fact, as $d$ has no left or right divisor in $U$ we may choose such
$u$ and $v$ so that $d^{-1}=u\cdot d \cdot v$. Unless $u=v=1$ this is a contradiction, and 
if $u=v=1$ then we have $d^{-1}=d$, so $d=1$ (as $\GG$ is torsion free). It follows that 
$UgU\cap Ug^{-1}U=\emptyset$, for all $g\in \GG\bs U$.  Moreover, if $d\in D$ and $d^{-1}\notin D$ then
we may replace $d'\in D$ such that $Ud^{-1}U=Ud'U$ by $d^{-1}$, since property \ref{it:dc3} holds
for $d$ if and only if it holds for $d^{-1}$.
\end{comment}
\end{proof}

The elements of $A$ are termed the {\em canonical} generators of $\GG(\G)$ and 
a subgroup of $\GG$ generated by a subset $Y$ of $\GG$ is called a {\em
canonical parabolic} subgroup. 
Let $Y\subseteq A$ and denote by $\G_Y$ the full subgraph of $\G$ with vertex set $Y$. Let
$\GG(\G_Y)$ be the
partially commutative group with commutation graph $\G_Y$. 
%given by the presentation with
%generators $Y$ and relators those $[x_i,x_j]\in R$ such that
%$x_i,x_j\in Y$. Thus $\la Y\ra$ has commutation graph the full subgraph
%of $\G(\GG)$ generated by $Y$.
It follows from the Transformation
Lemma, \cite[Lemma 5.5.1]{B}, that $\GG(\G_Y)=\la Y\ra$, the canonical
parabolic subgroup of $\GG$ generated by $Y$.
%If $g\in \GG$ and $H$ is
%a canonical parabolic subgroup of $\GG$ then we say that $H^g$ is a
%{\em parabolic subgroup}. Thus parabolic subgroups are those with
%generating set $\{y^g:y\in Y\}$, for some subset $Y$ of $A$ and some fixed 
%$g\in \GG$.

For cyclically minimal elements $g,h\in \GG$ we define $g\sim_0 h$ if $g=u\cdot v$ and 
$h=v\cdot u$, for some $u,v\in \GG$. Then let $\sim$ be the transitive closure of
 $\sim_0$ and denote by  $[g]$ the equivalence class of $g$ under  the equivalence relation $\sim$ on
cyclically minimal elements.
%}%endbrief
The following appears in \cite[Corollary 2.4]{DKR2} (where the set $[g]$ is incorrectly defined).
\begin{corollary}\label{cor:conj}
Let $w,g$ be (minimal forms of) elements of $\GG$ and 
let $w=u^{-1}\cdot v\cdot u$, where $v$ is cyclically minimal. Then
there exist minimal forms $a$, $b$, $c$, $d_1$, $d_2$ and 
$e$ such that $g=a\cdot b\cdot c\cdot d_2$, $u=d_1\cdot a^{-1}$, $d=d_1\cdot d_2$,
$w^g=d^{-1}\cdot e\cdot d$, $[e]=[v]$, $e=v^b$, $\supp(b)\subseteq \supp(v)$ and 
$[\supp(b\cdot c),\supp(d_1)]=[\supp(c),\supp(v)]=1$.
\end{corollary}

The next result is a direct consequence of  this lemma.
\begin{corollary}\label{cor:conj1}
Let $Y$ be a subset of $A$ and let $w,g$ be (minimal forms of) elements of $\GG$
such that $w$ and $w^g$ belong to $\la Y\ra$ and $g$ has no right or left 
divisor in $\la Y\ra$. Then $[\supp(g),\supp(w)]=1$. (That is $[x,y]=1$, for all $x\in \supp(g)$ and  $y\in \supp(w)$.)
\end{corollary}
%\brief{
\begin{proof}
In the notation of Corollary \ref{cor:conj} we have
$u,v, d \in \la Y \ra$ since $w$ and $w^g$ are in $\la Y\ra$. Hence
$a, d_1, d_2$ and $b$ are in $\la Y\ra$. 
As $g$ has no left or right divisors in $\la Y \ra$ it follows that
$g=c$ and $u=d_1$. To complete the proof we use 
the fact that $[\supp(b\cdot c),\supp(d_1)]=[\supp(c),\supp(v)]=1$. 
\end{proof}
%}%endbrief

\begin{corollary}\label{cor:conjfix}
Let $Y\subseteq A$ be a clique and let $w,g$ be (minimal forms of) elements of $\GG$
such that $w$ and $w^g$ belong to $\la Y \ra$. Then 
$[\supp(g),\supp(w)]=1$.
\end{corollary}
%\brief{
\begin{proof}
If $g\in \la Y\ra$, the result holds as $Y$ is a clique. Otherwise
we may write 
$g=a\cdot b\cdot c$, where $a,c\in \la Y\ra$ and $b$ has no left or 
right divisor in $\la Y \ra$. Then $w^g=w^{b\cdot c}\in \la Y\ra$. Again,
 as $Y$ is a clique and $c\in \la Y \ra$, we have $w^{b\cdot c}=w^b$.  
From  Corollary \ref{cor:conj1}, $[\supp(b),\supp(w)]=1$, and as $Y$ is  a clique, $[\supp(g),\supp(w)]=1$, as claimed.
\end{proof}
% }%endbrief
%\ajd{Here is an example of what might happen if $Y$ is not a clique in the above situation. Let $\G$ be the graph
%  with vertices $a,b,c,d,t$ and edges $\{d,t\}, \{t,a\}$, $\{a,b\}, \{b,c\}$ and $\{c,a\}$. Let $Y=\{a,d\}=\lk(t)$.
%  Now let $w=a$ and $g=cd$, so $w^g=d^{-1}ad$. Then $w,w^g\in \la Y\ra$ but $[g,a]\neq 1$. The $g$ here has no left divisor in $\la Y\ra$.}
\begin{comment}
Let $t\in A$. We define $N_0(t)=\{t\}$, 
$N_1(t)=\{y\in A|[t,y]=1\}$ and for a subset $T$ of 
$A$ define $N_1(T)=\cup_{t\in T}N_1(t)$. Now for $i\ge 2$ define
$[w,g]=1$.
\end{corollary}
%\brief{{
\begin{proof}
If $g\in \la Y\ra$, the result holds as $Y$ is a clique. Otherwise
we may write 
$g=a\cdot b\cdot c$, where $a,c\in \la Y\ra$ and $b$ has no left or 
right divisor in $\la Y\ra$. Then $w^g=w^{b\cdot c}\in \la Y\ra$. Again,
 as $Y$ is a clique and $c\in \la Y\ra$, we have $w^{b\cdot c}=w^b$.  
From  Corollary \ref{cor:conj1}, $[w,b]=1$, so $[w,g]=1$, as claimed.
\end{proof}
%}\endbrief
\begin{comment}
Let $t\in A$. We define $N_0(t)=\{t\}$, 
$N_1(t)=\{y\in A|[t,y]=1\}$ and for a subset $T$ of 
$A$ define $N_1(T)=\cup_{t\in T}N_1(t)$. Now for $i\ge 2$ define
$N_i(T)=N_1(N_{i-1}(T))$. Now for all $i\ge 1$ define $L_i(T)\backslash N_{i-1}(T)$. 
Then $L_1(T)=N_1(T)\backslash \{t\}$ and in general $L_i(t)$ is the set of 
vertices distance $i$ from $t$.
\end{comment}

For a subset $U$ of $A$ denote
by $\G_U$ the full subgraph of $\G$ generated by $U$ and for a cyclically
minimal word  $w$  over $\GG$ set $\G_w=\G_{\supp(w)}$. 
For a subset $Y\subset A$ define $\lk(Y)=\cap_{y\in Y}\lk(y)$    
%\begin{comment}
and $\st(Y)=\cap_{y\in Y}\st(y)$.
For an element $w\in \GG$ define $\lk(w)=\lk(\supp(w))$ and $\st(w)=\st(\supp(w))$.  
\begin{comment}
Let $T=\{s_1,\ldots,s_n\}$ be 
a set of cyclically minimal elements of $\GG$. Define $\lk(T)=\cap_{i=1}^n\lk(s)$ and $\supp(T)=
\cup_{i=1,\ldots ,n}\supp(s_i)$. 
\end{comment}
From \cite[Korollar 3]{B}, for $a\in A$ we have $C_\GG(a)=\la a\ra\times \la \lk(a)\ra$.

As necessary we shall use a normal form for %the \emph{block decomposition} of
 elements of the partially commutative group $\GG(\G)$,
which we now  define. Let $\D$ be the complement of $\G$ ($\D$ has the same vertex set as $\G$ and $\{u,v\}$ is an edge
of $\D$ if and only if $\{u,v\}$ is not an edge of  $\G$). For any element $w$ of  $\GG$ define $\D(w)$ to
be the full subgraph of $\D$ with vertices $\supp(w)$. If $\D(w)$ is connected we call $w$ a \emph{block}. If
$\D(w)$ has connected components  $\D_1,\ldots , \D_k$ then it follows that $w=w_1\cdot \,\cdots\, \cdot w_k$, where
$\D(w_i)=\D_i$ and $[\supp(w_i),\supp(w_j)]=1$, for $i\neq j$. We call this factorisation of $w$ the
\emph{block decomposition} of $w$. 

In the sequel, we shall use the notation and results of this section without further mention. 
\section{Partially commutative groups as HNN-extensions}\label{sec:hnndiag}

In this section we review the special case of diagrams 
over HNN-extensions we need in the proof of Theorem \ref{thm:main}. We refer the reader to 
\cite[Pages 291--292]{LS} for basic results on diagrams over HNN-extensions and to \cite{J} for a more general
version of the definitions given here. As above,  let $\GG$ be a partially commutative group with commutation graph $\G$ and 
canonical presentation $\la A| R\ra$. 
 To realise $\GG$ as an HNN-extension,  
given any $t\in A$, set  $\GG_t=\la A\backslash\{t\}\ra$ and define the  HNN-extension  
\begin{align}
\textrm{HNN}(t)
&=\la \GG_t, t\,|\, t^{-1}xt=x, \forall x\in \lk(t)\ra.\label{eq:hnn}
\end{align}
%&=\la \GG_t, t\,|\, t^{-1}ut=u, \forall u\in \la\lk(t))\ra\notag\\
%as in Lemma \ref{lem:hnnpresn} above,
%where $\GG_t=\laA\backslash\{t\})$.  
%Let $R_t$ be the set $\{[x,y]\in R\,|\, x\neq t\textrm{ and } y\neq t\}$. 
By definition  %group $HNN(t)$ given by the 
HNN$(t)$ is the group with % in the statement of the lemma is, by definition, the group with free 
presentation $\la A\,|\, R_t\cup\{t^{-1}ut=u, u\in \la\lk(t)\ra\}\ra$, where $R_t=\{[x,y]\in R\,|\, x\neq t\textrm{ and } y\neq t\}$.
We may perform Tietze transformations on the latter to replace it with the presentation
$\la A \,|\, R_t\cup\{t^{-1}xt=x, x\in \lk(t)\}\ra$. As $R_t\cup\{t^{-1}xt=x, x\in \lk(t)\}=R$, it follows that HNN$(t)=\GG$. 
We call HNN$(t)$ the \emph{HNN-presentation of} 
$\GG$ with respect to $t$.

For notational simplicity, write 
\biz
\item $F$ for the group $\GG=\hnn(t)$ with the HNN-presentation with respect to $t$, 
\item $H=\GG_t$ and   $U=\la \lk(t)\ra$, and in addition
\item let $D$ be a set of double coset representatives for $U$ in $H$, satisfying the properties of Corollary \ref{lem:doublecoset}, 
and let $\s$ be the function from $H$ to $D$ given by $\s(g)=d$, where $UgU=UdU$.
\eiz

(We assume that all elements of $F$ are represented
as reduced words of the free product  $H*\la t\ra$, unless an explicit exception is made.) 
Given $p\in F$, the factorisation 
\begin{equation}\label{eq:redfac}
p=g_0t^{\e_1}\cdots t^{\e_n}g_n,
\end{equation} 
where $g_i\in H$ and $\e_i\in \{\pm 1\}$, 
is called  \emph{reduced}, over $F$, if either $n\le 1$  or 
$n>1$ and, for $1\le i\le n-1$,  if $g_i\in U$ then $\e_i\e_{i+1}=1$. In this case we say that $p$ has  $t${\em -length}
\[|p|_t=\sum_{i=1}^n |\e_i|.\] Every element
of $F$ has a reduced factorisation, and two reduced factorisations which represent the same element have 
the same $t$-length. We explicitly allow reduced words to contain sub-words of the form $t^\e u t^\e$, where 
$u\in U$ and $\e =\pm 1$.

%\brief{{
If $u$ and $v$ are elements of $F$ such that $|uv|_t=|u|_t+|v|_t$, then we say the  product $uv$ is \emph{reduced} 
(if and only if, given reduced factorisations \[
u=g_0t^{\e_1}\cdots t^{\e_m}g_m\textrm{ and }v=h_0t^{\d_1}\cdots t^{\d_n}h_n,\]  
either $m+n\le 1$ or $g_mh_0\notin U$ or  $\e_m=\d_1$).   
More generally, if $u_1,\ldots, u_n$ are reduced factorisations of elements of $F$ such that 
\begin{equation}\label{eq:prodred}
|u_1\cdots u_n|_t=\sum_{i=1}^n |u_i|_t
\end{equation} 
we say that
the product $u_1\cdots u_n$ is reduced. 
An 
element  $p$ of $F$  is said to be \emph{cyclically reduced}  if the product  
$pp$ is reduced. 
(That is, given a reduced factorisation 
$p= g_0t^{\e_1}\cdots t^{\d_n}g_n$, either  $n\le 1$ or
 $g_ng_0\notin U$ or $\e_n=\e_1$.) Again, for $u_1,\ldots, u_n$ as above, the product $u_1\cdots u_n$ is said to
be \emph{cyclically reduced} if  it is reduced and $u_nu_1$ is reduced. 

The aim of the definitions in the remainder of this sub-section is to allow us to avoid products which split elements of $H$ in a non-trivial fashion. 
For example, if $a, b$ and $g$ are elements of $H\backslash U$ with $g=ab$, and we have  
$s=tgt$ 
then the factorisation $pq$, where $p=ta$ and $q=bt$, splits the factor $g=ab$.  
With this in mind, let the  element $p\in F$ have reduced factorisation \eqref{eq:redfac}. 
If $g_0=1$ then  $p$ is said to \emph{begin} with a $t$-letter and if $g_n=1$ then $p$ is said to \emph{end} with 
a $t$-letter. Now if $u_1,\ldots , u_k$ are reduced factorisations of elements of $F$ then   the 
product $u_1\cdots u_k$ is said to be \emph{right integral at} $i$, where $1\le i\le k$,  if it is  cyclically reduced and 
$u_i$ ends with a $t$-letter and  \emph{left integral at} $i$ if  $u_{i+1}$ begins with a $t$-letter (subscripts modulo $k$).
The product $u_1\cdots u_k$  is \emph{integral at} $i$ if it is either left or right integral at $i$. (Note that this is not quite the same
as the definition in \cite{J}.)
The  product $u_1\cdots u_k$ is said to be \emph{(right) integral}\label{def:integral} if it is (right) integral at $i$, 
for all $i\in \{1,\ldots ,k\}$. We say that $u$ is a \emph{(right) integral} subword of $w$ if $w=u\cdot v$, where
$uv$ is a (right) integral factorisation. 
[Note that the integral property depends on the choice of reduced factorisation of factors. For example, 
$w=tuht$, where $u\in U$, $h\in H\backslash U$, can be written as the product $pq$, 
where $p=tu$ and $q=ht$. This product is not integral, but we may write  $p=ut$ in $F$ and using this factorisation
of $p$ the product $pq$ is integral.]

\subsection{Thick subwords of HNN$(t)$}\label{sec:tthick}

Let $G$ be a group and $K$ a non-trivial subgroup of $G$. Define $\maln_G(K)=\{x\in G\,:\,x^{-1}Kx\cap K=\{1\}\}$.  
It follows that $x\in \maln_G(K)$ if and only if $x^{-1} \in \maln_G(K)$. Also, $x\in \maln_G(K)$ if and only if $KxK\subseteq \maln_G(K)$. Indeed,  
if $x\in \maln_G(K)$ and $u,v\in K$ then $(uxv)^{-1}K(uxv)\cap K= v^{-1}x^{-1}Kxv\cap K= v^{-1}x^{-1}Kxv\cap  v^{-1}Kv=v^{-1}(x^{-1}Kx\cap K)v=\{1\}$. 

\begin{lemma}\label{lem:maln}
Let $V$ be a subgroup of the group $\GG$ and let $D_V$ be a set of representatives of double cosets $VgV$ in $\GG$, satisfying 
the conditions of Corollary \ref{lem:doublecoset}. If $a,b\in \maln_{\GG}(V)$ and $u,v\in V$ are such that 
$aub^{-1}=v$ then  $VaV=VbV$ and there exist $d\in D_V$, $u',v'\in V$ such that $a=(vv')\cdot d\cdot u'$ and $b=v'\cdot d\cdot (u'u)$. 
\end{lemma}
\begin{proof}
Note first that $x\in \maln_{\GG}(V)$ if and only if $VxV\subseteq \maln_{\GG}(V)$. 
As $aub^{-1}\in V$ we have %$au\in Vb$, so 
$VaV=VbV$. Let $d$ the element of $D_V$ such that $VaV=VdV=VbV$. Then there exist 
$a_l, b_r, u',v'\in V$ such that $a=a_l\cdot d\cdot u'$ and $b=v'\cdot d\cdot b_r$. By assumption 
\[a_ldu'ub_r^{-1}d^{-1}v'^{-1}=aub^{-1}=v,\]
and since $a,b\in \maln_{\GG}(V)$ we have $d\in \maln_{\GG}(V)$, so $u'ub_r^{-1}=a_l^{-1}vv'=1$, from which the final statement follows. 
\end{proof}
\begin{defn}\label{def:tthick}
Let $w=g_0t^{\e_1}\cdots g_{k-1}t^{\e_k}g_k$ be a reduced factorisation of an element of $HNN(t)$, where $g_i\in \GG_t$ and $\e_i=\pm 1$,  
and let $U=\la \lk(t)\ra\le \GG_t$.  
We say $w$ is $t$\emph{-thick} 
if $g_i\in U\cup \maln_{\GG_t}(U)$, for $i=0,\ldots ,k$. We say $w$ is \emph{cyclically} $t$\emph{-thick} if $w$ is cyclically reduced, $t$-thick  
and $g_kg_0\in  U\cup \maln_{\GG_t}(U)$.
\end{defn}

The notion of $t$-thickness is well-defined since, if  $w$ has another reduced factorisation $w=h_0t^{\d_1}\cdots h_{m-1}t^{\d_m}h_m$, where $\d_i=\pm 1$, then 
$m=k$, $\e_i=\d_i$ and there are elements $u_0,\ldots ,u_{m-1}\in U$ such that $g_0=h_0u_0^{-1}$, $g_m=u_{m-1}h_m$ and $g_i=u_{i-1}h_iu_i^{-1}$, for 
$1\le i\le m-1$. Hence, from the comment preceding Lemma \ref{lem:maln}, $h_i\in U\cup \maln_{\GG_t}(U)$ if and only if $g_i\in U\cup \maln_{\GG_t}(U)$.  

\begin{lemma}\label{lem:thick}
Let $\GG=\la A\,|\,R\ra$ be a partially commutative group, let $B$ be a subset of $A$ and let $K=\la B\ra$. If  
$B$ is a clique and 
 $w$ is an element of $\GG\bs K$ then the following are equivalent. 
\be[label=(\roman*)]
\item\label{it:thick1} $w$ belongs  to $\maln_{\GG}(K)$. 
\item\label{it:thick2} For all $b\in B$, 
there exists $x\in \supp(w)\bs B$ such that 
$[x,b]\neq 1$.
\ee 
\end{lemma}
\begin{proof}
To see that \ref{it:thick1} implies \ref{it:thick2}, 
 assume $w\in \maln_{\GG}(K)$ and let $b\in B$. Then $w^{-1}bw\neq b$, so $w\notin C_{\GG}(b)$ and it follows that 
there exists $x\in \supp(w)\bs B$ such that 
$[x,b]\neq 1$. 

To prove \ref{it:thick2} implies \ref{it:thick1}, 
consider $w\in \GG\bs K$ such that  $w\notin \maln_{\GG}(K)$. Assume, in order to obtain a contradiction, 
that $w$ is of minimal length amongst all elements not in $K\cup \maln_{\GG}(K)$ which satisfy \ref{it:thick2}. 
%the property that 
%  for all $b\in B$, 
%there exists $x\in \supp(w)\bs B$, with 
%$[x,b]\neq 1$.  
 As $w\notin \maln_{\GG}(K)$, 
 there exists $1\neq v\in K$ such that  
$w^{-1}vw\in K$. We may write $w=w_0\cdot w_1$, where $w_0\in K$ and  $w_1$ is non-trivial and has no (non-trivial) left divisor in
$K$. By minimality of $w$, we conclude that $w_0=1$, and so we may assume $w$ has no left divisor in $K$. 
%We may now write $v=v_0^{-1}\cdot v_1\cdot v_0$, where $v_i\in K$ and $v_1$ is 
%a cyclically minimal element of $\GG$ (and non-trivial).
As $w$ has no left divisor
in $K$, (from the Cancellation Lemma) $w^{-1}vw=w^{-1}\cdot v\cdot w$,  unless $w$ has a left divisor $x\in A^{\pm 1}$ such that 
$[x, v]=1$. In the latter case,  $x$ must commute with all elements of $\supp(v)$.
%, 
%which is non-empty.
 We may write $w=x\cdot w_2$ and, by the assumption on $w$,  
$\supp(w_2)$ must contain an element which does not commute
with an element of $\supp(v)$; so $w_2\notin K$. Moreover $w_2^{-1}vw_2=w^{-1}vw\in K$. It follows that $w_2$ is 
shorter than $w$ and satisfies all the same properties, contrary to the choice of $w$. 
We conclude that $w^{-1}vw=w^{-1}\cdot v\cdot w$, so that $w\in K$, a contradiction. 
 Therefore, for all $w$
 not in $K\cup \maln_{\GG}(K)$  property \ref{it:thick2} fails, as required.
%there exists $b$ in $B$ such that,  $[x,b]=1$, for all $x\in \supp(w)\bs B$. 
\end{proof}

%\begin{proof}[Proof of Corollary \ref{cor:main}]
%This follows directly from Theorem \ref{thm:main} and Lemma \ref{lem:thick}.
%\end{proof}
\subsection{Diagrams over HNN extensions}\label{sec:diag}
In this section we follow  \cite[Chapter V, Section 11]{LS}.
Assume that $\GG$ is expressed as the  HNN-extension $F=\hnn(t)$, 
and set $H=\GG_t$ and $U=\la \lk(t)\ra$, as above. 
 %let $\cR$ be the symmetric closure of $R$ in $\GG$  and let $K=\la \supp(S)\ra$. 
 A set $\cR$ of elements of $F$ is said to be \emph{symmetrised} if every element of $\cR$ is cyclically reduced and, for all $r\in \cR$,
all cyclically reduced conjugates of $r$ and $r^{-1}$ are in $\cR$. 
The \emph{symmetrised closure} $\widetilde \cS$ of  a set $\cS$ of cyclically reduced elements  
of $F$ is the smallest symmetrised subset containing $\cS$; and consists of all cyclically reduced  conjugates of $s$, for 
all elements $s\in \cS\cup \cS^{- 1}$.  From Collins' Lemma (see \cite[Theorem 2.5, p. 185]{LS}) if $r\in F$ is cyclically reduced and ends in
$t^{\pm 1}$ then
a  cyclically reduced element $s\in F$, ending in $t^{\pm 1}$, is a  conjugate of $r$ if and only if it 
may be obtained by taking a cyclic permutation of $r$ and then conjugating by an
element of $U$. This means that even though $\cR$ may be finite its symmetrised closure $\widetilde \cR$ is, in general, infinite.
\begin{defn}\label{defn:piece}
Let $\widetilde \cR$ be a symmetrised subset of $F$. 
An element $p\in F$ is a \emph{piece} (\emph{over} $\widetilde \cR$) if there exist distinct elements $r_1$, $r_2\in \widetilde \cR$, and 
elements $u_1, u_2$ of $F$, such that $r_i$ factors as a reduced product $r_i=_F pu_i$, for $i=1$ and $2$. 
%\\\ajd{Is the next definition needed?}~\\ Given
%reduced products $p_1v_1$ and $p_2v_2$, such that  $p_1v_1\neq p_2v_2$, $p_1=p_2$ in $F$,  and  $p_iv_i\in \widetilde \cR$ we call 
%the pair of subwords $(p_1,p_2)$ a \emph{piece pair}. 
\end{defn}

\begin{defn}\label{defn:sc}
Let $\widetilde \cR$ be a symmetrised subset of $F$ and let $m$ be a positive integer. If $w\in\widetilde \cR$ has a reduced factorisation $w=p_1\cdots p_k$, where $p_i$ is a piece over $\widetilde \cR$, 
then $w$ is said to have a $k$\emph{-piece} factorisation. 
If no element of $\widetilde \cR$ has a $k$-piece factorisation, where $k< m$, then $\widetilde \cR$ is said to
satisfy \emph{small cancellation condition} $C(m)$.   
\end{defn}

For details of disc diagrams over HNN-extensions we refer the reader to \cite[pp. 291--294]{LS}, and the references therein. 
We outline here only what we need below, in particular defining diagrams on a disk. %sphere with $k$ boundary components, for some integer $k\ge 1$.
Let $\widetilde \cR$ be a symmetrised subset of $F$ and $w$ a reduced factorisation of an element of $F$. An $\widetilde \cR$ 
\emph{diagram with boundary label} $w$ \emph{over} $F$ consists of the following. A finite $2$-complex $M$ with underlying space $\S$  a compact, connected, 
simply connected, subset of the real plane; and a
 distinguished vertex $O$ of $M$ on $\pd \S$. ($0$-cells, $1$-cells and $2$-cells of $M$ are called vertices, edges and
 regions, respectively.) A labelling function $\phi$ from oriented edges of $M$   to $H\cup \{ t^{\pm 1}\}$. (Strictly speaking $\phi$ maps edges of $M$ to freely 
 reduced words in $(A\bs \{t\})^{\pm 1}$ or elements of   $\{ t^{\pm 1}\}$.)
 For an oriented edge $a$ % with initial and terminal $0$-cells $u$ and $v$ (not necessarily distinct) 
 we write $\bar a$ for the same edge given the opposite orientation (if $f:[0,1]\maps a$ is a homeomorphism determining the oriented edge
 $a$ then $\bar a$ is the edge determined by the map $\bar f$ mapping $x\in [0,1]$ to $f(1-x)$.) 
A \emph{boundary cycle} of a region $\D$ of $M$ is a closed path traversing the boundary $\pd \D$ of $\D$ exactly once (beginning and ending at a vertex $v$ of $M$).
A \emph{boundary cycle} of $M$ is a closed path traversing the boundary $\pd \S$ of $\S$ exactly once, beginning and ending at $O$. 
In addition the following conditions must be satisfied. 
\be[(1)]
\item\label{it:diag1}  If $a$ is an oriented edge with label $w=\phi(a)$ then $\phi(\bar a)=w^{-1}$. 
\item\label{it:diag2}  If $M$ has a boundary cycle $p=a_1,\ldots, a_n$, beginning and ending at $O$, then %$w$ has a reduced factorisation 
the product $\phi(a_1)\cdots \phi(a_n)$ is reduced and equal, in $F$, to $w$ or $w^{-1}$. 
\item\label{it:diag3}  
If $\D$ is a region of $M$ and  $\D$ has a boundary cycle $p=a_1,\ldots, a_n$, then the product $\phi(a_1)\cdots \phi(a_n)$ is reduced and 
equal in $F$ to 
an element $\widetilde \cR$. 
\ee
It follows (see for example \cite[Theorem 11.5, p. 292]{LS}) that there exists a diagram $\cD$ over $F$ with boundary label $w$ 
if and only if $w=1$ in $F/N$, where $N$ is the normal closure of $\cR$ in $F$.\label{p:diagram_prop}

If the label $\phi(a)$ of an edge $a$ is in $H$ then  
 $a$ is called an $H$\emph{-edge}, 
 and if  $\phi(a)=t^{\pm 1}$ then   $a$ is called a $t$\emph{-edge}. If $v$ is a vertex in the boundary $\pd \D$ of a region $\D$ and
 $v$ is incident to   %both and $H$-edge and 
 a $t$-edge of $\pd \D$, then we call $v$ a \emph{primary vertex}, with respect to $\D$. The products $\phi(a_1)\cdots \phi(a_n)$ appearing
 in \ref{it:diag2} and \ref{it:diag3} are called \emph{boundary labels} of $M$ and $\D$ respectively. In our case $\widetilde \cR$ is
 the symmetrised closure of the single element $r$, which is cyclically minimal as an element of the partially commutative group $\GG$.
 Every element $w$ of $\widetilde \cR$ is therefore equal in $F$ to $u^{-1}r'u$, for some $u\in U$ and cyclic permutation $r'$ of $r$ or $r^{-1}$.
 Let $\D$ be a region with boundary label equal
 in $F$ to $r'$, when read from an appropriate vertex $o$ on its boundary.
Attaching an edge $e$ labelled $u$ by identifying its initial
 vertex with $o$, we obtain a diagram with a single region $\D$ and boundary label $w$. 
 As $r'$ is cyclically minimal we may replace \ref{it:diag3} with
 \be[label=(\arabic*')]
 \addtocounter{enumi}{2}
\item\label{it:diag3'} If $\D$ is a region of $M$ and  $\D$ has a boundary cycle $p=a_1,\ldots, a_n$,
  then the product $\phi(a_1)\cdots \phi(a_n)$ is reduced and 
equal in $F$ to a cyclic permutation of $r$ or $r^{-1}$. 
  \ee
%%%ajd%%% end cancel put here

As usual we shall restrict to diagrams which do not have pairs of
redundant regions of the following sort.  Let $\D_1$ and $\D_2$ be distinct regions of the diagram $\cD$   
with boundary cycles $\r_1=\mu\nu_1$ and $\r_2=\mu\nu_2$ (where $\mu, \nu_i$ are subpaths of $\r_i$)  
 meeting in the connected boundary component $\mu$. 
If $\phi(\nu_1)\equiv\phi(\nu_2)$ (as words in the free monoid $(A\cup A^{-1})^*$) then,
%In this case,
 as in Figure \ref{fig:regioncancel}, we may remove the interior  
of $\D_1\cup \D_2$ and identify the sub-paths
$\nu_1$ and $\nu_2$ of the boundary cycles of $\D_1$ and $\D_2$, 
via their (equal) labels $\phi(\nu_1)$ and $\phi(\nu_2)$, to leave a new diagram $\cD'$ with the same reduced boundary label as $\cD$ but fewer regions. 
The modification of $\cD$ to produce $\cD'$ is called a \emph{cancellation 
  of regions}. A diagram in which no cancellation of regions is possible is called \emph{reduced}.  

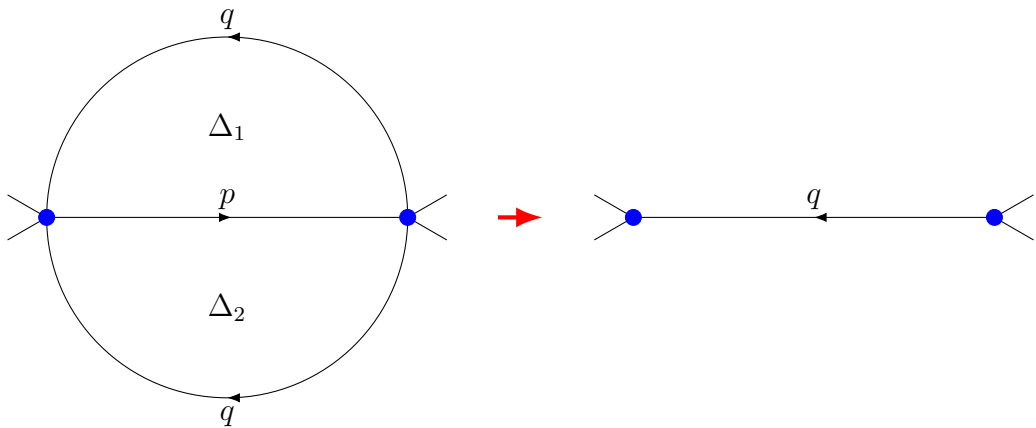
\begin{figure}
\begin{center}
\pgfmathsetmacro{\di}{4}
\begin{tikzpicture}[scale=.6,arrowmark/.style 2 args={decoration={markings,mark=at position #1 with \arrow{#2}}}]%
  \tikzstyle{every node}=[circle, draw, fill=blue, color=blue,
  inner sep=0pt, minimum width=6pt]
  \draw (-{\di},0) node {};
  \draw ({\di},0) node {};
  \begin{pgfonlayer}{background}
    \draw (0,0) circle ({\di});
    \draw (-{\di},0) -- ({\di},0);
    \draw[-{Latex}] (0,0) -- (0.1,0);
    \draw[-{Latex}] (0.1,{\di}) -- (0,{\di});
    \draw[-{Latex}] (0.1,-{\di}) -- (0,-{\di});
    \draw (-{\di},0) +(150:1) -- (-{\di},0); 
    \draw (-{\di},0) +(210:1) -- (-{\di},0);
    \draw ({\di},0) +(30:1) -- ({\di},0); 
    \draw ({\di},0) +(-30:1) -- ({\di},0);
  \end{pgfonlayer}
  \draw (0,0) +(0,.4) node[draw=none,fill=none,color=black] {$p$};
  \draw (0,{\di}) +(0,.4) node[draw=none,fill=none,color=black] {$q$};
  \draw (0,-{\di}) +(0,-.4) node[draw=none,fill=none,color=black] {$q$};
  \draw (0,\di/2) node[draw=none,fill=none,color=black] {$\D_1$};
  \draw (0,-{\di}/2) node[draw=none,fill=none,color=black] {$\D_2$};
  %%%%%%%%%%%%% 
  %%%%%%%%%%%%% 
  \draw[-{Latex}, ultra thick,color=red] ({\di+2},0) -- (\di+3,0);
  %%%%%%%%%%%%
  %%%%%%%%%%%%
  \begin{scope}[shift={(\di+9,0)}]
    \draw (-{\di},0) node {};
  \draw ({\di},0) node {};
  \begin{pgfonlayer}{background}
    \draw (-{\di},0) -- ({\di},0);
    \draw[-{Latex}] (0.1,0) -- (0,0);
%    \draw[-{Latex}] (0.1,{\di}) -- (0,{\di});
 %   \draw[-{Latex}] (0.1,-{\di}) -- (0,-{\di});
    \draw (-{\di},0) +(150:1) -- (-{\di},0); 
    \draw (-{\di},0) +(210:1) -- (-{\di},0);
    \draw ({\di},0) +(30:1) -- ({\di},0); 
    \draw ({\di},0) +(-30:1) -- ({\di},0);
  \end{pgfonlayer}
  \draw (0,0) +(0,.4) node[draw=none,fill=none,color=black] {$q$};
%  \draw (0,{\di}) +(0,.4) node[draw=none,fill=none,color=black] {$q_1$};
%  \draw (0,-{\di}) +(0,-.4) node[draw=none,fill=none,color=black] {$q_2$};
%  \draw (0,\di/2) node[draw=none,fill=none,color=black] {$D_1$};
%  \draw (0,-{\di}/2) node[draw=none,fill=none,color=black] {$D_2$};
  \end{scope}
\end{tikzpicture}
\end{center}
\caption{Cancelling regions: $\phi(\mu)\equiv p, \phi(\nu_1)\equiv\phi(\nu_2)\equiv q$}\label{fig:regioncancel}
\end{figure}
%%%%%%%%%%%%%%%%%%%%%%%%%%%
%%%%%%%%%%%%%%%%%%%%%%%%%%%
%%
%The result is 
%%%ajd%%% end cancel put here
It turns out that if a pair of distinct regions $\D_1$ and $\D_2$ have boundary  
 cycles $\r_1=\mu\nu_1$ and $\r_2=\mu\nu_2$, as above, but instead of satisfying $\phi(\nu_1)\equiv \phi(\nu_2)$ the labels satisfy only   
 $\phi(\nu_1)=_F \phi(\nu_2)$ then, after some minor modifications to the diagram, which do not alter the label of its boundary,
 we may again cancel the regions $\D_1$ and $\D_2$. The general process of modification is described in \cite[Page 292]{LS}.
 Here we  describe modifications sufficient for our particular case. 
 %that may be applied to a given diagram, without altering the element of $F$ represented
 %by its boundary label, which remains reduced.
 There are two types of these; 
 the first of which results in a new diagram in which the product of edge labels around a boundary cycle may not be freely reduced.
 The second type may then be
 used to freely reduce labels on edges of boundary cycles if necessary.  

\paragraph{Shuffling labels.}
Suppose that  a subpath $\mu$ of the boundary of a region $\D$ has label with subword   
$x y$, where $x,y \in A^{\pm 1}$ and $xy=yx$.  Then we may modify the 
diagram $\cD$ (without altering the element of  $F$ represented by its boundary label) to obtain a new diagram %, with boundary label $w$, 
 in which this subword $xy$ is replaced by $yx$, as follows. 
 If  $xy$ occurs as a subword of the label $\phi(e)=axyb$ of  a single edge $e$ of $\mu$ then we modify 
$\phi$ by setting $\phi(e)=ayxb$. Otherwise $\mu$ contains a subpath $e_1e_2$, where $e_1$ and $e_2$ are edges, 
such that $\phi(e_1)=ax$ and  $\phi(e_2)=yb$. In this case $e_1\subseteq \D\cap \D'$ and $e_2\subseteq \D\cap \D''$, for 
some regions $\D'$ and $\D''$ and we modify $\cD$ as shown in Figure \ref{fig:xyshuffle}.
There is a choice here: we may label the new edge, in $\D'\cap\D''$ with $yx^{-1}$, as shown, or with $x^{-1}y$. 
% (followed by free reduction of edge labels if possible). 
 In both cases we refer to this modification as  a \emph{shuffle of the label of} $\D$. \label{p:shuffle}
%%%%%%%%%%%%%%%%%%%%%%%%%%%
%%%%%%%%%%%%%%%%%%%%%%%%%%%
\begin{figure}
\begin{center}
\begin{tikzpicture}[scale=.75,arrowmark/.style 2 args={decoration={markings,mark=at position #1 with \arrow{#2}}}]%
  \tikzstyle{every node}=[circle, draw, fill=blue, color=blue,
  inner sep=0pt, minimum width=6pt]
  \foreach \x  in {0,...,2} {
    \draw (2*\x,0) node {};
  }
  \begin{pgfonlayer}{background}
    \draw (0,0) +(240:1) -- (0,0);
    \draw (0,0) +(120:1) -- (0,0);
    \draw (4,0) +(60:1) -- (4,0);
    \draw (4,0) +(-60:1) -- (4,0);
    \foreach \x/\y in {0/{ax},1/{yb}} {
      \draw[-{Latex}] (2*\x,0) -- (2*\x+1.1,0);
      \draw ((2*\x+1,0) -- (2*\x+2,0);
      \draw (2*\x+1,0) +(0,-.4) node[draw=none,fill=none,color=black] {$\y$};
    }
    \draw (2,0) +(60:1.5) -- (2,0);
    \draw (2,0) +(120:1.5) -- (2,0);
  \end{pgfonlayer}
  \draw (2,-1.5) node[draw=none,fill=none,color=black] {$\D$};
  \draw (.7,1) node[draw=none,fill=none,color=black] {$\D'$};
  \draw (3.3,1) node[draw=none,fill=none,color=black] {$\D''$};
  %%%%%%%%%%%%%
  %%%%%%%%%%%%%
  \draw[-{Latex}, ultra thick,color=red] (6,0) -- (7,0);
  %%%%%%%%%%%%
  %%%%%%%%%%%%
  \begin{scope}[shift={(9,0)}]
    \foreach \x  in {0,...,2} {
      \draw (2*\x,0) node {};
    }
    \draw (2,2) node {};
    \begin{pgfonlayer}{background}
      \draw (0,0) +(240:1) -- (0,0);
      \draw (0,0) +(120:1) -- (0,0);
      \draw (4,0) +(60:1) -- (4,0);
      \draw (4,0) +(-60:1) -- (4,0);
      \foreach \x/\y in {0/{ay},1/{xb}} {
        \draw[-{Latex}] (2*\x,0) -- (2*\x+1.1,0);
        \draw ((2*\x+1,0) -- (2*\x+2,0);
        \draw (2*\x+1,0) +(0,-.4) node[draw=none,fill=none,color=black] {$\y$};
      }
      \draw (2,2) +(60:1.5) -- (2,2);
      \draw (2,2) +(120:1.5) -- (2,2);
      \draw[-{Latex}] (2,2) --(2,0.9);
      \draw (2,1) -- (2,0);
    \end{pgfonlayer}
    \draw (2,1) +(.8,0) node[draw=none,fill=none,color=black] {$yx^{-1}$};
    \draw (2,-1.5) node[draw=none,fill=none,color=black] {$\D$};
    \draw (.7,1.8) node[draw=none,fill=none,color=black] {$\D'$};
    \draw (3.3,1.8) node[draw=none,fill=none,color=black] {$\D''$};
  \end{scope}
\end{tikzpicture}
\end{center}

\caption{Shuffling in the boundary label of $\D$.}\label{fig:xyshuffle}
\end{figure}
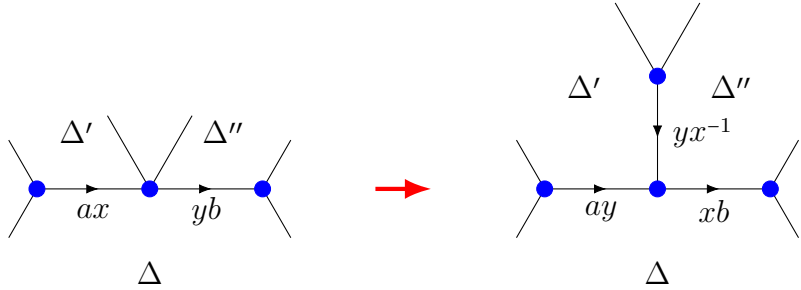
 
\paragraph{Free reduction of labels.}
 Let $\cD$ be an $\widetilde \cR$-diagram over $F$ %, with boundary label $w$,
  and let $\D$ be a region of $\cD$. % As described in \cite{LS}, 
 Suppose a subpath $\mu$ of a  boundary cycle of $\D$ has label containing a subword $xx^{-1}$, where 
$x\in A^{\pm 1}$. If  $xx^{-1}$ is a subword of  the label $\phi(e)=axx^{-1}b$ of a single edge $e$, where $a,b\in H$, then 
we modify $\phi$ by setting $\phi(e)=ab$. (Such edges do not occur in $\cD$ itself, 
but may arise after a shuffling of labels as above.)  Otherwise $\mu$ contains a subpath
$e_1e_2$ of two edges such that $\phi(e_1)=ax$ and  $\phi(e_2)=x^{-1}b$, where $a,b\in H$ (and if $x=t^{\pm 1}$ then necessarily $a=b=1$,
since $\phi(e_i)$ is by definition either an element of $H$ or of $\{t^{\pm 1}\}$).
In this case we modify $\cD$ as shown in Figure \ref{fig:xcancel}. We call this \emph{free reduction of the label of} $\D$. 
 Note that this reduces the sum of lengths of boundary labels of regions of $\cD$, where length here means length as a 
word over $A\cup A^{-1}$, so we may repeat free reductions of labels until no such subpaths occur in 
the boundary label of any region.  
%%%%%%%%%%%%%%%%%%%%%%%%%%%
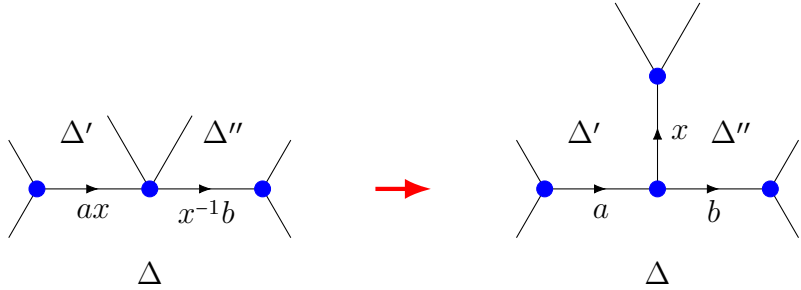
\begin{figure}
  \begin{center}
\begin{tikzpicture}[scale=.75,arrowmark/.style 2 args={decoration={markings,mark=at position #1 with \arrow{#2}}}]%
  \tikzstyle{every node}=[circle, draw, fill=blue, color=blue,
  inner sep=0pt, minimum width=6pt]
  \foreach \x  in {0,...,2} {
    \draw (2*\x,0) node {};
  }
  \begin{pgfonlayer}{background}
    \draw (0,0) +(240:1) -- (0,0);
    \draw (0,0) +(120:1) -- (0,0);
    \draw (4,0) +(60:1) -- (4,0);
    \draw (4,0) +(-60:1) -- (4,0);
    \foreach \x/\y in {0/{ax},1/{x^{-1}b}} {
      \draw[-{Latex}] (2*\x,0) -- (2*\x+1.1,0);
      \draw ((2*\x+1,0) -- (2*\x+2,0);
      \draw (2*\x+1,0) +(0,-.4) node[draw=none,fill=none,color=black] {$\y$};
    }
    \draw (2,0) +(60:1.5) -- (2,0);
    \draw (2,0) +(120:1.5) -- (2,0);
  \end{pgfonlayer}
  \draw (2,-1.5) node[draw=none,fill=none,color=black] {$\D$};
  \draw (.7,1) node[draw=none,fill=none,color=black] {$\D'$};
  \draw (3.3,1) node[draw=none,fill=none,color=black] {$\D''$};
  %%%%%%%%%%%%%
  %%%%%%%%%%%%%
  \draw[-{Latex}, ultra thick,color=red] (6,0) -- (7,0);
  %%%%%%%%%%%%
  %%%%%%%%%%%%
  \begin{scope}[shift={(9,0)}]
    \foreach \x  in {0,...,2} {
      \draw (2*\x,0) node {};
    }
    \draw (2,2) node {};
    \begin{pgfonlayer}{background}
      \draw (0,0) +(240:1) -- (0,0);
      \draw (0,0) +(120:1) -- (0,0);
      \draw (4,0) +(60:1) -- (4,0);
      \draw (4,0) +(-60:1) -- (4,0);
      \foreach \x/\y in {0/{a},1/{b}} {
        \draw[-{Latex}] (2*\x,0) -- (2*\x+1.1,0);
        \draw ((2*\x+1,0) -- (2*\x+2,0);
        \draw (2*\x+1,0) +(0,-.4) node[draw=none,fill=none,color=black] {$\y$};
      }
      \draw (2,2) +(60:1.5) -- (2,2);
      \draw (2,2) +(120:1.5) -- (2,2);
      \draw[-{Latex}] (2,0) --(2,1.1);
      \draw (2,1) -- (2,2);
    \end{pgfonlayer}
    \draw (2,1) +(.4,0) node[draw=none,fill=none,color=black] {$x$};
    \draw (2,-1.5) node[draw=none,fill=none,color=black] {$\D$};
    \draw (.7,1) node[draw=none,fill=none,color=black] {$\D'$};
    \draw (3.3,1) node[draw=none,fill=none,color=black] {$\D''$};
  \end{scope}
\end{tikzpicture}
\end{center}

\caption{Free reduction of the boundary label of $\D$}\label{fig:xcancel}
\end{figure}

Note that the boundary label of $\cD$ is unaffected by such modifications. 
Both free reduction and shuffling of labels of $\cD$ result in a new diagram $\cD'$, which has regions in one to one correspondence
with the regions of $\cD$. %; indeed  with the same names.
If $\D_1$ and $\D_2$ are regions of $\cD$ with a common boundary component $\mu$ and boundary cycles $\mu\nu_1$ and $\mu\nu_2$, respectively,
such that
$\phi(\nu_1)=_F\phi(\nu_2)$ then
these modifications may be used  alter the diagram so that  $\phi(\nu_1)\equiv \phi(\nu_2)$; and $\D_1$ and $\D_2$ become
cancelling regions. %We obtain the following version of \cite[Theorem 11.5]{LS}. 
% In view of this Lemma
 We say a diagram $\cD$ is \emph{strongly reduced} if it satisfies the condition that, whenever
$\D_1$ and $\D_2$ are distinct regions with a common boundary component $\mu$, of positive length, the label $\phi(\mu)$ is
a piece. From \cite[Theorem 11.5]{LS} we obtain the following. 
\begin{prop}%[{\emph{cf.}} {\cite[Theorem 11.5]{LS}}]
  \label{prop:Fcancel}
  Let $\cR$ be a subset of $F$ and $\widetilde{\cR}$ its symmetric closure. There exists a strongly reduced $\widetilde \cR$ diagram $\cD$ over
  $F$ with boundary label $w$ if and only if $w=1$ in $F/N$, where $N$ is the normal closure of $\cR$ in $F$. 
\end{prop}
\subsection{Reduction to the free product}\label{sec:redn}

In this section we specialise the methods of \cite{J} to the case in hand. 
As above, let $D$ be a set of double coset representatives of $U$ in $H$ satisfying the properties of Corollary \ref{lem:doublecoset}, 
and let $\s:H\maps D$ be the function mapping $h\in H$ to $d\in D$ such that $UhU=UdU$. 
Partition the set $D$ into disjoint subsets $D^+$, $D^-$ and $\{1_H\}$ such that $d\in D^+$ if and only if $d^{-1}\in D^-$. Now define the free product  
$F_D=\FF(D^+)\ast \la t\ra$, where $\FF(D^+)$ is the free group on $D^+$. 
%\ajd{That is, $F_D$ is the free group on $D^+\cup \{t\}$.}~\\

The canonical map from $D^+$ to $\FF(D^+)$ extends to an injective map from $D$ to $\FF(D^+)$ by mapping $d^{-1}$ in $D^{-}$ to $d^{-1}\in \FF(D^+)$, for all $d\in D^+$, 
and mapping $1_H$ to the empty word. Composing this map with the canonical injection from $\FF(D^+)$ into $F_D$ we have an injective map
$\i$ from $D$ to $F_D$. The composition $\i\circ \s$ is then a map from $H$ to $F_D$, which we shall now also refer to as $\s$.
%
%Composing $\s$ with this map followed by the injection from $\FF(D^+)$ into $F_D$ gives a map from $H$ into 
%$F_D$, which we shall now also refer to as $\s$. 
We extend $\s$ to a function from  $F$ to $F_D$ as follows. 
If $p\in F$ has reduced factorisation  $p=g_0t^{\e_1}\cdots t^{\e_n}g_n$ then we define 
\[\s(p)=\s(g_0)t^{\e_1}\cdots t^{\e_n}\s(g_n).\]
If $g,h\in H$ and $u\in U$ are such that $h=ug$ then $\s(g)=\s(h)$, so $\s$ is a well-defined map
from $F$ to $F_D$. 
To simplify notation we may write $\bar p$ for $\s(p)$.
\begin{defn}\label{def:troot} 
If $p\in F$ is such that $\s(p)$ is a not a proper power in $F_D$ then we say that $p$ is a $t$\emph{-root}. 
\end{defn}
\begin{comment}{Here is an example to show  that the ``factors'' of a  $t$-root
   do not generate a free group. Let $\G$ be the
  graph with vertices $A=\{t,a,b,c,d\}$ and edges $\{a,t\}$, $\{d,t\}$,
  $\{a,b\}$, $\{b,c\}$, $\{c,d\}$ and $\{d,a\}$. (This is $C'_5$ elsewhere
  in the text.) Let $s=dcbtacbtcbt$.  Here $cb\in D$ (as it has no left or right divisor in $U=\la a,d\ra$; so $\s(s)=(cbt)^3$, and $s$ is a
  $t$-root. The ``factors'' of the reduced form
$s=dcbtacbtcbt$
are all in  $UDt$, and are $x=dcbt$, $y=acbt$ and $z=cbt$.
Let $K=\la x,y,z\ra$. 
As $x$ and
$z$ do not commute, together they generate a free group of rank 2 (any
2 elements of a pc group generate  a free abelian group
or a free group). Thus $K$ has a free subgroup of rank 2.
Also $xz^{-1}=d$ and $xy^{-1}=da^{-1}$,
so $[xz^{-1},xy^{-1}]=1$. As $d$ and $da^{-1}$ do not have a common root, $K$ is not free.}
\end{comment}

Note that %the property of being a $U$-root is independent of the choice of set of double coset representatives $D$ and 
 a cyclically reduced element of $F$ of $t$-length at least $1$, which ends in a $t$ letter and  
which is  a $t$-root; cannot be a proper power in $F$. Indeed, 
suppose that  $p$ is  a cyclically reduced element of $F$ ending in a $t$-letter, with $|p|_t\ge 1$, 
and  $p=q^n$, for some element $q\in F$ and positive integer $n$. It follows, from \cite[Chapter IV, Section 2]{LS}, 
that $q$ is also cyclically reduced, of positive $t$-length and ends in a $t$-letter. 
Therefore $\s(p)=\s(q^n)=\s(q)^n$, so $p$ is not a $t$-root. (Note that $\s$ is not a group homomorphism.)

Let $n\ge 2$ be a positive integer and let $s$ be a 
cyclically reduced word of non-zero $t$-length, 
 $s=h_0t^{\e_1}\cdots h_{m-1}t^{\e_m}$,
where $h_i\in \GG_t$ and $\e_i\in\{\pm 1\}$, for $i=1,\ldots ,m-1$.  
Let $r=s^n$, let $N$ be the normal closure of $\cR=\{r\}$ in $\GG$  
and let $G=\GG/N$.
If $r_1$ and $r_2$ are %distinct
elements of the symmetrised closure $\widetilde \cR$, of $\cR$, such that $\s(r_1)=\s(r_2)$ we say that $r_1$ and $r_2$ are in \emph{periodic position}.

\begin{lemma}\label{lem:perpospiece}
Assume $s$ has $t$-length $m\ge 1$, ends in  a $t$-letter, 
is cyclically $t$-thick and a cyclically reduced $t$-root. Let $r=s^n$, for some $n\ge 1$ and $\cR=\{r\}$.
If $r_1$ and $r_2$ are cyclic permutations of $r$ or $r^{-1}$ such that  $r_1=pq_1$ and $r_2=pq_2$  and $r_1$ and $r_2$ are in
periodic position, then $\phi(q_1)=_F\phi(q_2)$. 
\begin{comment}
Let  $p$ be a piece over $\widetilde \cR$ such that  $r_1=pq_1$ and $r_2=pq_2$ are distinct elements of $\widetilde \cR$ and
$p$ is a right integral cyclic subword of 
$r_1$. If $r_1$ and $r_2$ are in
periodic position then $p=u_0t^\e  \cdots u_{k-1} t^\e$, for some $u_i\in U$, $k\in \ZZ$  and $\e=\pm 1$. 
\end{comment}
\end{lemma}

\begin{proof}
We may assume that $r_1=r$ and $r_2$ is obtained from $r^{\e}$, where $\e=\pm 1$, by cyclic permutation. 
 If $\e=-1$ then $\s(r_2)$ is a cyclic permutation of $\s(r_1^{-1})=\s(r_1)^{-1}$ and $\s(r_1)=\s(r_2)$ in $F_D$. This implies
$\s(r_1)=1$, so $r_1\in U$,  a contradiction. Therefore $\e=1$ and as 
$s$ is  a $t$-root it follows that $r_2$ is obtained from $r$ by a cyclic permutation of length $k|s|=2km$, for some positive integer $k$.
Hence $r_2=_Fr$ and so $\phi(q_1)=_F\phi(q_2)$.
\begin{comment}
As $r_1$ and $r_2$ are in $\widetilde \cR$ and $\s(r_1)=\s(r_2)$ we may assume that $r_1=r$ and $r_2$ is obtained from $r^{\e}$, where $\e=\pm 1$, by cyclic permutation followed by conjugation by an element of $U$. 
If $\e=-1$ then $\s(r_2)$ is a cyclic permutation of $\s(r_1^{-1})=\s(r_1)^{-1}$ and $\s(r_1)=\s(r_2)$ in $F_D$. This implies
$\s(r_1)=1$, so $r_1\in U$,  a contradiction. Therefore $\e=1$ and as 
$s$ is  a $t$-root it follows that $r_2$ is obtained from $r$ by a cyclic permutation of length $k|s|=2km$, for some positive integer $k$. Without loss 
of generality, we may therefore assume that 
$r_1=r$ and $r_2=u^{-1}ru$, for some non-trivial $u\in U$. Therefore $u^{-1}pq_1u=pq_2$ and, since $p$ is a right integral cyclic subword of $r_1$, 
we may assume that we have reduced factorisations $r_1=h_0t^{\e_1}\cdots h_{l-1}t^{\e_l}$, where $l=mn$,  and 
$p=h_0t^{\e_1}\cdots h_{k-1}t^{\e_k}$, for some $k\le l$. 
Then $h_0^{-1}uh_0\in  U$, with $h_0\in U \cup \maln(U)$. This implies that $h_0\in U$. Similarly $h_1\in U$ and so $\e_1=\e_2$. In the same 
way we find that $h_2, \ldots , h_{k-1}\in U$, so $\e_i=\e_1$, for $i=1,\ldots, k$.  %If $h'_k=h_k$ then also $h_k\in U$. 
% As $h_0t^{\e_1}\cdots h_{m-1}t^{\e_m}$ is a reduced factorisation of $r_1$, we have $\e_1=\cdots =\e_k=\e$, for some $\e\in\{\pm 1\}$. 
\end{comment}
\end{proof}
\begin{corollary}\label{cor:noperpos}
  If $M$ is a strongly reduced diagram and $\mu$ is the common boundary component of regions $\D_1$ and $\D_2$, with
  boundary cycles $\rho_1=\mu\nu_1$ and $\rho_2=\mu\nu_2$, respectively, then $\phi(\rho_1)$ and $\phi(\rho_2)$ 
  are not in periodic position. In particular, $\phi(\mu)$ is a piece. 
\end{corollary}
\begin{proof}
  As $M$ is strongly reduced $\phi(\mu)$ is a piece, so from the previous lemma and condition \ref{it:diag3'} for diagrams,
  $\phi(\rho_1)$ and $\phi(\rho_2)$ 
  are not in periodic position.
\end{proof}
For the remainder of this section assume that $s$ and $r$ are elements
of $F$ satisfying the hypotheses of Lemma \ref{lem:perpospiece}. 
Let $\FF(X)$ be the free group on a basis $X$ and let  $w$ be an an element of $\FF(X)$, written as a reduced word.  
We say a cyclic subword $a$ of $w$ is
\emph{uniquely positioned} if no other cyclic subword of $w$ or $w^{-1}$ is equal to $a$. 
%Now suppose that $\cR=\{r\}$, where $r=s^n$, for some cyclically reduced $U$-root $s$ of $F$,  and let $\widetilde{\cR}$ be the symmetric closure of $\cR$.
As $\bar s$ is cyclically reduced and not a proper power in $F_D$, it follows from \cite[Theorem 2.1]{DH} that $\bar s$  
has a cyclic permutation with reduced factorisation equal to
$\bar a\bar b$, where $\bar a$ and $\bar b$ are non-empty uniquely positioned subwords of $\bar s$. Therefore $s$ has a cyclic permutation
$\tilde s$, with a reduced factorisation $\tilde s=ab$, such that $a$ and $b$ are integral cyclic subwords of $s$, $\s(a)=\bar a$ and $\s(b)=\bar b$.
In this definition, the words $a$ and $b$ are not necessarily uniquely determined. Suppose $s$ has a cyclic permutation $\tilde s$ which factors
as $\tilde s=w_0t^\e ut^\e w_1$, where $u\in U$, then $\sigma(\tilde s)=\bar w_0t^\e t^\e\bar w_1$. If, in this expression,
$\bar a$ ends with the first occurrence of $t^\e$
and $\bar b$ begins with the second, then $a$ may be chosen to end $t^\e u$ or $t^\e$, with $b$ beginning $t^\e$ or $ut^\e$, respectively. To
avoid such ambiguity, in this situation, we always choose $a$ to be right integral, that is $a$ ends in $t^\e$.
Similarly, if this situation arises with the roles of $a$ and $b$ interchanged, we choose $b$ to be
right integral. With this convention, the words $a$ and $b$ are uniquely determined cyclic subwords of $s$.  

Now let $\D$ be a region, of an $\widetilde \cR$ diagram over $F$,
with boundary cycle $\rho$, such that $\phi(\rho)=r_0\in \widetilde \cR$. By changing the 
base point of $\rho$, if necessary, we may assume that $r_0$ has reduced factorisation 
$r_0=(g_0t^{\e_1}\cdots g_{m-1}t^{\e_m})^n$, where $s_0=g_0t^{\e_1}\cdots g_{m-1}t^{\e_m}$ is a cyclic permutation of $s$ or $s^{-1}$. 
The primary vertices of $\D$ form a subsequence $\d_1,\g_1\ldots, \d_{mn},\g_{mn}$ of the vertex sequence 
of $\rho$, such that, setting $\g_0=\g_{mn}$,
$\phi([\d_{i+km},\g_{i+km}])=t^{\e_i}$ and $\phi([\g_{i-1+km},\d_{i+km}])=g_{i-1}$, 
 for 
$1\le i\le m$ and $0\le k\le n-1$ (where $[x,y]$ denotes the subpath of $\rho$ from vertex $x$ to vertex $y$). 
In this notation, if $g_i=1$ then the primary vertices $\g_{i+km}$ and $\d_{i+1+km}$ are the same.

Let $\overline{\phi}=\s\circ \phi$, defined on subintervals of $\pd \D$ beginning and ending
at primary vertices; %so $\rho$ is also labelled by $\overline{\phi}$ and 
%the boundary cycle with label $\phi(\rho)=r_0$ also has 
%label 
so 
\[\overline{\phi}([\g_{i-1+km},\d_{i+km}])=\bar g_{i-1},\]
and  
\[
  \overline{\phi}([\g_{km},\g_{(k+1)m}])=\overline{s}_0 =\bar g_0t^{\e_1}\cdots \bar g_{m-1}t^{\e_m}\in F_D,\]
  for $0\le k\le n-1$. 
  From the above, $s_0$ has a cyclic permutation which factorises
  as $ab$, where $a$ and $b$ are integral cyclic subwords of $s$, and  %definitions, $\overline{s}_0$ belongs to the symmetrised closure $\bar R$ of $\s(r)$, 
 $\overline{s}_0$ has a corresponding cyclic permutation which factorises as $\bar a\bar b$, where $\bar a$, $\bar b$
 are uniquely positioned  cyclic subwords of a cyclic permutation of $\bar s$. % and $s$ has a cyclic permutation which
  %factors as $ab$, where $\s(a)=\bar a$ and $\s(b)=\bar b$.  
%as in the previous paragraph. 
Consequently, $\bar r_0=(\bar s_0)^n$ has a cyclic permutation which factorises as $(\bar a\bar b)^n$. 
  Therefore there is a subsequence   $\a_1,\b_1, \ldots \a_n,\b_n$ of the sequence   $\d_1,\g_1,\ldots, \d_{mn},\g_{mn}$  of primary 
vertices of $\D$,     
such that $\phi([\a_i,\b_i])=a$ and  $\phi([\b_i,\a_{i+1}])=b$, and %$\a_1,\b_1,\a_2,\b_2,\ldots, \a_m, \b_m$ is a cyclic subsequence of the vertex sequence of the boundary cycle 
%$\rho=e_0f_1 \cdots e_{n-1}f_n$, with  
  $\bar\phi([\a_i,\b_i])=\bar a$ and  $\bar \phi([\b_i,\a_{i+1}])=\bar b$, for $i=1,\ldots ,n$; 
  % Then ;
  as illustrated in Example \ref{ex:sepD}. % and .
\begin{example}\label{ex:sepD}
  Figure \ref{fig:sepD} illustrates a possible distribution of the first four of the 
vertices $\a_i$, $\b_i$ on $\pd \D$, assuming that 
 $s_0=g_0t g_1t^{-2}$, $g_0=g'_0g_0''$, $g_1=g_1'g_1''$, $\bar a=t\bar g_1 t^{-1}$ and $\bar b=t^{-1}\bar g_0$. In the diagram, primary vertices
are those  with the larger diameter. 
  \begin{figure}
\begin{center}

\begin{tikzpicture}[scale=.75,arrowmark/.style 2 args={decoration={markings,mark=at position #1 with \arrow{#2}}}]%
  \tikzstyle{every node}=[circle, draw, fill=blue, color=blue,
  inner sep=0pt, minimum width=6pt]
\pgfmathsetmacro{\l}{1.5}
\draw (\l*0,0) node[color=black] {};
\draw (\l*1,0) node {};
\draw (\l*2,0) node[color=black] {};
\draw (\l*3,0) node[color=black, minimum width=3pt] {};
\draw (\l*4,0) node[color=black] {};
\draw (\l*5,0) node[color=red] {};
\draw (\l*6,0) node[color=black] {};
\draw (\l*7,0) node[color=black, minimum width=3pt] {};
\draw (\l*8,0) node {};
\draw (\l*9,0) node[color=black] {};
\draw (\l*10,0) node[color=black] {};
\draw (\l*11,0) node[color=red] {};
\draw (\l*12,0) node[color=black] {};
\draw (\l*0,0) +(0.75,0.4) node[draw=none,fill=none,color=black] {$g_0$};
%\draw (\l*1,0)  node[draw=none,fill=none,color=black] {$$};
\draw (\l*2,0) +(-0.65,0.4)  node[draw=none,fill=none,color=black] {$t$};
\draw (\l*3,0)  +(-0.65,0.4) node[draw=none,fill=none,color=black] {$g'_1$};
\draw (\l*4,0) +(-0.65,0.4)  node[draw=none,fill=none,color=black] {$g''_1$};
\draw (\l*5,0) +(-0.65,0.4)  node[draw=none,fill=none,color=black] {$t$};
\draw (\l*6,0) +(-0.65,0.4)  node[draw=none,fill=none,color=black] {$t$};
\draw (\l*7,0) +(-0.65,0.4)  node[draw=none,fill=none,color=black] {$g'_0$};
\draw (\l*8,0) +(-0.65,0.4)  node[draw=none,fill=none,color=black] {$g''_0$};
\draw (\l*9,0) +(-0.65,0.4)  node[draw=none,fill=none,color=black] {$t$};
\draw (\l*10,0) +(-0.65,0.4) node[draw=none,fill=none,color=black] {$g_1$};
\draw (\l*11,0) +(-0.65,0.4) node[draw=none,fill=none,color=black] {$t$};
\draw (\l*12,0) +(-0.65,0.4) node[draw=none,fill=none,color=black] {$t$};
\draw (\l*1,0) +(0,-0.4*\l) node[draw=none,fill=none,color=black] {$\a_1$};
\draw (\l*8,0) +(0,-0.4*\l) node[draw=none,fill=none,color=black] {$\a_2$};
\draw (\l*5,0) +(0,-0.4*\l) node[draw=none,fill=none,color=black] {$\b_1$};
\draw (\l*11,0) +(0,-0.4*\l) node[draw=none,fill=none,color=black] {$\b_2$};
\draw (\l*0,0) +(120:0.5*\l) -- (\l*0,0);
\draw (\l*0,0) +(240:0.5*\l) -- (\l*0,0);
\draw (\l*3,0) +(120:0.5*\l) -- (\l*3,0);
\draw (\l*3,0) +(60:0.5*\l) -- (\l*3,0);
\draw (\l*7,0) +(120:0.5*\l) -- (\l*7,0);
\draw (\l*8,0) +(60:0.5*\l) -- (\l*8,0);
\draw (\l*12,0) +(60:0.5*\l) -- (\l*12,0);
\draw (\l*12,0) +(-60:0.5*\l) -- (\l*12,0);
\begin{pgfonlayer}{background}
  \foreach \x  in {0,1,2,3,6,7,8,9} {
    \draw[-{Latex}] (\l*\x,0) -- (\l*\x+\l*0.65,0);
    \draw (\l*\x+\l*0.5,0) -- (\l*\x+\l*1,0);
  }
\foreach \x  in {4,5,10,11} {
    \draw[-{Latex}] (\l*\x+\l,0) -- (\l*\x+\l*0.35,0);
    \draw (\l*\x+\l*0.4,0) -- (\l*\x,0);
  }
\draw[{Bar[width=0pt 6]}-{Latex}] (0,+.8*\l) -- (6*\l,+0.8*\l);
\draw[{Bar[width=0pt 6]}-{Latex}] (6*\l+.05,+.8*\l) -- (12*\l,+0.8*\l);
\draw[{Bar[width=0pt 6]}-{Latex}] (1*\l,-.8*\l) -- (5*\l,-0.8*\l);
\draw[{Bar[width=0pt 6]}-{Latex}] (5*\l+.05,-.8*\l) -- (8*\l,-0.8*\l);
\draw[{Bar[width=0pt 6]}-{Latex}] (8*\l+.05,-.8*\l) -- (11*\l,-0.8*\l);
\end{pgfonlayer}
\draw (\l*3,+1.2*\l) node[draw=none,fill=none,color=black] {$s_0$};
\draw (\l*9,+1.2*\l) node[draw=none,fill=none,color=black] {$s_0$};
\draw (\l*6,-1.6*\l) node[draw=none,fill=none,color=black] {$\D$};
\draw (\l*3,-.6*\l) node[draw=none,fill=none,color=black] {$a$};
\draw (\l*9.5,-.6*\l) node[draw=none,fill=none,color=black] {$a$};
\draw (\l*6.5,-.6*\l) node[draw=none,fill=none,color=black] {$b$};
 %%%%%%%%%%%%%
  %%%%%%%%%%%%%
  \draw[{Latex}-{Latex}, ultra thick,color=black]  (6*\l,-3.25*\l) -- (6*\l,-2.5*\l);
  %%%%%%%%%%%%
  %%%%%%%%%%%%
\begin{scope}[shift={(0,-5.25*\l)}]
\draw (\l*0,0) node[color=black] {};
\draw (\l*1,0) node {};
\draw (\l*2,0) node[color=black] {};
\draw (\l*3,0) node[color=black, minimum width=3pt] {};
\draw (\l*4,0) node[color=black] {};
\draw (\l*5,0) node[color=red] {};
\draw (\l*6,0) node[color=black] {};
\draw (\l*7,0) node[color=black, minimum width=3pt] {};
\draw (\l*8,0) node {};
\draw (\l*9,0) node[color=black] {};
\draw (\l*10,0) node[color=black] {};
\draw (\l*11,0) node[color=red] {};
\draw (\l*12,0) node[color=black] {};
\draw (\l*0,0) +(0.75,0.4) node[draw=none,fill=none,color=black] {$\bar g_0$};
\draw (\l*2,0) +(-0.65,0.4)  node[draw=none,fill=none,color=black] {$t$};
\draw (\l*3,0)  +(0,1) node[draw=none,fill=none,color=black] {$\bar g_1$};
\draw (\l*7,0)  +(0,1) node[draw=none,fill=none,color=black] {$\bar g_0$};
%\draw (\l*4,0) +(-0.65,0.4)  node[draw=none,fill=none,color=black] {$g''_1$};
\draw[{Bar[width=0pt 6]}-{Latex}] (2*\l,+.4*\l) -- (4*\l,+0.4*\l);
\draw[{Bar[width=0pt 6]}-{Latex}] (6*\l,+.4*\l) -- (8*\l,+0.4*\l);
\draw (\l*5,0) +(-0.65,0.4)  node[draw=none,fill=none,color=black] {$t$};
\draw (\l*6,0) +(-0.65,0.4)  node[draw=none,fill=none,color=black] {$t$};
%\draw (\l*7,0) +(-0.65,0.4)  node[draw=none,fill=none,color=black] {$g'_0$};
%\draw (\l*8,0) +(-0.65,0.4)  node[draw=none,fill=none,color=black] {$g''_0$};
\draw (\l*9,0) +(-0.65,0.4)  node[draw=none,fill=none,color=black] {$t$};
\draw (\l*10,0) +(-0.65,0.4) node[draw=none,fill=none,color=black] {$\bar g_1$};
\draw (\l*11,0) +(-0.65,0.4) node[draw=none,fill=none,color=black] {$t$};
\draw (\l*12,0) +(-0.65,0.4) node[draw=none,fill=none,color=black] {$t$};
\draw (\l*1,0) +(0,-0.4*\l) node[draw=none,fill=none,color=black] {$\a_1$};
\draw (\l*8,0) +(0,-0.4*\l) node[draw=none,fill=none,color=black] {$\a_2$};
\draw (\l*5,0) +(0,-0.4*\l) node[draw=none,fill=none,color=black] {$\b_1$};
\draw (\l*11,0) +(0,-0.4*\l) node[draw=none,fill=none,color=black] {$\b_2$};
\draw (\l*0,0) +(120:0.5*\l) -- (\l*0,0);
\draw (\l*0,0) +(240:0.5*\l) -- (\l*0,0);
%\draw (\l*3,0) +(120:0.5*\l) -- (\l*3,0);
%\draw (\l*3,0) +(60:0.5*\l) -- (\l*3,0);
%\draw (\l*7,0) +(120:0.5*\l) -- (\l*7,0);
%\draw (\l*8,0) +(60:0.5*\l) -- (\l*8,0);
\draw (\l*12,0) +(60:0.5*\l) -- (\l*12,0);
\draw (\l*12,0) +(-60:0.5*\l) -- (\l*12,0);
\begin{pgfonlayer}{background}
  \foreach \x  in {0,1,2,3,6,7,8,9} {
    \draw[-{Latex}] (\l*\x,0) -- (\l*\x+\l*0.65,0);
    \draw (\l*\x+\l*0.5,0) -- (\l*\x+\l*1,0);
  }
\foreach \x  in {4,5,10,11} {
    \draw[-{Latex}] (\l*\x+\l,0) -- (\l*\x+\l*0.35,0);
    \draw (\l*\x+\l*0.4,0) -- (\l*\x,0);
  }
%\draw[{Bar[width=0pt 6]}-{Latex}] (0,+.8*\l) -- (6*\l,+0.8*\l);
%\draw[{Bar[width=0pt 6]}-{Latex}] (6*\l+.05,+.8*\l) -- (12*\l,+0.8*\l);
\draw[{Bar[width=0pt 6]}-{Latex}] (1*\l,-.8*\l) -- (5*\l,-0.8*\l);
\draw[{Bar[width=0pt 6]}-{Latex}] (5*\l+.05,-.8*\l) -- (8*\l,-0.8*\l);
\draw[{Bar[width=0pt 6]}-{Latex}] (8*\l+.05,-.8*\l) -- (11*\l,-0.8*\l);
\end{pgfonlayer}
%\draw (\l*3,+1.2*\l) node[draw=none,fill=none,color=black] {$s$};
%\draw (\l*9,+1.2*\l) node[draw=none,fill=none,color=black] {$s$};
\draw (\l*6,-1.6*\l) node[draw=none,fill=none,color=black] {$\D$};
\draw (\l*3,-.6*\l) node[draw=none,fill=none,color=black] {$\bar a$};
\draw (\l*9.5,-.6*\l) node[draw=none,fill=none,color=black] {$\bar a$};
\draw (\l*6.5,-.6*\l) node[draw=none,fill=none,color=black] {$\bar b$};
\end{scope}

\end{tikzpicture}
\end{center}
    \caption{Labelling with $\bar\phi$ to define $\sep(\D)$}\label{fig:sepD}
  \end{figure}
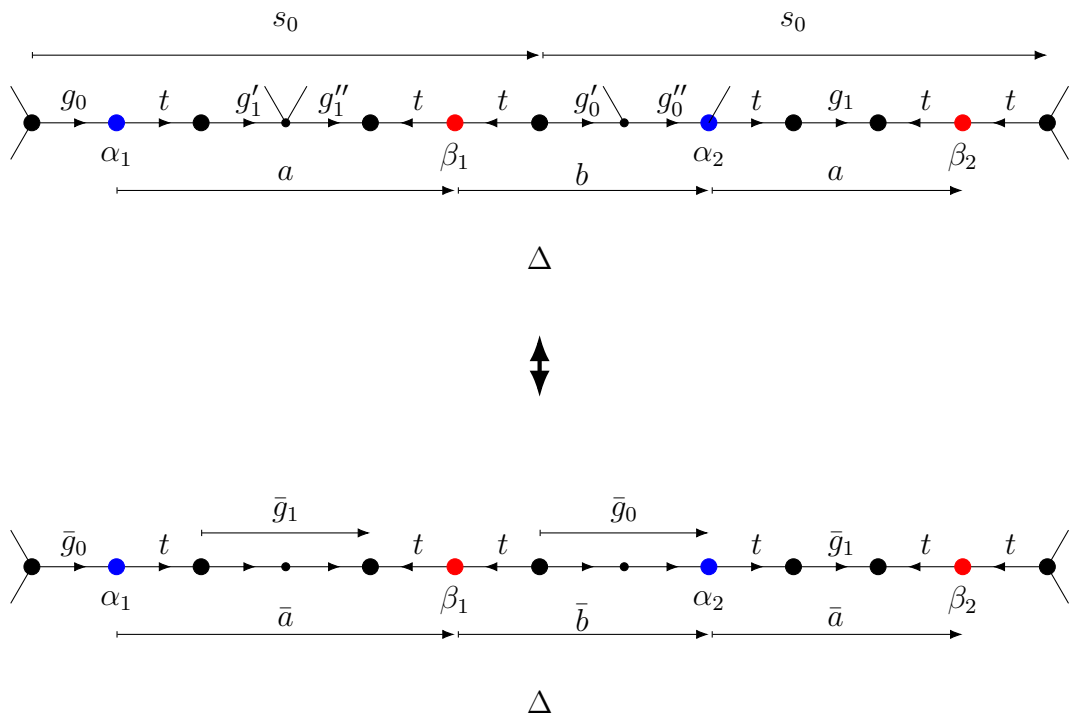
\end{example}
 We define $\sep(\D)$ to be 
%the set of intervals \[\sep(D)=\{[\a_i,\b_i]\,:\, 1\le i\le m\}\cup \{[\b_{i},\a_{i+1}]\,:\,1\le i\le m\}\]
\[\sep(\D)=\{\a_i,\b_i\,:\, 1\le i\le n\}.\]%\cup \{[\b_{i},\a_{i+1}]\,:\,1\le i\le m\}\]
% (subscripts modulo $m$).

 If $\mu$ is a simple cyclic sub-path of $\rho$ then we may assume that 
$\mu$ is a path from points $\a$ to $\b$, on $\pd \D$ (when read with the same orientation as $\rho$).  We define %$\sep(\mu)=\{I\in \sep(\D)\,:\, I\subseteq \mu\}$.
 \[\sep(\mu)=\sep(\D)\cap [\mu\backslash \{\b\}].\] 
(The terminal point of $\mu$ never contributes to $\sep(\mu)$.) 
Thus, if $\rho$ has a cyclic subpath $\xi$ with  decomposition $\xi=\mu\nu$, where $\mu$ and  $\nu$ have 
 disjoint interiors, then $|\sep(\xi)|=|\sep(\mu)|+|\sep(\nu)|$.
\begin{lemma}\label{lem:piecesep}
Let $s$ and $r$ be as in Lemma \ref{lem:perpospiece}. % and let 
 %$p$ be a piece over $\widetilde R$. %such that $r_1=pq_1$ and $r_2=pq_2$ are distinct elements of $\tilde R$ and 
%either 
% $p$ is a right integral cyclic subword of $r_1$. %or $r_1$ 
% and $r_2$. % are not in periodic position.  
 Assume $\D$ is a region of a strongly reduced $\widetilde \cR$ diagram over $F$, which has  boundary cycle $\rho$, with a  sub-interval $\mu$ 
 such that %$\phi(\rho)=pq_1$ and
  $\phi(\mu)=p$, where $p$ is a piece over $\widetilde \cR$. 
 Then %either %
 $|\sep(\mu)|\le 1$. % and if $|\sep(\mu)|=2$ then $\sep(\mu)=\{\a_1,\a_2\}$, where
 %$\a_1$ is the initial vertex of $\mu$ and $\a_2$ is  an interior vertex of $\mu$. 
%\be
%\item $|\sep(\mu)|\le 1$, or 
%\item $|\sep(\mu)|=2$, $p=ut^\a$, for some $u\in U$ and $\a\in\ZZ$,  $\mu$ (or its reverse) contains one of intervals $I=[\a_i,\b_i]$ or $I=[\b_i,\a_{i+1}]$ and 
%$\bar\phi(I)=t^{\pm \a}$.  
%(subscripts modulo $m$). 
%\ee
\end{lemma}
\begin{proof}
%As $p$ is a piece over $R$ there are distinct elements $r_1$ and $r_2$ of $R$ such that  $r_1=pq_1$ and $r_2=pq_2$. 
%Let $D$ be disc with boundary label $r_1$, which we may assume is a conjugate of a  cyclic permutation of $r$,  
%and let $\bar a\bar b$ be a factorisation of a cyclic permutation of $\bar s$, as above, such that $\s(r_1)$ is a 
%cyclic permutation of $(\bar a\bar b)^n$. 
  We may assume that  $p\notin H$ and $p\notin\{t^{\pm 1}\}$, since otherwise the result holds immediately.
Let $p_0$ be the maximal integral subword of $p$ (which must be non-empty under this assumption).
Then $p_0$ is a piece over $\widetilde \cR$, and as the diagram is strongly reduced, $p_0$ is not in periodic position,
so $\bar p_0$ is a piece over the symmetrised closure of $\s(\widetilde \cR)$. 
As $\bar a$ and $\bar b$ are
uniquely positioned subwords of $\bar s$, neither can be a subword of $\bar p_0$. 
As $a$ and  $b$ are integral subwords of the boundary label of $\D$,  by definition of $p_0$, if  $a$ or $b$ is a subword of
 $p$ then it is a subword of $p_0$; so  $\bar a$ or $\bar b$  is a subword of $\bar p_0$, a contradiction. 
Thus neither $a$ or $b$ is a subword of $p$, and the Lemma follows.
\end{proof}
Following \cite{mccawise} we make the following definition.
\begin{defn}\label{def:shell}
  Let  $\D$ be a region of a diagram $M$ over $\widetilde \cR$
   and suppose that $\D$ has a boundary cycle that 
  factors as $\mu_1\cdots \mu_k\b$, where $\mu_i$ is the common boundary component of
  $\D$ and a region $\D_i$, and $\b$ is a component of a boundary cycle of $M$. % (or $\b$ is empty).
  Then $\D$ is called a $k$-shell, with boundary component $\b$. 
\end{defn}

\begin{prop}\label{prop:sc}
  Let $s$, $r=s^n$ and $\cR$ be as in Lemma \ref{lem:perpospiece}.
     \be[label=(\alph*)]
\item\label{it:sc1}
  If 
   $\D$ is a  $k$-shell
  of a strongly reduced
   diagram over $\widetilde \cR$, %strongly reduced
     with boundary component  $\b$ then 
  \begin{equation}\label{eq:dshlth} l(\phi(\b))
    \ge \begin{cases}
      \left(n-  \frac{k+2}{2}\right)l(s)+2,\textrm{ if $k$ is even}\\
      \left(n-\frac{k+1}{2}\right)l(s)+1, \textrm{ if $k$ is odd}
    \end{cases}
    .
\end{equation}
\item\label{it:sc2}  $\widetilde \cR$ satisfies
  small cancellation condition $C(2n)$.
  \ee
\end{prop}
We remark that this theorem applies to diagrams over $\widetilde \cR$, and not to diagrams over the free presentation for $G$. Indeed,
$G$ may contain subgroups isomorphic to $\FF_2\times \FF_2$, in which case, as shown by Bigdely and Wise \cite{BW}, no free
 presentation for $G$
satisfies $C(6)$. 
\begin{proof}
\be[label=(\alph*)]
\item
  Let $\D$ have boundary cycle $\rho$ with decomposition $\rho=\mu_1\cdots \mu_k\b$, as in the Definition of $k$-shell.
  From Corollary \ref{cor:noperpos}, $\phi(\mu_i)=p_i$ is a piece over $\widetilde \cR$, for $1\le i\le k$. Therefore Lemma \ref{lem:piecesep}
  implies that 
  $\sep(\mu_1\cdots \mu_k)\le k$. As $\sep(\D)=2n$, this gives $\sep(\b)\ge 2n-k$. We claim that if $\nu$ is a subpath of $\rho$
  with $\sep(\nu)\ge K$ then $l(\phi(\nu))\ge ((K-2)/2)l(s)+2$, if $K$ is even, and that $l(\phi(\nu))\ge ((K-1)/2)l(s)+1$,
  if $K$ is odd. This certainly holds if $K=0, 1$ or $2$, since one of $a$, $b$ may have length $1$. An elementary
  induction then shows that the claim holds for all $K\ge 0$.  From the claim, with $K=2n-k$, the first part of the Lemma follows. 
\item
  Let $w$ be an element of $\widetilde \cR$ with a $k$-piece factorisation $w=p_1\cdots p_k$. 
  First we note that, 
  if $p$ is a piece over $\widetilde \cR$, and $v=pq\in \widetilde \cR$, where  $v=u^{-1}v_0u$, for some cyclically minimal element $v_0$ and $u\in U$, then
  $v_0=upu^{-1}(uqu^{-1})$ and it follows that $u^{-1}pu$ is also a piece over $\widetilde \cR$. Hence we may assume that $w$ is cyclically minimal
  and we may form a diagram with a single cell $\D$ and boundary cycle $\rho$ such that $\rho=\mu_1\cdots \mu_k$, where $\phi(\mu_i)=p_i$.  
  Then $2n=\sep(\D)=\sum_{i=1}^k\sep(\mu_i)\le k$, proving the second part of the lemma.
\ee
\end{proof}
\section{Proofs of main theorems}%Theorem \ref{thm:main}}
\label{sec:main}
%In this section  we prove the main result of the paper, namely the following theorem. 

\begin{proof}[Proof of Theorem \ref{thm:main}]
Let $r=s^n$, $\cR=\{r\}$ and $\GG_t=\la A\bs\{t\}\ra$. If $\lk(t)$ is empty then $\GG=\GG_t* \la t\ra$ and Theorem \ref{thm:main} follows from
well known results on one-relator products. We may therefore assume $\lk(t)$ is a non-empty clique. 
\be[(a)]
\item Let $w\in ((A\bs\{t\})^{\pm 1})^*$ be  a freely reduced word such that $w=1$ in $G$. 
 From Proposition \ref{prop:Fcancel}, there exists a strongly reduced $\widetilde \cR$-diagram  $M$ over $F$, 
 with boundary label $w$.  
From Proposition \ref{prop:sc},  $\widetilde \cR$ satisfies $C(2n)$, where $n\ge 3$. Therefore 
Greendlinger's Lemma for $C(6)$ diagrams 
(see for example \cite[Theorem 9.4]{mccawise}) implies that $M$ has
a $k$-shell, for $k\le 3$.  
%a region $\D$ such that the boundary 
%cycle of $\D$ is $\a\b$, where $\b\subseteq \pd M$, $\a$ has no boundary arc, and the label $\phi(\a)$ is a product of 
%at most three pieces.    From Lemma \ref{lem:piecesep} the product of three pieces has length no more than 
%than $2|s|$.
Hence $\phi(\b)$ has length at least $(n-2)l(s)\ge l(s)$, so is a word of positive $t$-length, contrary to the 
assumption that $w\in \GG_t$. 
\item 
  % Using Corollary \ref{cor:Fcancel},
  Let $M$ be a strongly reduced $\widetilde \cR$-diagram, 
  with boundary label $s^k$, for some positive integer $k<n$.   
%As in \cite{W}, call a region $\D$ and $d$-shell if the boundary cycle of $\D$ factors as $\rho_1 \cdots \rho_d \b$, where 
%$\rho_i$ is the common boundary path of $\D$ and another region $\D_i$, and $\b\subseteq \pd M$.
  As $\pd M$ has length $kl(s)$,  
  $M$ must contain at least $2$ regions and since $s^k$ is freely reduced,
  %from Lemma \ref{lem:piecesep}, no piece has length more than 
  %$|s|-1$, $M$ does not contain any $1$-shells. F
  from Greendlinger's Lemma for $C(6)$ diagrams again, $M$ must contain
  at least three $d$-shells, where $d\le 3$,
  %The boundary arc of a $d$-shell has length at least $(n-2)|s|$, as in the 
  %previous part of the proof,
  so $|\pd M|\ge 3(n-2)l(s)$. As $n\ge 3$, $3(n-2)\ge n$, so this implies $kl(s)=|\pd M|\ge nl(s)$, a contradiction.
  \ee
For the final statement 
% \ajd{(restricted for the time being to the initial case of the word problem DP$(0)$)}
 let $w\in \FF(A)$.
We may assume without loss of generality that $w$ is a minimal form for a non-trivial element of $F$. 
If $w=1$ in $F/N$ then
there exist elements %$R_1,\ldots ,R_k\in \widetilde R$ and
$y_1,\ldots y_k\in \FF(A)$ and $\e_1,\ldots, e_k\in \{\pm 1\}$, such that
\begin{equation}\label{eq:wp}
  w=\prod_{i=1}^ky_i^{-1}r^{\e_i}y_i.
\end{equation}
From this
expression we construct, in the usual way, a diagram  over $\widetilde \cR$, with  boundary label $w$ and
regions $\D_i$, $i=1,\ldots ,k$. After deleting
cancelling regions if necessary we obtain a reduced  diagram $M$ with at most $k$ regions and boundary label $w$.
From Proposition \ref{prop:Fcancel} we may assume $M$ is strongly reduced. 
As $n\ge 4$, from
Proposition \ref{prop:sc}, $\widetilde \cR$ satisfies $C(8)$ and so from  standard small cancellation
arguments (e.g. \cite[page 246, Proposition 27]{GdH}) the number of regions of $M$
is at most $8l(w)$.
 % From Lemma \ref{lem:Rfin} there are finitely many cyclically minimal elements of $\widetilde R$ ending in a $t$-letter, so
To decide whether or not $w=1$ in $F/N$ therefore amounts
to deciding whether or not one of equations   
\eqref{eq:wp}, in variables $y_i$,  with $k\le 8l(w)$ has solution. From \cite{DM} equations are decidable over partially commutative groups, so the latter problem
is decidable.

\end{proof}
\begin{proof}[Proof of Theorem \ref{thm:amalgam}]
  As above, %in the preamble to Theorem \ref{thm:amalgam},
  we may decompose $\GG$ as a free product with amalgamation, 
\[\GG=\AA_0\ast_U\AA_1.\]
By definition of $\lk(s)$, we have $\AA_0= U \times K$  so 
\[G=\GG/N\cong(U\times K/M)\ast_U \AA_1.\]
If $\la Y'\ra$ embeds in $K/M$ it follows that $\la \lk(s)\cup Y'\ra$ embeds in $U\times K/M$, so $\la (A\bs\supp(s)) \cup Y'\ra$ embeds
in $(U\times K/M)\ast_U \AA_1$, as required. 
Moreover, if $s$ has order $n$ in $K/M$, it follows that $s$ also has order $n$ in $G$.

Given an element $w\in \FF(A)$ we may write $w$ in the form $w=g_1u_1h_1\cdots g_ku_kh_ku_{k+1}=[\prod_{i=1}^k (g_iu_ih_i)]u_{k+1}$,
where $k\ge 0$, $u_i\in U$, $g_i\in K$ and $h_i\in \AA_1$;
using the Transformation Lemma (see Section \ref{sec:prelim}), and the fact that $[K,U]=1$. Moreover, we may assume $w$ is a minimal length word in $(A\cup A^{-1})^*$ representing 
its class as an element of $\GG$, and that $h_i$ has no left or right divisor in $U$, for all $i$.   Then $w\in N$ if and only if $g_i\in M$, for $i=1,\ldots, k$,
and $u_1h_1\cdots u_kh_ku_{k+1}=1$ in $\AA_1$. 
 If the word problem is solvable in $K/M$ we may decide whether or not $g_i\in M$; and so   the word problem is decidable in $G$. 

 Given elements $v,w\in \FF(A)$, to decide whether or not $v$ and $w$ are conjugate in $G$, we apply \cite[Theorem 4.6, page 212]{MKS}.   We may first
 replace $v$ and $w$ by elements of $\FF(A)$ representing cyclically minimal elements of $\GG$. To simplify notation let  $\AA_2=U\times K/M$. 
 If $v\in U$,  and $v$ is conjugate to $w$, then $w$ represents an element of $\AA_1\cup \AA_2$  and there exists a sequence $v,v_1,\ldots v_l=w$, where $v_i\in U$, for $i<l$ and consecutive terms
 are conjugate in $\AA_j$, for $j=1$ or $2$. As $\AA_2=U\times K/M$, this implies that $v$ is conjugate to $w$ in $\AA_1$, and this may
 be effectively verified, in the partially commutative group $\AA_1$. If $v\notin U$ but $v\in \AA_1\cup \AA_2$ then $v$ and $w$ are conjugate
 only if they belong to the
 same factor and  are conjugate in that factor. If $v$ and $w$ belong to $\AA_1$ we may verify if they are conjugate or not, as in the previous case.
 If both belong to $\AA_2$ then we use the solvability of the conjugacy problem in $K/M$ to decide whether or not they are conjugate.

 Otherwise 
   $v$ and $w$ are
 cyclically minimal as elements of  $\AA_2*_U\AA_1$ which we may write (after a cyclic permutation if necessary) 
 $w=g_1u_1h_1\cdots g_ku_kh_ku_{k+1}$ and $v= g'_1u'_1h'_1\cdots g'_lu'_lh'_lu'_{m+1}$,  where $k>0$, $u_i,u_j'\in U$, $g_i,g_j'\in K$, $h_i,h_j'\in \AA_1$ and
 $h_i$, $h_j'$ have no left or right divisors in $U$, for all $i,j$. We may
 identify any $i$ such that 
 $g_i=_G1$ and rewrite $w$, replacing $g_i$ by $1$, using the fact that the conjugacy (so word) problem is solvable in $K/M$. After repeating this process sufficiently
 often we may assume that $g_i\notin M$, for all $i$; without altering the conjugacy class of $w$ in $G$. 
 Similarly we may assume no $g_i'\in M$.  Note that $w$ is cyclically reduced as an element of  $(U\times K/M)\ast_U \AA_1$ if and only
 the first and last letters come from distinct factors (that is $g_1\neq 1$ and $h_k\neq 1$) and we may assume $w$ begins with an element
 of $K/M$, by cyclically permuting if necessary. We may therefore decide whether or not $w$ is cyclically reduced
 and, if it is not, replace it by a cyclic permutation with fewer factors. Hence we may assume that $w$ and $v$ have been rewritten as representatives of
 cyclically reduced forms (in the sense of free products with amalgamation) of elements of $\AA_2\ast_U \AA_1$, both beginning with an element of $K$.
 Then $v$ is conjugate to $w$ if and only if a cyclic permutation of $v$, followed by conjugation by an element of $U$, results in a representative of $w$.
 We may assume that $v$ as written above is an appropriate cyclic permutation and it remains to decide if $u^{-1}vu=w$, for some $u\in U$.

 First consider the case where $u=1$ and $v=w$. Then we have
 \begin{equation}\label{eq:conjeq}
   g_1u_1h_1\cdots g_ku_kh_ku_{k+1}= g'_1u'_1h'_1\cdots g'_lu'_lh'_lu'_{l+1}.
 \end{equation}
 We claim that this holds if and only if $k=l$, $g_i=g_i'$, $h_i=h_i'$, $[h_i,u'_{i+1}\cdots u'_{k+1}u^{-1}_{k+1}\cdots u_{i+1}^{-1}]=1$, for $i=1,\ldots ,k$
 and $u'_{1}\cdots u'_{k+1}=u_1\cdots u_{k+1}$.
 To see this, first observe that if \eqref{eq:conjeq} holds then
 we have $h_l'u'_{l+1}u_{k+1}^{-1}h_{k}^{-1}\in U$, and since $h_k$ and $h_l'$ have no left or right divisors
 in $U$, this forces $h_l'=h_k$ and $[h_k,u'_{l+1}u_{k+1}^{-1}]=1$. Therefore
 \[g_1u_1h_1\cdots h_{k-1}u_ku_{k+1}g_k= g'_1u'_1h'_1\cdots h_{l-1}u'_lu'_{l+1}g_l',\]
 from which we have $g_l'g_k^{-1}\in U$, so $g_l'=g_k$. Continuing this way, it follows that
 $k=l$, $g_i=g_i'$ and $[h_i,u'_{i+1}\cdots u'_{k+1}u^{-1}_{k+1}\cdots u_{i+1}^{-1}]=1$, for $i=1,\ldots ,k$. This in turn, together with \eqref{eq:conjeq}, implies that
 $u'_{1}\cdots u'_{k+1}=u_1\cdots u_{k+1}$. Conversely, if all these conditions hold, then so does \eqref{eq:conjeq}, by direct computation.  

 In the general case, $u^{-1}vu=w$ if and only if
 \begin{equation}\label{eq:gencon}
   g_1u_1h_1\cdots g_ku_kh_ku_{k+1}= g'_1u^{-1}u'_1h'_1\cdots g'_lu'_lh'_lu'_{l+1}u.
 \end{equation}
 From the claim above this holds if and only if $k=l$,
 \[g_i=g_i',\, h_i=h_i',\, [h_i,u'_{i+1}\cdots u'_{k+1}u u^{-1}_{k+1}\cdots u_{i+1}^{-1}]=1,\]
 for $i=1,\ldots ,k$
 and \[u^{-1}u'_{1}\cdots u'_{k+1}u=u_1\cdots u_{k+1}.\]
 That is, $u^{-1}vu=w$ if and only if  $k=l$, $g_i=g_i'$,  $h_i=h_i'$ and the system of $k+1$ equations
 \begin{align}
   x^{-1}(u'_{i+1}\cdots u'_{k+1})^{-1} h_i u'_{i+1}\cdots u'_{k+1}x&=(u_{i+1}\cdots u_{k+1})^{-1} h_i u_{i+1}\cdots u_{k+1},\notag\\
   x^{-1}u'_{1}\cdots u'_{k+1}x&=u_1\cdots u_{k+1},\label{al:conjeqs}
 \end{align}
 in the variable $x$, where $1\le i \le k$, has a solution in the group $U$. 
 In the terminology of \cite{DM}, $U$ is a normalised rational subset of the partially commutative
 group $\AA_1$. Thus  \eqref{al:conjeqs} is a system of equations over $\AA_1$ with normalised rational
 constraints. From \cite[Corollary 1]{DM}, the system  \eqref{al:conjeqs} is decidable. Therefore we
 may decide whether or not $v$ is conjugate to $w$ by an element of $U$.
 As $v$ has only finitely many cyclic permutations, combining the above we have a solution to the conjugacy problem in $G$.

 The final part of this proof is unsatisfactory, in that it uses the decidability of all systems of equations over partially commutative groups, to lift
 to decidability of the conjugacy problem in $G$. We therefore give an alternative argument, which depends only on decidability
 of the conjugacy problem.\footnote{The authors are grateful to Armin Weiss for pointing out this approach.}  First let $a\in A\cap K$, such that $a\notin M$. Then, replacing $g_i$ and  $g_i'$ by $a$, for all $i$ 
 in  \eqref{eq:gencon}, the argument above shows that
\begin{equation}\label{eq:alias}
a u_1h_1\cdots a u_kh_ku_{k+1}= u^{-1}a u'_1h'_1\cdots a u'_lh'_lu'_{l+1}u,
\end{equation}
for some $u\in U$, 
if and only if $k=l$, $h_i=h_i'$ and the system of equations \eqref{al:conjeqs} has a solution $x\in U$.
%That is, if and only if
%\begin{equation}\label{eq:alias2}
%a u_1h_1\cdots a u_kh_ku_{k+1}= u^{-1}a u'_1h_1\cdots a u'_kh_ku'_{k+1}u,
%\end{equation}
%for some $u\in U$. 
Since $a\notin U$ and $h_i$ has no left or right divisor in $U$ the subgraph of the non-commutation graph $\D$ of $\GG$
with vertices $\{a\}\cup\bigcup_{i=1}^k\supp(h_i)$ is connected. Therefore $\hat w=a u_1h_1\cdots a u_kh_ku_{k+1}$ has block decomposition 
$\hat w=b_0b_1\cdots b_p$, where the $b_i$ are blocks,
$b_0=a v_1h_1\cdots a v_kh_kv_{k+1}$, for some $v_i\in U$, and $b_i\in U$, for $i\ge 1$. If \eqref{eq:alias} holds,
then it follows from \cite[Proposition 5.7]{EKR}
that its right hand side, $\hat v=a u'_1h_1\cdots a u'_kh_ku'_{k+1}$, %of \eqref{eq:alias}
 has block decomposition $\hat v=b'_0b'_1\cdots b'_p$, where the $b'_i$ are blocks,
 $b'_0=a v'_1h_1\cdots a v'_kh_kv'_{k+1}$, for some $v'_i\in U$ and $b'_i\in U$, for $i\ge 1$. Moreover, in this case,
 after reordering the $b_i's$, $i\ge 1$, if necessary, the blocks 
$b_i$ and $b_i'$ are cyclically minimal and conjugate. 
% as
% $w$ and $w'$ are cyclically minimal, so are the $b_i$ and  $b_i'$, % and %.
From the Transformation Lemma, $b_0$ is conjugate to $b_0'$ if and only if there is an element
$z_0\in U\cap \la\supp(b_0)\ra$ such that  $z_0^{-1}b_0'z_0=b_0$, and for $i\ge 1$, $b_i$ is conjugate to $b_i'$
if and only if there exist $z_i$ such that $z_i\in \la\supp(b_i)\ra$, and  $z_i^{-1}b_iz_i=b_i'$, for $i=1,\ldots ,p$.
As $[\supp(b_i),\supp(b_j)]=1$, for all $i\neq j$,  it
follows that there exists $u\in U$ such that $u^{-1}\hat wu=\hat v$ if and only if there exist such $z_i$. 
\begin{comment}
 there is
 a conjugator $v$ such that $v\in \la\supp(\hat v)\ra$ and
 % As $w$ and $w'$ are cyclically minimal,
 $\supp(b_i)=\supp(b_i')$, for all $i$.
Also, from the proof of \cite{???}, $b_0$ is conjugate to $b_0'$ by an element of $U$ if and only if
$b_0$ is conjugate to $b_0'$ by an element of $U\cap \la\supp(b_0)\ra$. As $[\supp(b_i),\supp(b_j)]=1$, for all $i\neq j$,  it
follows that there exists $u\in U$ such that $u^{-1}\hat wu=\hat v$ if and only if there exist $z_i$ such that $z_i\in \la\supp(b_i)\ra$,
$z_0\in U\cap \la\supp(b_0)\ra$ and $z_i^{-1}b_iz_i=b_i'$, for $i=1,\ldots ,p$.
\end{comment}
For $i\ge 1$ the question of whether or not such a $z_i$ exists is the conjugacy problem in the partially commutative
group $\la \supp(b_i)\ra$, so is decidable. For $i=0$, the question of existence of such a $z_0$ is decidable by
\cite[Proposition 5.8]{EKR}. Therefore, we may decide whether or not $\hat w$ is conjugate to $\hat v$. This means
that we may also decide whether or not \eqref{al:conjeqs}  has a solution  $x\in U$; and applying this argument
to all cyclic permutations of $v$, the result follows.
 \end{proof}
\section{Cycle graphs with a chord}\label{sec:almostall}

In this section we argue that Theorem \ref{thm:main} almost always applies, and a Freiheitssatz holds, in the situation of
Example \ref{ex:thm_main}.\ref{it:pentchord}. %; at least for words of non-zero $t$-length.
We conjecture that a similar statement holds for arbitrary graphs. Here though we use only naive counting arguments;
a study of the conjecture over arbitrary graphs is likely to require a more systematic approach. 

As in Example \ref{ex:thm_main}.\ref{it:pentchord}, $\GG'=\GG(C'_n)$, where $C_n'$ is the graph on the left of
Figure \ref{fig:pentchord}, and $n\ge 5$. The subgroup $H$ of $\GG'$ generated by $A_0=\{a_1,\ldots, a_{n-1}\}$ is isomorphic to $\GG(C_{n-1})$, 
$\lk(t)=\{a_1, a_{n-1}\}$, $U=\la \lk(t)\ra$ and HNN$(t)=\la H, t\,|\, t^{-1}at=a, \forall a\in \lk(t)\ra$.
As input to the question ``for which elements of $\GG'$ does Theorem \ref{thm:main} hold?'' we choose a particular set
$L$ (see Section \ref{sec:nf}) of words over $(A\cup A^{-1})^*$ with the property that every cyclically minimal element of $H$ is represented by a unique element of $L$,
every cyclically minimal element of $\GG'$ ending in a $t$-letter is represented by a unique element of $L$
and every element of $L$ represents a cyclically minimal element of $\GG'$. 
Every element of $L$ is reduced with respect to HNN$(t)$ and is written as 
  $w=g_1t^{\a_1}g_2\cdots g_mt^{\a_m}$, where  $m\ge 0$, $g_i\in H$, $\e_i=\pm 1$,  and $\a_i\in \ZZ\bs\{0\}$. The set of  elements $w\in L$  for which 
  $|\a_1|+\cdots +|\a_m|\le k$ and $l(g_i)\le d$, for all $i$, is denoted by $L(d,k)$.
We show that almost all words of $L$ satisfy the  hypotheses of Theorem  \ref{thm:main}, in a sense which we now make precise. 
For a subset $S$ of $L$ we define the \emph{asymptotic density} of $S$ to be
\[\rho(S)=\lim_{d,k\maps \infty} \frac{|S\cap L(d,k)|}{|L(d,k)|}.\]
(In this limit, both $d$ and $k$ must become, simultaneously, sufficiently large. The asymptotic density of $S$
is undefined if the limit does not exist.) 
For general discussion of asymptotic density and it properties we refer, for example, to \cite{MSU, KMSS, FMR}: 
the 2-parameter ``bidimensional'' asymptotic density defined in \cite{FMR} plays a role equivalent to the
limits as both $d$ and $k$ approach infinity, defined here.

As subset $S\subseteq L$ is said to be \emph{generic} if $\rho(S)=1$ and \emph{negligible} if $\rho(S)=0$.
\begin{prop}\label{prop:ad}
  In the above terminology, let $L_Y$ be the subset of $L$ consisting of words which satisfy the
  hypotheses of Theorem \ref{thm:main}. Then $L_Y$ is generic.
\end{prop}
%A notion of  asymptotic density of subsets of amalgams involving 2-parameters is used in \cite{FMR}, tho
In the remainder of this section we define  the set $L$  and prove this proposition.

%Elements of $L$ all end in a $t$-letter and have $H$-syllables
%of a form which we now define.  
%We begin by indentifying a set $L$ of cyclically reduced elements  of $\GG'$.
% We shall show that the subset of $L$ for which the conditions of the theorem hold is generic, relative to $L$.
\subsection{Normal forms}\label{sec:nf}

The definition of  $L$ depends on a choice of normal forms for elements of $H$. The subgroup $H$ is $\GG(C_{n-1})$ where $C_{n-1}$ is the
cycle graph with vertices $A_0$.
 We may define a set $L_H$ of unique normal forms for all $n\ge 5$, as we do below, but observe that 
the case $n=5$ is special, as in this case $H=\FF(a_1,a_3)\times \FF(a_2,a_4)$, and an alternative normal form may
be chosen using this decomposition. It turns out that some of the bounds we derive for our general
normal form involve division by $n-5$, so would need special treatment when $n=5$. Moreover, the alternative
normal form for the case $n=5$ allows relatively straightforward computation of the precise sizes of
sets we're interested in. This being the case, we make separate definitions of normal forms for elements of $H$ in the
case $n=5$ and the case $n\ge 6$. These definitions result in different choices of sets $L$ in the two cases. 

In the  case $n=5$ we say a word $w$  over $(A_0\cup A_0^{-1})^*$ is a \emph{square normal form} if 
$w=w_1w_2$, where $w_1$ is a reduced word in $\FF(a_2,a_4)$ and $w_2$ is a reduced word in $\FF(a_1,a_3)$.
(The graph $C_4$ is a ``square''.)
 Then  every square normal form is a minimal form and every element of $H$ is represented by  a unique square normal form.
\begin{comment}
  ; first we claim that every element
of $H$ has a unique normal form as a word $w$, over $(A_0\cup A_0^{-1})^*$, such that

\ajd{Started editing from here ... say we use 2 normal forms for elements of $H$, one works for all $n\ge 5$ and one only for $n=5$.
  In both cases $L_H$ denotes the set of these normal forms. Then we choose a particular subset $L$ of elements of $F$ such that every
  cyclically reduced element of $F$ is represented by one element of $L$.  Then define asymp. dens.
}
The case $n=5$ is special, as in this case $H=\FF(a_1,a_3)\times \FF(a_2,a_4)$ and normal forms may
be chosen using this decomposition. 
\end{comment}

For $n\ge 5$, 
we say a word $w$  over $(A_0\cup A_0^{-1})^*$ is a \emph{normal form} if 
\be[label=(\roman*)]
\item $w$ is freely reduced and
\item contains no subword of the
  form $a_{i+1}^\e a_{i-1}^\b a_{i}^\d$, (subscripts modulo $n-1$) where $1\le i\le n-1$, %or $a_1^{\e}a_{n-2}^\b a_{n-1}^\d$, where
   $\e,\d\in \{\pm 1\}$ and $\b\in \ZZ$ ($\b$ may be zero).
  \ee
Note that if $w$ is a word in normal form then it is minimal, because the centraliser of $a_i$ in $H$ is $C_H(a_i)=
\la a_{i-1},a_i,a_{i+1}\ra$, and $w$ can contain no subword of the form $a_i^{-\d}u a_i^{\d}$, with $u\in \la a_{i-1}, a_{i+1}\ra$.
 We claim that every element of $H$ is equal to a unique word in normal form. 
To see this, define a word to be \emph{prohibited} if it 
has the form \[(\prod_{j=1}^r a_{i+1}^{\a_j}a_{i-1}^{\b_j})a_i^\d,\] where $r\ge 1$, $\d\in \{\pm 1\}$, $\a_j\neq 0$, for all $j$ and $\b_j\neq 0$, for $j<r$
(where subscripts are written modulo $n-1$). Let $v$ be a minimal length word representing an element of $H$ and suppose that
$v\equiv v_0v_1v_2$, where $v_1$ is prohibited and is the leftmost prohibited subword of $v$: that is, if $v\equiv u_0u_1u_2$, where $u_1$ is prohibited,
then $|v_0|\le |u_0|$. If $v_1=(\prod_{j=1}^r a_{i+1}^{\a_j}a_{i-1}^{\b_j})a_i^\d$ then let $v_1'=a_i^\d(\prod_{j=1}^r a_{i+1}^{\a_j}a_{i-1}^{\b_j})$, so
$v=_Hv_0v_1'v_2$ and if $v_0v_1'v_2\equiv u_0u_1u_2$, where $u_1$ is prohibited, then $|u_0|>|v_0|$ (as the first letter $a_i^\d$ of $v_1'$ cannot be part
of a prohibited subword). Continuing this way we may eventually write $v$ to a normal form $u$ with $v=_H u$. Hence every element of $H$ is
 represented by a normal form.

It remains to show that normal forms are unique, for which we shall use the fact that, if $w$ is a normal form and $w\equiv w_0a_k^\g w_1$, where 
$[a_k, \supp(w_1)]=1$, then $w_0w_1$ is also a normal form. Indeed, if $w_0w_1$ contains a subword $u=a_{i+1}^\e a_{i-1}^\b a_{i}^\d$, with $\e,\d\in \{\pm 1\}$, then 
the first letter $a_{i+1}^\e$ of $u$ occurs in $w_0$ and the last letter $a_i^\d$ occurs in $w_1$. Thus $w_0=w_0'a_{i+1}^\e a_{i-1}^{\b_0}$ and
$w_1=a_{i-1}^{\b_1} a_{i}^\d w_1'$, 
where $\b_0+\b_1=\b$ and $w_0',w_1'$ are initial and terminal subwords of $w_0$ and $w_1$, respectively. As $[a_k, \supp(w_1)]=1$, we have $[a_k,a_i]=1$ so $k\in \{i-1,i,i+1\}$. If $k=i-1$ then $w$ has a subword
$a_{i+1}^\e a_{i-1}^{\b+\g}a_i^{\d}$. If $k=i$ then $w$ has a subword $a_{i+1}^\e a_{i-1}^{\b_0}a_i^{\g}$ and if $k=i+1$ then $w$ has a subword
$a_{i+1}^\g a_{i-1}^{\b_1}a_i^\d$. None of these are possible as $w$ is a normal form, so $w_0w_1$ is  a normal form, as claimed.

We now use induction on length to show each element of $H$ has exactly one normal form. This is certainly true for the element of length $0$, and 
we assume that it is true for all elements of length at most $m$, for some $m\ge 0$. Next suppose  $w$ and $w'$ are normal forms with 
$w=_Hw'$, $l(w)=m+1$ (so $l(w')=m+1$) and $w\nequiv w'$. Let $w\equiv w_0a_i^\e$, where $\e\in \{\pm 1\}$.
Then $w'\equiv w_0'a_j^\d$, for some $\d\in \{\pm 1\}$, where  $j\neq i$; otherwise
$w_0$ and $w_0'$ are distinct normal forms of length $m$ for the same element of $H$. Thus $w$ has right divisors $a_i^\e$ and $a_j^\d$, so $[a_i,a_j]=1$ and
we have $w_0\equiv w_1a_j^{\d}w_2$ and $w_0'\equiv w_1'a_i^{\e}w_2'$, where $[a_j, \supp(w_2)]=[a_i,\supp(w_2')]=1$, $a_j\notin\supp(w_2)$,
$a_i\notin\supp(w_2')$ and $j=i\pm 1$. Therefore,
$w\equiv  w_1a_j^{\d}w_2a_i^{\e}$ and $w'\equiv w_1'a_i^{\e}w_2'a_j^{\d}$ and, using the fact above, both $w_1w_2a_i^{\e}$ and $w_1'a_i^{\e}w_2'$ are normal forms.
As $w_1w_2a_i^{\e}=_H w_1'a_i^{\e}w_2'$, the inductive assumption  implies $w_1w_2a_i^{\e}\equiv w_1'a_i^{\e}w_2'$ and, as $a_i\notin \supp(w_2')$ it follows
that $w_2'\equiv 1$. Similarly $w_2\equiv 1$, and now we have $w_1\equiv w_1'$. Therefore $w\equiv w_1a_j^{\d}a_i^\e$ and $w'\equiv w_1a_i^{\e}a_j^{\d}$, with
$j=i\pm 1$. This is a contradiction, since $w$ and $w'$ cannot both be normal forms, and the result follows.

Therefore we have unique normal forms for elements of $H$, for $n\ge 5$.
Next we define subsets of the set of  normal forms needed for analysis of Theorem \ref{thm:main}. 
\begin{itemize}
\item For $n=5$, let $L_H$ be the set of square normal forms for elements of $H$, and for  $n\ge 6$ let $L_H$ be the set of
  normal forms for elements of $H$. In both cases 
  let $L_H(d)$ be the set of elements of $L_H$ of length at most $d$, and set $l_H(d)=|L_H(d)|$.
%\item $L(0)\subseteq L_H$, the set of normal forms for cyclically reduced elements of $H$,   $l(0)=|L(0)|$;
\item Let $L_H^U\subseteq L_H$  be the set of normal forms
  for elements of $H$ which have no non-trivial left divisor in $U$. Let $L_H^U(d)=L_H^U\cap L_H(d)$ and let $l_H^U(d)=|L_H^U(d)|$.
\item Let $L_U$ be the subset $\{a_{n-1}^\a a_1^\b:\a,\b\in \ZZ\}$ of $L_H$. Then $L_U$ is 
  the subset of $L_H$ consisting of normal forms for elements of $U$ (for all $n\ge 5$). 
  Let $L_U(d)$ be the set of elements of $L_U$ of length at most $d$ and set 
  $l_U(d)=|L_U(d)|$.
\item Let $L_{H,S}(d)$, $L_{H,S}^U(d)$ and $L_{U,S}(d)$ be the subsets $L_H$, $L_H^U$ and $L_U$, respectively, of elements of length exactly $d$.
  Set $l_{H,S}(d)=|L_{H,S}(d)|$, $l_{H,S}^U(d)=|L_{H,S}^U(d)|$ and $l_{U,S}(d)=|L_{U,S}(d)|$.  
\item Let $L_t(k)$ be the set of cyclically reduced words of $(A\cup A^{-1})^*$ representing elements of $\GG'$ which end in a $t$-letter and have
  the form $ug_1t^{\e_1}\cdots g_k t^{\e_k}$, where $g_i\in L_H^U$ and $u\in L_U$. 
\end{itemize}
%}~\\
%
%Let $L(0)$ denote the set normal forms for elements of $H$ and let $L_U(0)\subseteq L(0)$ denote the set of normal forms
%for elements of $H$ which have no non-trivial left divisor in $U$.
 %Next, let $L_t'$ denote the set of cyclically reduced words representing elements of $\GG'$ which end in a $t$-letter and have
 %the form $ug_1t^{\e_1}\cdots g_k t^{\e_k}$, where $g_i\in L_H^U$ and $u\in L_U$. %The set $L$ is our set of representatives of conjugacy classes.
 As the unique element of $L_U\cap L_H^U$ is $1$, the condition that elements of $L_t(k)$ are cyclically
 reduced amounts to the condition that, for all $i$; either $g_i\neq 1$ or
 $\e_{i-1}=\e_i$ (subscripts modulo $n$). 

 Fix $k\ge 1$ and  let $w$ be an element of $L_t(k)$. We say that $w$ is \emph{composed} if either
 \be[label=(\roman*)]
 \item\label{it:composed1} $w=ut^k$ or $w=ut^{-k}$, for some $u\in U$, or
 \item\label{it:composed2}
%   $w$ is conjugate to an element of $L_0$ of the form
 $w=ug_1t^{\a_1}\cdots g_rt^{\a_r}$, where $g_i\in L_H^U\bs \{1\}$, $u\in L_U$, $|\a_i|\ge 1$ and
 $(|\a_1|,\ldots ,|\a_r|)$ is a composition of $k$ into $r$ parts; for some $r$ such that $1\le r\le k$.
 %Moreover, in this case
% we may assume that either $u$ is trivial or $[u,g_i]=1$, for all $i$.
 \ee
 %We call the elements of $L_0$ of the form in (i) and \ref{it:composed2}  \emph{composed} elements of $L_0$.
\begin{itemize}
\item  Let $L(k)$  the set of composed elements of $L_t(k)$ and %, for $d\ge 0$,  let $L_H(d)$ denote the set of elements of $L_H$ of length at most $d$.
 let $\Li(k)$ and $\Lii(k)$ be the sets of elements of $L(k)$ of  types \ref{it:composed1}
 and \ref{it:composed2}, respectively.
\item For $d\ge 0$ denote by $L(d,0)$ the set of normal forms in $L_H$ representing cyclically minimal elements of $H$ and set 
  $l(d,0)=|L(d,0)|$.
\item For  $d\ge 0$ and $k\ge 1$ denote by $\Li_S(d,k)$ the set of elements of $\Li(k)$ such that $u\in L_U(d)$. Set $\li_S(d,k)=|\Li_S(d,k)|$.
\item For $d\ge 0$ and $k\ge 1$  denote by $\Lii_S(d,k)$ the set of elements of $\Lii(k)$ such that $g_i\in L_H^U(d)$, for all $i$, $u\in L_U$ and
  $l(ug_1)\le d$.   Set $\lii_S(d,k)=|\Lii_S(d,k)|$.
\item For  $d\ge 0$ and $k\ge 1$ set $L_S(d,k)=\Li_S(d,k)\cup \Lii_S(d,k)$ and let $l_S(d,k)=|L_S(d,k)|$.
\item  For  $d\ge 0$ and $k\ge 1$ set   $\Li(d,k)=\cup_{l=1}^k\Li_S(d,l)$, $\Lii(d,k)=\cup_{l=1}^k\Lii_S(d,l)$, $L(d,k)=L(d,0)\cup\Li(d,k)\cup\Lii(d,k) $ and 
    $\li(d,k)=|\Li(d,k)|$, $\lii(d,k)=|\Lii(d,k)|$ and $l(d,k)=|L(d,k)|$.
%\end{itemize}
 %Let $L(k)$ be the set of composed elements of $L_t(k)$ and %, for $d\ge 0$,  let $L_H(d)$ denote the set of elements of $L_H$ of length at most $d$.
 %let $\Li(k)$ and $\Lii(k)$ be the sets of elements of $L(k)$ of  types \ref{it:composed1}
 %and \ref{it:composed2}, respectively.
%
%\begin{itemize}
\item \label{it:LiidkrP}
For each integer $r$ such that $1\le r\le k$ and each composition $P$ of $k$ into $r$ parts, %we partition the set $L_2(d,k)$ into
set
\[\Lii_S(d,k,r,P)=\{ug_1t^{\a_1} \cdots g_rt^{\a_r}\in \Lii_S(d,k)\,:\, (|\a_1|,\ldots ,|\a_r|)=P\}.\]
Set $\lii_S(d,k,r,P)=|\Lii_S(d,k,r,P)|$.
\item Let $L=\bigcup_{d\ge 0, k\ge 0}L(d,k)$. 
\end{itemize}

\subsection{Counting normal forms}
To find bounds on the size of $L(d,k)$ we first  compute bounds on the sizes of  $L_H(d)$, $L_U(d)$ and $L_H^U(d)$.
 We consider cases $n=5$ and $n\ge 6$ separately, though the bounds we find for $n\ge 6$ also apply in the case $n=5$,
except in cases where division by $n-5$ is involved. First suppose that $n=5$. 
%Write $l_{H,S}(d)$ for the number of elements of $L_H$ of length exactly $d$.
All words of length $0$ and $1$ are in $L_H$, so $l_H(0)=1$ and $l_{H,S}(1)=2(n-1)=8$.
For $k\ge 2$, the number of normal forms $w_1w_2$ such that $l(w_1)=p$ and $l(w_2)=k-p$, where $1\le p\le k-1$, is $4^23^{k-2}$, and there
are $8\cdot 3^{k-1}$ normal forms $w_1\cdot w_2$ of length $k$ with either $w_1$ or $w_2$ trivial.
Hence $l_{H,S}(k)=8\cdot 3^{k-2}(2k+1)$, for $k\ge 1$. Summing over $k$ from $0$ to $d$ gives
\begin{equation}\label{eq:lhd_5}
  l_H(d)=1+8d3^{d-1}, \textrm{ when } n=5.
\end{equation}

Now consider the case $n\ge 6$. Again all words of length $0$ and $1$ are in $L_H$, so $l_H(0)=1$ and $l_{H,S}(1)=2(n-1)$. 
A word of length $2$ is in $L_H$ if and
only if it is of the form $a_i^\e a_j^\d$, where $\e,\d\in \{\pm 1\}$, and satisfies the conditions that $j\neq i-1$ and that if $i=j$ then $\e=\d$. 
Therefore $l_{H,S}(2)=2(n-1)(2n-5)$. We claim that, setting $\a=2n-5$ and $\g=2n-7$, 
%\begin{equation}\label{eq:LHSbound}
%  2(n-1)\a\g^{m-2}\le l_{H,S}(m)\le  2(n-1)\a^{m-1}, \textrm{ for all } m\ge 2.
%\end{equation}
\begin{equation}\label{eq:LHSbound}
  2(n-1)\g^{m-1}\le l_{H,S}(m)\le  2(n-1)\a^{m-1}, \textrm{ for all } m\ge 2.
\end{equation}
To verify this, assume that it holds for some $m\ge 2$. If $w$ is a normal form of length $m+1$ then $w=ua_j^\d$, for some $u$ of length $m$,
which is necessarily in $L_H$, $a_j\in A_0$ and $\d=\pm 1$. From the argument used to establish the value of $l_{H,S}(2)$; for each such $u$,
there are at most
$2n-5$ choices for $a_j^\d$. Hence $l_{H,S}(m+1)\le \a l_{H,S}(m)\le  2(n-1)\a^{m}$.
Some of these choices do not give rise to normal forms, as they result in prohibited subwords of $w$. On the
other hand if $u=va_i^\e$, and $i\notin \{j-1,j+1\}$ then (as $u\in L_H$) $ua_j^\d$ contains no prohibited subword, so is in $L_H$ (as long as it
is reduced). Hence there are at least $2n-7$ choices of $a_i^\e$ which do result in $ua_i^\e$ being  a normal form. Therefore there
are at least $\g l_{H,S}(m)$ possibilities for $w$. That is, $l_{H,S}(m+1)\ge \g l_{H,S}(m)\ge  2(n-1)\a\g^{m-1}$. Hence, as $\a\ge \g$,  \eqref{eq:LHSbound}
follows by induction on $m$. Summing over $m$ from $0$ to $d$ we obtain 
\begin{equation}\label{eq:LHbound}
  1+\frac{n-1}{n-4}[\g^{d}-1]\le l_{H}(d)\le  1+\frac{n-1}{n-3}[\a^{d}-1],
\end{equation}
for all  $d\ge 1$, in the case $n\ge 6$.

Normal forms of elements of $U$ all have the form $a_{n-1}^a a_1^b$, for some integers $a,b$, and for all $n\ge 5$. Therefore
the number of elements %$l_{U,S}(d)$
of $U$ of length exactly $d$ is $4d$, and 
\begin{equation}\label{eq:lu}
  l_U(d)=1+2d(d+1), \textrm{ for all } d\ge 0 \textrm{ and } n\ge 5.
\end{equation}

%To bound the size of $L(d,k)$, for $k\ge 1$ we first estimate the size of $L_H^U(d)$. 
Next we consider $L_H^U$, the set of normal forms for elements of $H$ which have no non-trivial left divisor in $U$.
When $n=5$ a normal form $w_1w_2$ belongs to $L_H^{U}$ if and only if $w_1$ does not begin with $a_4^{\pm 1}$ and $w_2$ does not begin
with $a_1^{\pm 1}$. There are $2\cdot 3^{l-1}$ elements of $F(a_2,a_4)$ of length $l\ge 1$, beginning with $a_4^{\pm 1}$; and similarly 
 for $F(a_1,a_3)$ and $a_3^{\pm 1}$. It follows that the number of elements of $L_H^U$ of length exactly $k$ is $4\cdot 3^{k-2}(2+k)$. Therefore
\begin{equation}\label{eq:lhdu_5}
l_H^U(d)=3^{d-1}(3+2d),\textrm{ for all } d\ge 0\textrm{ when } n=5.
\end{equation}  
In the case $n\ge 6$ we find bounds on the size of $L_H^U(d)$.
We %first consider the set $L_{H,S}^U(d)$ of elements in $L_H^U$ of length exactly $d$. We
partition $L_H^U$ into three subsets: \ref{it:LH1} normal forms beginning with elements $a_i^{\pm 1}$, where $3\le i\le n-3$; \ref{it:LH2} normal
forms beginning $a_2^{\pm 1}$ and \ref{it:LH3} normal forms beginning $a_{n-2}^{\pm 1}$. %We define $L_{H,S}^U(d)$ to be the set of elements in $L_H^U$ of length exactly $d$ and 
%set $l_{H,S}^U(d)=|L_{H,S}^U(d)|$.\\
\be[label=(\alph*),ref=(\alph*)]
\item\label{it:LH1} Normal forms beginning  $a_i^{\pm 1}$, where $3\le i\le n-3$, have no left divisor in $U$. 
    There are   $2(n-5)$ choices for the  first letter $a_i^{\e}$, where $3\le i\le n-3$, of such a
    normal form. As in the derivation of bounds in \eqref{eq:LHSbound} above, %as the next letter is then $a_k^{\d}$, where $a_k^{\d}\neq a_j^{-\e}$ and $k\neq j-1$. T
    there are  $2n-5$ possibilities for the second letter, and between $2n-7$  and $2n-5$ possibilities for each subsequent letter,  of such a normal form.
    %of such a normal form, so there are at most $2(n-5)(2n-5)^{d-1}$ such normal forms of length exactly $d$.
    Thus, for $d\ge 2$,   $L_{H}^U$ contains at least $2(n-5)\a\g^{d-2}\ge 2(n-5)\g^{d-1}$ and at most  $2(n-5)\a^{d-1}$ elements of type  \ref{it:LH1}, with length exactly $d$.
    Therefore, 
the number $a(d)$ of elements of type  \ref{it:LH1} in  $L_{H}^U(d)$ satisfies
%\begin{align}
%  (n-5)&\left[2+\frac{\a}{n-4}(\g^{d-1}-1)\right]\le a(d)\notag\\
%                                                       &\frac{n-5}{n-3}(\a^{d}-1).\label{eq:abounds}
%\end{align}
\begin{equation}\label{eq:abounds}
  \frac{n-5}{n-4}(\g^{d}-1)\le a(d) \le \frac{n-5}{n-3}(\a^{d}-1).
\end{equation}
\item\label{it:LH2} To bound the number $b(d)$ of normal forms of length at most $d$
  which begin with $a_2^{\pm 1}$ and have no left divisor in $U$, we begin by finding bounds on the number
  $b_S(d)$ of such normal forms of length exactly $d$. We have $b_S(0)=0$ and $b_S(1)=2$, as the possibilities are $a_2^{\pm 1}$. 
To estimate $b_S(d)$, for $d>1$, 
  it is convenient to consider normal forms beginning
  with elements of the set $M$ of reduced words in $\{a_2^{\pm 1}\}(\{a_2^{\pm 1},a_{n-1}^{\pm 1}\})^*$.
  To this end suppose $w$ is a normal form of type \ref{it:LH2} and $w\equiv gh$, where
  $g$ is a maximal prefix of $w$ belonging to $M$ and $g$ has length $t\ge 1$. There are $2$ possibilities for the first letter of $g$ and $3$ for each subsequent letter;
  so $2.3^{t-1}$ such $g$ altogether. To find those $gh$ with %length at most $d$ and
  no left divisor in $U$ we distinguish between those
  $g$ ending $a_2^{\pm 1}$ and those ending $a_{n-1}^{\pm 1}$. Let $N$ be the set of all elements of $M$  ending in
  $a_2^{\pm 1}$ and let $P=M\bs N$. Let $n(t)$ and $p(t)$ be the number of elements of $N$ and $P$, respectively, of length $t\ge 1$.
  It follows (by induction) that
  \[n(t)=\begin{cases}
      3^{t-1}-1& \textrm{if }t\textrm{ is even}\\
      3^{t-1}+1& \textrm{if }t\textrm{  is odd}
    \end{cases}
    ,
    \textrm{ and so } p(t)=\begin{cases}
      3^{t-1}+1& \textrm{if }t\textrm{ is even}\\
      3^{t-1}-1& \textrm{if }t\textrm{  is odd}
    \end{cases}
    ,
  \]
  for all $t\ge 1$.

 If $w\equiv gh$, where $g\in P$ so $g$ ends in $a_{n-1}^{\pm 1}$, then $h$ cannot begin with a letter in $\{a_2^{\pm 1},a_{n-1}^{\pm 1}\}$,
  by maximality of $g$, 
  cannot begin with $a_{n-2}^{\pm 1}$, by definition of normal form, and cannot begin with $a_1^{\pm 1}$ as $w$ has no left divisor
  in $U$. Thus the first letter of $h$ belongs to $\{a_3^{\pm 1},\ldots ,a_{n-3}^{\pm 1}\}$. As neither $a_1$ or $a_{n-1}$ commutes with any
  element of this set, each subsequent letter of $h$ may be any member of  $A_0^{\pm 1}$ which results in a normal form. As before, the 
  the number of possibilities for such an $h$ of length $s$ is at most  $2(n-5)\a^{s-1}$ and at least $2(n-5)\g^{s-1}$.

 On the other hand, if $w\equiv gh$, where $g\in N$ so $g$ ends in $a_{2}^{\pm 1}$, then $h$ cannot begin with a letter in 
  $\{a_2^{\pm 1},a_{n-1}^{\pm 1}\}$ and cannot begin with $a_{1}^{\pm 1}$. As $a_{n-1}$ does not
  commute with $a_2$ and $a_1$ does not commute with the first letter of $h$, each subsequent letter of $h$ may again be any element
  of $A_0^{\pm 1}$ which results in a normal form. The number 
  of possibilities for $h$ of length $s$ in this case is then at most   $2(n-4)\a^{s-1}$and at least  $2(n-4)\g^{s-1}$.

  Now, for $k\ge 2$,
  let $b(k,t)$ be the number of normal forms $w\in L_H^U$ of length $k$, of type \ref{it:LH2}, where $w\equiv gh$, as above, and $g$ has length $t\ge 1$;
  so $h$ is of length $s=k-t$. When $t<k$ (so $s\ge 1$) %writing $\a=2n-5$ and $\g=2n-7$, 
  \begin{align}
    b(k,t)&\le 
            n(t)(2(n-4)\a^{s-1})+p(t)(2(n-5)\a^{s-1})\notag\\
          &=2\a^{s-1}[(n-4)n(t)+(n-5)p(t)].\label{al:bktupper}
  \end{align}
  Similarly, %writing $\g=2n-7$, 
  \begin{equation}\label{eq:bktlower}
 b(k,t)\ge 2\g^{s-1}[(n-4)n(t)+(n-5)p(t)],
\end{equation}
where $s=k-t$. 
\begin{comment}
  We have  
  \[(n-4)n(t)+(n-5)p(t)=\begin{cases}
      (2n-9)3^{t-1}+1 & \textrm{if }t\textrm{ is odd}\\
       (2n-9)3^{t-1}-1& \textrm{if }t\textrm{  is even}
     \end{cases}.
   \]
\end{comment}   
For the cases where $t=k$ and $h$ is trivial,  $b(k,k)=n(k)+p(k)=2\cdot 3^{k-1}$.
We leave consideration of $b(d)$ at this point and turn to normal forms of type (c), as this makes later computation simpler.
\item\label{it:LH3}
Again, to bound  the number $c(d)$ of normal forms of length at most $d$,
which begin with $a_{n-2}^{\pm 1}$ and have no non-trivial left divisor in $U$,  we shall first bound the number $c_S(d)$ of such normal forms of length exactly $d$.
As before, $c_S(0)=0$ and $c_S(1)=2$. To find bounds on $c_S(k)$ for $k\ge 2$ we consider normal forms $w$ which factor as $gh$, where $g$ is a maximal prefix in the 
set of reduced words in $\{a_{n-2}^{\pm 1}\}(\{a_{n-2}^{\pm 1},a_{1}^{\pm 1}\})^*$. Let the number of such $g$ of length $t\ge 1$ which
end in $a_{n-2}$ be  $n'(t)$ and the number ending in $a_1$ be $p'(t)$.  Then $n'(t)=n(t)$ and $p'(t)=p(t)$, where $n(t)$ and $p(t)$
are as in case  \ref{it:LH2} above. Following through the argument of case \ref{it:LH2}, we find that the number $c(k,t)$  of
  normal forms $w\in L_H^U$ of length $k$, of type \ref{it:LH3}, where $w\equiv gh$, as above, and $g$ has length $t\ge 1$, satisfies
  \begin{align}
    c(k,t) &\le 2\a^{s-1}[(n - 5)n(t) + (n - 4)p(t)] \textrm{ and }\label{al:cktupper}\\
     c(k,t)&\ge 2\g^{s-1}[(n - 5)n(t) + (n - 4)p(t)]\label{al:cktlower},
  \end{align}
  where $k\ge 2$, $s=k-t$ and $\a$ and$\g$ are as above.

  As   \[[(n - 4)n(t) + (n - 5)p(t)] + [(n - 5)n(t) + (n - 4)p(t)]=2(2n-9)3^{t-1}\]
  combining \eqref{al:bktupper}, \eqref{eq:bktlower}, \eqref{al:cktupper} and  \eqref{al:cktlower}
  and setting $\b=2n-9$, gives
  \begin{equation*}%\label{eq:bcbounds}
4\b\g^{s-1}3^{t-1}\le b(k,t)+c(k,t)\le 4\b\a^{s-1}3^{t-1},
\end{equation*}
when $k\ge 2$ and $t\ge 1$.
Also, $c(k,k)=b(k,k)=2\cdot 3^{k-1}$, when $k\ge 2$.

 Therefore, for $n\ge 6$ and $k\ge 2$, %setting  $b_S(k)$ and $c_S(k)$ equal to the number of type \ref{it:LH2} and \ref{it:LH3} normal forms in $L_H^U$,  of length $k\ge 2$, (and having no left divisor in $U$), 
   \begin{align}
     b_S(k)+c_S(k)&\le 4[3^{k-1}+\b\sum_{t=1}^{k-1} \a^{k-t-1} 3^{t-1}]\notag\\
                  &\textcolor{black}{=4\left(3^{k-1}+\b\frac{\a^{k-1}-3^{k-1}}{\a-3}\right)}\notag\\
           &\textcolor{black}{=\frac{2}{n-4}(3^{k-1}+ \b\a^{k-1})}\notag\\
                  &\le \frac{2}{n-4}(\b+1)\a^{k-1},\textrm{ as }3\le \a,\notag\\
                    & =4\a^{k-1}.\label{al:bcsupper}
   \end{align}
%             &=A_03^{k-1}+B_0\a^{k-1},\label{al:bcsupper}
%   where $A_0$ and $B_0$ are rational numbers dependent only on $n$, with $0<A_0\le 2$ and $2\le B_0<4$. 
   Similarly, for $n\ge 6$ and $k\ge 2$, 
    \begin{align}
    b_S(k)+c_S(k)&\ge 4[3^{k-1}+\b\sum_{t=1}^{k-1} \g^{k-t-1} 3^{t-1}]\notag\\
                  &\textcolor{black}{=4\left(3^{k-1}+\b\frac{\g^{k-1}-3^{k-1}}{\g-3}\right)}\notag\\
                  &\textcolor{black}{=\frac{2}{n-5}(-3^{k-1}+ \b\g^{k-1})}\notag\\
                  &\textcolor{black}{\ge\frac{2}{n-5}(\b-1)\g^{k-1}},\textrm{ as }3\le \g,\notag\\
                  &=4\g^{k-1}.\label{al:bcslower}
    \end{align}
Note that in fact \eqref{al:bcsupper} and \eqref{al:bcslower} hold for all $k\ge 1$. 
\begin{comment}
   When $n=5$, we have $\g=3$ so
   \[b_S(k)+c_S(k)\ge 4[3^{k-1}+\b\sum_{t=1}^{k-1} \g^{k-t-1} 3^{t-1}]=\g^{k-1}(k+3)\ge 5\g^{k-1},\]
   as $k\ge 2$. Therefore, in all cases
\begin{equation}%\label{eq:bcslower}
b_S(k)+c_S(k)\ge B'_0\g^{k-1},
\end{equation}
where $B'_0=4$ or  $5$, and depends only on $n$.
\end{comment}
 Then
\begin{align}   
  b(d)+c(d)&= \sum_{k=1}^d b_S(k)+c_S(k)\le 4\sum_{k=1}^d \a^{k-1}\notag\\
      &= 4\left(\frac{\a^{d}-1}{\a-1}\right)=2\left(\frac{\a^{d}-1}{n-3}\right).\label{al:bcdupper}
\end{align}
%where $0<A_1\le 1$, $0<B_1<\frac{1}{2}$ and $C_1$ are rational numbers (depending only on $n$).
% &=A_0\left(\frac{3^d}{2}\right)+B_0\left(\frac{\a^d}{2n-6}\right)+C_0\left(\frac{(-1)^d}{2}\right)+D_1\\
Similarly,
%   &\ge 2+\sum_{k=2}^dB_0'\g^{k-1}\notag\\
\begin{equation} \label{al:bcdlower}  
 b(d)+c(d) \ge 2\left(\frac{\g^{d}-1}{n-4}\right).
\end{equation}
\ee
As $l_H^U(d)=a(d)+b(d)+c(d)$, from \eqref{eq:abounds}, \eqref{al:bcdupper} and \eqref{al:bcdlower},
\begin{equation}\label{eq:lhdu}
\g^d-1\le l_H^U(d)\le \a^d-1, \textrm{ when } n\ge 6.
\end{equation}
%where $B'_2=(n-3)/2(n-4)$, so $1/2<B'_2\le 1$.
% where the rational numbers $A_2,B_2,B'_2,C_2$ and $C'_2$ depend only on $n$, $0<A_2\le 1$, $0<B_2<\frac{3}{2}$ and $0<B'_2<\frac{7}{2}$,
%where the rational numbers $B_2,B'_2,C_2$ depend only on $n$, $0<B_2<\frac{5}{2}$ and $0<B'_2<\frac{7}{2}$,
%$\a=2n-5$ and $\g=2n-7$. 

In what follows we shall need to know how many elements of $L_H^U(d)$ do not belong to $U\cup \maln_{H}(U)$.  
The elements of $H$ which are not in $\maln_H(U)$ or $U$ are those which belong either to
\begin{itemize}
\item $\la a_{n-1}\ra\la a_{n-2},a_1\ra$ and have support containing $a_{n-2}$, or to
\item  $\la a_{1}\ra\la a_{n-1},a_2\ra$ and have support containing $a_{2}$.
\end{itemize}
As elements of $L_H^U$ have no non-trivial left divisor in $U$, elements of $L_H^U$ which do not belong to $U\cup \maln_{H}(U)$ are those in
$\la a_{n-2},a_1\ra$ which begin with $a_{n-2}^{\pm 1}$, or those in $\la a_{2},a_{n-1}\ra$ which begin with $a_{2}^{\pm 1}$, all of which are in normal form,
as long as they are freely reduced. There
are $4\cdot 3^{p-1}$ elements of length $p$ of these forms, so 
the number $e(d)$ of elements of $L_H^U(d)$ which do not belong to $U\cup \maln_{H}(U)$ is
\begin{equation}\label{eq:notmaln}
  e(d)=2(3^d-1), \textrm{ for all } n\ge 5.
\end{equation}

To bound the size of $L(d,0)$ we simply note  that
\begin{equation}\label{eq:Ld0bound}
  1\le l(d,0)\le l_H(d), \textrm{ for all } d\ge 0.
\end{equation}
%which is contained in $\{a_1,a_2,a_{n-2},a_{n-1}\}$ and which contains exactly one of the letters $a_2$ or $a_{n-2}$. Thus the elements of $H$ which
%are %of length $d\ge 1$ and
%not in $\maln_H(U)$ or $U$ are those with normal forms \[a_1^{\a_1}a_{n-2}^{\b}a_{n-1}^{\a_{n-1}} \textrm{ or }
% a_{n-1}^{\a_{n-1}}a_1^{\a_1}a_{2}^{\b},\]
%where 
% such that $|\a_1|+|\a_{n-1}|+|\b|=d$ and
%$\b\neq 0$. Therefore
% the elements of $L_H^U$  which are not in $\maln_H(U)$ or $U$ 
%are the non-trivial elements of $\la a_2\ra \cup \la a_{n-2}\ra$. 
Bounds on $l(d,k)$, for $k>0$, are found by considering $\Li(d,k)$ and $\Lii(d,k)$ separately. 
From \eqref{eq:lu}, there are 
\begin{equation*}
\li_S(d,k)=2[1+2d(d+1)]
\end{equation*}
elements $ut^{\pm k}$ of type \ref{it:composed1} in $\Li_S(d,k)$.
Therefore
\begin{equation}\label{eq:l1dk}
\li(d,k)=2k[1+2d(d+1)],
\end{equation}
for all $n\ge 5$.

Elements of $\Lii_S(d,k,r,P)$ have the form $ug_1t^{\a_1}\cdots g_{r}t^{\a_r}$ subject to
the conditions of the definition on page \pageref{it:LiidkrP}.
%The number $\lii_S(d,k,r,P)$ of  elements of $\Lii_S(d,k,r,P)$ %of type \ref{it:composed2} in $L(d,k)$
Here $ug_1\in H$ and $l(ug_1)\le d$ and, as every element of $H$ may be uniquely expressed as $vg$, where $v\in U$ and
$g$ has no left divisor in $U$, there are $l_H(d)$ possibilities for $ug_1$. For
$i>1$, $g_i\in L_H^U$, so there are $l_H^U(d)^{r-1}$ possibilities for $(g_2,\ldots ,g_r)$.
Given the fixed composition $P=(a_1,\ldots ,a_r)$ of $k$, there are $2^r$ choices for $(t^{\a_1},\ldots ,t^{\a_r})$, with $(|\a_1|,\ldots ,|\a_r|)=(a_1,\ldots ,a_r)$.
Therefore
\begin{equation}\label{eq:liidkrP}
  \lii_S(d,k,r,P)=2^rl_H(d)(l_H^U(d))^{r-1}.
\end{equation}
\begin{comment}
From \eqref{eq:LHbound} and \eqref{eq:lhdu}, for some
constants $B_4,B'_4, C_4, C'_4$, with $B_4, B'_4>0$, 
\begin{align*}
  2^r(B'_4\g^d+C'_4)(B'_2\g^d+C'_2)^{r-1}&\le \lii_S(d,k,r,P)\\
                                         & \le 2^r(B_4\a^d+C_4)(A_23^d+B_2\a^d+C_2)^{r-1},
\end{align*}
% 
so there are constants $B_5, B'_5$ and $C_5$, with $B_5$ and $B'_5$ positive, such that 
\begin{equation}\label{eq:liidkrPbound}
B'_5\g^{dr}\le   \lii_S(d,k,r,P) \le (B_5\a^{d}+C_5)^r.
\end{equation}
\end{comment}
%%%%%%%%%%%%%%%%%%%%%%%%%%%%%%%%%%%%%%%%%%%

There are $\binom{k-1}{r-1}$ compositions of $k$ into $r$ parts  so % , from \eqref{eq:liidkrPbound}, 
the number $\lii_S(d,k,r)$ of elements of $\Lii_S(d,k)$ which involve a partition of $k$ into $r$ parts
is
\[\lii_S(d,k,r)=\binom{k-1}{r-1}\lii_S(d,k,r,P).\]
\begin{comment}
and, from \eqref{eq:liidkrPbound}, satisfies 
\begin{equation*}
\binom{k-1}{r-1}B'_5\g^{dr}\le \lii_S(d,k,r)\le \binom{k-1}{r-1} (B_5\a^{d}+C_5)^r.
\end{equation*}
\end{comment}
As $\lii_S(d,k)=\sum_{r=1}^k \lii_S(d,k,r)$, \eqref{eq:liidkrP} gives
\begin{align}
  \lii_S(d,k)&= \sum_{r=1}^k\binom{k-1}{r-1} 2^rl_H(d)(l_H^U(d))^{r-1}\notag\\
  &=2l_H(d)(2l_H^U(d)+1)^{k-1}.\notag 
\end{align}
%\label{eq:liisdk}
\begin{comment}
and %\eqref{eq:liidkrPbound}  implies that 
\begin{equation*}
  B'_5\g^d(1+\g^d)^{k-1}\le \lii_S(d,k) \le (B_5\a^d+C_5)[(B_5\a^d+C_5)+1]^{k-1}
\end{equation*}
so
\begin{equation*}
  B'_5\g^{dk}\le \lii_S(d,k) \le (B_5\a^d+C_6)^{k},
\end{equation*}
for some constant $C_6(=C_5-1)$.
\end{comment}
Now
\begin{align}
  \lii(d,k)&=\sum_{p=1}^k \lii_S(d,p),\notag\\
           &= \sum_{p=1}^k 2l_H(d)(2l_H^U(d)+1)^{p-1}\notag\\
           &= 2l_H(d)\frac{(2l_H^U(d)+1)^k-1}{2l_H^U(d)}\notag\\
  &=\frac{l_H(d)}{l_H^U(d)}[(2l_H^U(d)+1)^k-1].\label{eq:liidk}
\end{align}
\begin{comment}
\begin{equation*}
B'_5\g^d\left(\frac{\g^{dk}-1}{\g^d-1}\right)\le \lii(d,k) \le(B_5\a^d+C_6)\left(\frac{(B_5\a^d+C_6)^k-1}{B_5\a^d+C_5}\right)
\end{equation*}
so
\begin{equation}\label{eq:l2dkbounds}
B'_5(\g^{dk}-1)\le \lii(d,k) \le \frac{(B_5\a^d+C_6)^{k+1}}{B_5\a^d+C_5}.
\end{equation}
%  \lii_S(d,k) &\ge \sum_{r=1}^k l_2(d,k,r)\notag\\
%  &= \textcolor{black}{\sum_{r=1}^k \binom{k-1}{r-1}2^rl_H(d)^r}\notag\\
%  &=2l_H(d)(1+2l_H(d))^{k-1}\label{al:l2dkl}.
%Similarly, using \eqref{eq:ldk2u} instead of \eqref{al:ldk2l}, we have
%\begin{equation}\label{eq:l2dku}
%l_2(d,k)\le \left(\frac{n-1}{n-3}\right) (\a^{d+1}-1) (1+2l_H(d))^{k-1}.
%\end{equation} 
Finally, as $l(d,k)=l(d,0)+\li(d,k)+\lii(d,k)$, \eqref{eq:Ld0bound}, \eqref{eq:LHbound}, \eqref{eq:l1dk} and \eqref{eq:l2dkbounds} together imply that
\begin{align}
  2k[1+2d(d+1)]&+ B'_5(\g^{dk}-1)\le l(d,k) \notag\\
  &\le 1+\frac{n-1}{n-3}(\a^d-1) +2k[1+2d(d+1)]+\frac{(B_5\a^d+C_6)^{k+1}}{B_5\a^d+C_5}.\label{eq:ldk}
\end{align}
\end{comment}

We are now ready to calculate the asymptotic density of subsets of $L$ which satisfy conditions
\ref{it:main_hyp1} -- \ref{it:main_hyp4} of Theorem \ref{thm:main}.
\subsection{Proof of Proposition {\ref{prop:ad}}}
For each i, from 1 to 4, let $Z_i$ be the subset of
$L$ which satisfies condition i of Theorem \ref{thm:main}. We shall show that $\rho(Z_i)=1$, for $1\le i\le 4$. As a finite
intersection of generic sets is generic, and $L_Y=\cap_{i=1}^4 Z_i$, this proves the proposition. 
\paragraph{Theorem \ref{thm:main}, condition {\ref{it:main_hyp1}}.} An element $s\in L$ satisfies this condition if and only if it
does not belong to $H$.  Thus the intersection of $L(d,k)$ with the set $Z_1^c$
of these failing elements is $L(d,0)$. From \eqref{eq:Ld0bound}, the asymptotic density of the set $Z_1^c$ %the set of elements which fail condition {\ref{it:main_hyp1}}
is
\[\rho(Z_1^c)\le \lim_{d,k\maps \infty} \frac{l_H(d)}{l(d,k)},\]
and from \eqref{eq:lhd_5}, \eqref{eq:LHbound}, \eqref{eq:liidk} and the definition of $L(d,k)$,  it follows that $\rho(Z_1^c)=0$.
\begin{comment}, for some positive constant $D_1$,
\[\frac{l_H(d)}{l(d,k)}\le D_1\frac{\a^d}{\g^{dk}}.\]
As $\g^3\ge \a^2$, for all $n\ge 5$, we have $\g^{dk}\ge \a^{2dk/3}$, so \[\frac{l_H(d)}{l(d,k)}\le D_1 \frac{\a^d}{\a^{2dk/3}},\] and $\rho(Z^c_1)=0$.  
\end{comment}
Therefore, $\rho(Z_1)=1$; that is the set of elements of $L$ satisfying condition {\ref{it:main_hyp1}} has asymptotic density $1$. 

\paragraph{Theorem \ref{thm:main}, condition {\ref{it:main_hyp2}}.}
% Let $Z_2^c$ denote the set of elements of $L$ which are not $t$-thick.
There are no $t$-thick elements in $L(d,0)$
%
%From the previous paragraph, the set of elements of $t$-length at least $1$ has asymptotic density $1$, so it suffices to
%consider here elements of $t$-length at least $1$; that is $k\ge 1$. 
and elements of $\Li$, are 
by definition, all $t$-thick.
Thus the number of $t$-thick elements in $L(d,k)$ is $z_2(d,k)=\li(d,k)+z_2^{(ii)}(d,k)$, where $z_2^{(ii)}(d,k)$ is the
number of $t$-thick elements in $\Lii(d,k)$.  
To find the number of $t$-thick   elements in $\Lii(d,k)$, 
 begin by considering the $t$-thick
  elements of $\Lii_S(d,k,r,P)$. An element
 $w=ug_1t^{\a_1}\cdots g_rt^{\a_r} \in \Lii_S(d,k,r,P)$ fails to be 
 $t$-thick only if at least one of the $g_i$ is not in $\maln_H(U)$.
 \begin{comment}
 Denote by  $Z_2^c(d,k,r,P)$ the set of elements of $\Lii_S(d,k,r,P)$, which are not $t$-thick
 and by $Y_i$ the set of elements of $\Lii_S(d,k,r,P)$, with  $g_i\notin \maln_H(U)$,
 so $Z_2^c(d,k,r,P)=\cup_{i=1}^r Y_i$. 
 %The number $\nmal_2(d,k,r,P)$ of elements of $Z_2(d,k,r,P)$, with  $g_2\notin \maln_H(U)$. 
 Consider the set $Y_2$.
\end{comment}
For $i\ge 2$, $g_i$ is an element of $L_H^U(d)$ and, from  \eqref{eq:notmaln},
there are $e(d)=2(3^d-1)$ elements of $L_H^U(d)$ which are not in $\maln_H(U)$. Therefore, there are
$t_H^U(d)=l_H^U(d)-2(3^d-1)$ possible choices for each $g_i$ in the $t$-thick element $w$.
\begin{comment}
  . As in the derivation
 of \eqref{eq:liidkrP} it follows that, setting $y_i=|Y_i|$,
\begin{equation}\label{eq:nmal2}
y_2=2^rl_H(d)(l_H^U(d))^{r-2}\cdot 2(3^d-1).
\end{equation}
Note that, for $i>2$ we have $y_i=y_2$. 

%\begin{comment}
From \eqref{eq:liidk} and \eqref{eq:nmal2}, since $r\le k$, 
\begin{align}
  \frac{y_2}{l(d,k)}&\le \frac{y_2}{\lii(d,k)}\notag\\
                                 &\le \frac{[2^{r+1}l_H(d)(l_H^U(d))^{r-2} 3^d]l_H^U(d)}{l_H(d)[(2l_H^U(d)+1)^k-1]}\notag\\
                                 &\le \frac{2^{k+1} l_H^U(d)^{k-1} 3^d}{(2l_H^U(d)+1)^k-1}\notag\\
  &=\frac{2 \cdot 3^d}{l_H^U(d)[(1+x)^k-x^k]},\notag
\end{align}
where $x=1/2l_H^U(d)$. 
When $n=5$, from \eqref{eq:lhdu_5} it follows that $\rho(y_2)=0$.
When $n\ge 6$, as $\g=2n-7\ge 5$, and $l_H^U(d)\ge B'_5\g^d$, from \eqref{eq:lhdu}, again $\rho(y_2)=0$.
This holds in exactly the same way for $2\le i\le r$: that is, $\rho(y_i)=0$, for all such $i$.
\end{comment}

Now consider $ug_1$. 
By definition, $l(ug_1)\le d$ and supposing that $u$ has length $s$, where $1\le s\le d-1$, from \eqref{eq:notmaln} and the
fact that there are $4s$ elements of $U$ of length $s$, 
 there are $4s\cdot e(d-s)$ possible $ug_1$, with $g_1\notin \maln_{H}(U)$. Summing over $s$, such that $0\le s\le d-1$,
the total possible number of such $ug_1$ is $e'(d)$ where
\begin{align}
  e'(d)&=e(d)+4\sum_{s=1}^{d-1}s\cdot e(d-s)\notag\\
  &=e(d)+8\sum_{s=1}^{d-1} s(3^{d-s}-1)\notag\\
  &=2(3^d-1)+8\left(\frac{3(3^d-1)-6d}{4}\right)-4d(d-1)  \notag\\
  &= 8(3^d-1)-4d(2+d).\notag
\end{align}
\begin{comment}
Therefore there are 
% \label{al:1nmaln}
%&=t(d)+8\sum_{s=1}^{d-1} (s 3^{d-1}- s) \notag\\
Then,
\begin{equation}\label{eq:nmal1}
y_1=2^re'(d)(l_H^U(d))^{r-1}.
\end{equation}

From \eqref{eq:liidk} and \eqref{eq:nmal1}, since $r\le k$, 
\begin{align}
  \frac{y_1}{l(d,k)}&\le \frac{y_2}{\lii(d,k)}\notag\\
                                 &\le \frac{[2^{r}e'(d)l_H^U(d)^{r-1}]l_H^U(d)}{l_H(d)[(2l_H^U(d)+1)^k-1]}\notag\\
   &\le\frac{e'(d)}{l_H(d)[(1+x)^k-x^k]},\notag
\end{align}
where $x=1/2l_H^U(d)$.
When $n=5$, as $l_H(d)=1+8d\cdot 3^{d-1}$, it follows that $\rho(y_1)=0$.
When $n\ge 6$, as $\g=2n-7\ge 5$, and $l_H(d)\ge (n-1)\left(2+\frac{\a}{n-4}[\g^{d-1}-1]\right)$, from \eqref{eq:LHbound},
again $\rho(y_1)=0$.

This implies that $Z_2^c(d,k,r,P)$ has  asymptotic density $0$, as it is the union of the sets $Y_i$, $i=1,\ldots ,r$, all of which have asymptotic density $0$. As
\end{comment}
The element $ug_1$ is an arbitrary element of $L_H(d)$, so there are $t_H(d)=l_H(d)-e'(d)$ possible choices for $ug_1$ in the $t$-thick element $w$
of $\Lii_S(d,k,r,P)$.
Combining these facts, as in the calculation of $\lii(d,k)$, we see that the number $z_2^{(ii)}(d,k)$ of $t$-thick elements in $\lii(d,k)$ is
\begin{equation*}%\label{eq:z2iidk}
z_2^{(ii)}(d,k)=\frac{t_H(d)}{t_H^U(d)}[(2t_H^U(d)+1)^k-1].
\end{equation*}
\begin{comment}
\begin{align}
  z_2^c(d,k,r,P)&=|Z_2^c(d,k,r,P)|\le \sum_{i=1}^r y_i\notag\\
  &=2^rl_H^U(d)^{r-2}[2(r-1)(3^d-1)l_H(d)+e'(d)l_H^U(d)].\notag\\
\end{align}
Summing over all partitions of $r$ and then over all $r$ from $1$ to $k$, we see that the number
$z_{2,S}^c(d,k)$ of elements of $\Lii_s(d,k)$ which are not $t$-thick satisfies
\begin{align}
  z_{2,S}^c(d,k)&=\sum_{r=1}^k\binom{k-1}{r-1}  z_2^c(d,k,r,P)\notag\\
  &=2^rl_H^U(d)^{r-2}[2(r-1)(3^d-1)l_H(d)+e'(d)l_H^U(d)].\notag\\
\end{align}

%%%%%%%%%%%%%%%%%%%%%%%%%%%%%%%%%%%%%%%%%%%%%%%%%%%%%%%%%%%%%%%%%%%%%%%%%%%%%%%%%%%%%%%%%%%%%%%%%%%%%%%%%%%%%%%%%%%%%%%%%%%%%%%%%%%%%%%%%%%%%%%%%%%%%
When $i=1$, we have 
%There are $4d$ elements  of $L_H^U(d)$  which are not in $\maln_H(U)$ or $U$ 
%so there are $4d\cdot 2^k|L_H^U(d)|^{k-1}$ elements of $L(d,k)$ which are not $t$-thick.

Denote the set of elements of $L$ which are not $t$-thick by $T$.
Then from \eqref{eq:lhdu}, \eqref{al:l2dkl} and \eqref{al:nmaldk}, %for all $k>0$, 
\[\lim_{d,k\maps \infty}\frac{|T\cap L(d,k)|}{|L(d,k)|}=0.\]
Therefore $\rho(T)=0$ and, %the convergence rate of to zero is exponential; so
in the terminology of \cite{MSU},
$T$ is %exponentially
negligible, relative to $L$. 
%%Therefore the subset of $L$ satisfying condition \ref{it:main_hyp3} of Theorem \ref{thm:main} is (exponetially) generic, relative to $L$.
\end{comment}
As the number of $t$-thick elements in $L(d,k)$ is  $z_2(d,k)=\li(d,k)+z^{(ii)}_2(d,k)$, 
\begin{align}
   \rho(Z_2)&=\lim_{d,k\maps \infty} \frac{z_2(d,k)}{l(d,k)}= 
 \frac{\li(d,k)+z^{(ii)}_2(d,k)}{l(d,k)}\notag\\
           &\ge \frac{\li(d,k)+z^{(ii)}_2(d,k)}{l_H(d)+\li(d,k)+\lii(d,k)},\textrm{ from \eqref{eq:Ld0bound}}\notag\\
  &= \frac{[\li(d,k)\cdot t_H^U(d)+t_H(d)((2t_H^U(d)+1)^k-1)]l_H^U(d)}{[(l_H(d)+\li(d,k))l_H^U(d)+l_H(d)((2l_H^U(d)+1)^k-1)]t_H^U(d)}.\notag
\end{align}
First note that
\[\frac{l_H^U(d)}{t_H^U(d)}=\frac{l_H^U(d)}{l_H^U(d)-e(d)}=1+\frac{e(d)}{l_H^U(d)-e(d)}.\]
If $n=5$ then, from \eqref{eq:lhdu_5} and \eqref{eq:notmaln}, 
\[\frac{e(d)}{l_H^U(d)-e(d)}=\frac{2(3^{d}-1)}{3^{d-1}(3+2d)-2(3^d-1)},\]
which has limit $0$ as $d,k\maps \infty$, so 
\begin{equation}\label{eq:ltlim}
  \lim_{d,k\maps \infty}\frac{l_H^U(d)}{t_H^U(d)}=1.
\end{equation}
If $n\ge 6$ then $\g^d-1\le l_H^U(d)$, and $\g\ge 5$, from \eqref{eq:lhdu},
so $e(d)/(l_H^U(d)-e(d))$ has limit $0$ as $d,k\maps \infty$, and \eqref{eq:ltlim} holds again.

Next, 
\[ \lim_{d,k\maps \infty}\frac{\li(d,k)\cdot t_H^U(d)}{(l_H(d)+\li(d,k))l_H^U(d)+l_H(d)((2l_H^U(d)+1)^k-1)}=0,\]
for all $n\ge 5$.
Finally, as
\[\lim_{d,k\maps \infty} \frac{(l_H(d)+\li(d,k))l_H^U(d)}{(2l_H^U(d)+1)^k-1}=0,\]
\begin{align*}
  \lim_{d,k\maps \infty} &
  \frac{t_H(d)((2t_H^U(d)+1)^k-1)}{(l_H(d)+\li(d,k))l_H^U(d)+l_H(d)((2l_H^U(d)+1)^k-1)]}\\
  &=\lim_{d,k\maps \infty}\frac{t_H(d)((2t_H^U(d)+1)^k-1)}{l_H(d)((2l_H^U(d)+1)^k-1)]}.
\end{align*}
We have
\[\frac{t_H(d)}{l_H(d)}=\frac{l_H(d)-e'(d)}{l_H(d)}=1-\frac{e'(d)}{l_H(d)}.\]
If $n=5$,
\[\frac{e'(d)}{l_H(d)}=\frac{8(3^d-1)-4d(2+d)}{1+8d\cdot 3^{d-1}},\]
which has limit $0$ as $d,k\maps \infty$.

If $n\ge 6$, then from \eqref{eq:LHbound}, $l_H(d)\ge  (n-1)(\g^{d}-1)/(n-4)$ so again $e'(d)/l_H(d)$
has limit $0$.
Therefore in all cases
\[\lim_{d,k\maps \infty}\frac{t_H(d)}{l_H(d)} =1.\]

Also, it follows from \eqref{eq:ltlim} that
\[\lim_{d,k\maps \infty}\frac{(2t_H^U(d)+1)^k-1}{(2l_H^U(d)+1)^k-1}=1.\]
Combining all the above we see that $\rho(Z_2)=1$: that is  the set of elements of $L$ which satisfy Theorem \ref{thm:main}.\ref{it:main_hyp2} has
asymptotic density $1$. 
%%%%%%%%%%%%%%%%%%%%%%%%%%%%%
\paragraph{Theorem \ref{thm:main}, condition {\ref{it:main_hyp3}}.}
A word in $L$ represents an element of $\la \st(t)\ra$ if and only if it belongs to $L_U$ or $\Li$. Therefore, denoting by $Z_3^c$ the set
of elements of $L$ which do not satisfy Theorem \ref{thm:main}, condition {\ref{it:main_hyp3}}, 
$\rho(Z_3^c)$ is the limit of $[l_U(d)+\li(d,k)]/l(d,k)$ as $d,k\maps \infty$. As this limit is $0$, it follows that $\rho(Z_3)=1$.   
%%%%%%%%%%%%%%%%%%%%%%%%%%%%%%%%%%%%%%%%%%%%%%%%%%%%%%%%%%%%%%%%%%%%%%
\paragraph{Theorem \ref{thm:main}, condition {\ref{it:main_hyp4}}}
%Finally we consider elements of $L$ which fail condition \ref{it:main_hyp4}
%of Theorem \ref{thm:main}.
If $w$ in $\GG'$ is not a $t$-root we say it is a $t$\emph{-power}. 
An element $g\in L(d,0)$ has $\s(g)\in D$, so $\s(g)$ is a generator, or the inverse of a generator,
of the factor $\FF(D^+)$ of $F_D$; so $\s(g)$ is not a proper power. Consequently  there are no $t$-powers in $L(d,0)$.  
We shall count the number of elements of $L(d,k)$ which are $t$-powers and are in $\Li$ and
in $\Lii$ %type \ref{it:composed1} and of \ref{it:composed2}
separately, and show that both
have asymptotic density zero, from which it follows that the set of elements $Z_4^c$ of $L$ which do not satisfy condition \ref{it:main_hyp4} 
%$t$-roots in $L$
 has asymptotic density zero. 
%
%so both are negligible subsets
%relative to $L$; and so the set of $t$-roots is negligible relative to $L$.
% Let $r_1(d,k)$ and $r_2(d,k)$ denote the number of $t$-powers of type type \ref{it:composed1} and of \ref{it:composed2}, respectively, in $L(d,k)$.
  Elements $ut^{\pm k}$ of $\Li$ are all $t$-powers. There are $2[1+d(d+1)]$ of these in $L(d,k)$, from \eqref{eq:l1dk},  and 
  % $|L(d,k)|=l_1(d,k)+l_2(d,k)$
  it follows from  \eqref{eq:liidk} and the definition of $L(d,k)$, that %$\lim_{{d,k}\maps \infty} r_1(d,k)/|L(d,k)|=0$, so
 the set of $t$-powers that are in $\Li$ has asymptotic density zero.

 Let $p^{(ii)}(d,k)$ denote the number of $t$-powers in $\Lii(d,k)$. 
 If $w=ug_1t^{\a_r}\cdots g_rt^{\a_r}\in \Lii_S(d,k,r,P)$ %is of type \ref{it:composed2}, say
 is a $t$-power then there are elements $d_i\in D$ and $u_i\in U$ and an integer $p$,  such that
 $1\le p\le \floor{r/2}$, $r=pq$, $w=_F ud_1u_1t^{\a_1}\cdots d_ru_rt^{\a_r}$ and $\s(w)=d_1t^{\a_1}\cdots d_rt^{\a_r}=(d_1t^{\a_1}\cdots d_pt^{\a_p})^q$.
 (The $d_i$ are fixed generators of the free group $\FF(D^+)$ (or their inverses).)
%We call this a $t$-root of \emph{exponent} $q$.
  In this case it follows  that $P=(|\a_1|,\ldots ,|\a_r|)=P_p^q$, where $P_p=(|\a_1|,\ldots ,|\a_p|)$ is a composition of $q=r/p$ into $p$ parts.
 Conversely, if there exist positive integers $p$, $q$ such that $p=r/q$, $2\le q\le r$ and $P=P_p^q$ for some composition $P_p$ of $q$ into $p$ parts, then for all
 \begin{itemize}
   \item elements 
     $d_i\in D$ and  $u, u_{i,j} \in U$ such that $1\le i\le p$ and $1\le j\le q$, with 
     \item $l(ud_1u_{1,1})\le d$, $l(d_iu_{i,j})\le d$ and $[u_{i,j},d_i]\neq 1$, and
     \item all $(\a_1,\ldots, \a_p)$ such that  $P_p=(|\a_1|,\ldots, |\a_p|)$;
     \end{itemize}
     \[
       w=ud_1u_{1,1}t^{\a_1}d_2u_{1,2}t^{\a_2}\cdots d_pu_{1,p}t^{\a_p}\cdots t^{\a_p}d_1u_{q,1}t^{\a_1}\cdots d_pu_{q,p}t^{\a_p}
     \]
     is a $t$-power in $L(d,k,r,P)$.
     Here $w$ has initial subword $w_p=  ud_1u_{1,1}t^{\a_1}d_2u_{1,2}t^{\a_2}\cdots d_pu_{1,p}t^{\a_p}$ with normal form in $\Lii_S(d,q,p,P_p)$ and
     $\s(w_p)$ is the $q$th root of $\s(w)$. Therefore the $t$-power $w$ is determined by the element $w_p\in \Lii_S(d,q,p,P_p)$ and the
     sequence $(u_{2,1},\ldots ,u_{2,p},\ldots u_{q,1},\ldots, u_{q,p})$ of length $p(q-1)=k-p$, of elements of $U$ (each of length at most $d$).
     There are $(4d)^{k-p}$ sequences of length $k-p$, of such elements of $U$, so writing 
     $p^{(ii)}(d,k,r,P)$ for the number of $t$-powers in $\Lii_S(d,k,r,P)$, and assuming an appropriate composition $P_p$ of $r/p$, exists for each
     $p$ such that $1\le p\le \floor{r/2}$, we have
     \begin{align}
       p^{(ii)}(d,k,r,P)&\le \sum_{p=1}^{\floor{r/2}}(4d)^{k-p}\lii_S(d,q,p,P_p)\notag\\
                   &= 2^kl_H(d)\sum_{p=1}^{\floor{r/2}}(2d)^{k-p}(l_H^U(d))^{p-1}\textrm{ from } \eqref{eq:liidkrP},\notag\\
       &\le 2^k(2d)^{k-\floor{r/2}}l_H(d)\frac{(l_H^U(d))^{\floor{r/2}}}{l_H^U(d)-2d},\notag
     \end{align}
\begin{comment}
       &\textcolor{black}{\le\left(\frac{n-1}{n-3}\right)(\a^d-1)\sum_{p=1}^{\floor{r/2}}(4d)^{k-p}2^{p+1}l(d)^{p-1}\textrm{ from \eqref{eq:ldk2u}}}\notag\\
                   &\textcolor{black}{\le (4d)^k\left(\frac{n-1}{n-3}\right)(\a^d-1)\sum_{p=1}^{\floor{r/2}}l(d)^{p-1}}\notag\\
                   &\textcolor{black}{\le (4d)^k\left(\frac{n-1}{n-3}\right)(\a^d-1)\left(\frac{l_H(d)^{\floor{r/2}}-1}{l_H(d)-1}\right).}\label{al:tpowerdkP}
\end{comment}
                 Summing over all compositions of $k$ into $r$ parts and over $r$ from $1$ to $k$, and setting $a=l_H^U(d)$, $b=l_H(d)$ and $c=2d$,
                 the number $p^{(ii)}_S(d,k)$ of $t$-powers in
$\Lii_S(d,k)$ satisfies
\begin{align}
  p^{(ii)}_S(d,k)&\le \frac{2^k b}{a-c}\sum_{r=1}^k \binom{k-1}{r-1} c^{k-\floor{r/2}}a^{\floor{r/2}}\notag\\
                 &\le \frac{2^k b}{a-c}\sum_{r=1}^k \binom{k-1}{r-1} c^{(2k-r+1)/2}a^{r/2}\notag\\
                 &= \frac{2^k a^{1/2}bc^{(k+1)/2}}{a-c}\sum_{r=1}^k \binom{k-1}{r-1} c^{(k-r)/2}a^{(r-1)/2}\notag\\
                 &= \frac{2^k a^{1/2}bc^{(k+1)/2}}{a-c}(a^{1/2} + c^{1/2})^{k-1}.\notag
\end{align}
Therefore the number $p^{(ii)}(d,k)$ of $t$-powers in $\Lii(d,k)$, satisfies
\begin{align}
  p^{(ii)}(d,k)&\le \frac{a^{1/2}b}{a-c}\sum_{p=1}^k 2^pc^{(p+1)/2}(a^{1/2} + c^{1/2})^{p-1}\notag\\
               &=\frac{2a^{1/2}bc}{a-c}\sum_{p=1}^k [2c^{1/2}(a^{1/2} + c^{1/2})]^{p-1}\notag\\
  &\le\frac{2a^{1/2}bc[2c^{1/2}(a^{1/2} + c^{1/2})]^{k}}{(a-c)[2c^{1/2}(a^{1/2} + c^{1/2})-1]}.\notag
\end{align}
In current notation, from \eqref{eq:liidk}, $l(d,k)\ge \lii(d,k)=b[(2a+1)^{k}-1]/a$, so
\begin{align*}
  \frac{p^{(ii)}(d,k)}{l(d,k)}&\le \frac{2a^{1/2}bc[2c^{1/2}(a^{1/2} + c^{1/2})]^{k}a}{(a-c)[2c^{1/2}(a^{1/2} + c^{1/2})-1]b[(2a+1)^{k}-1]}\notag\\
                              &=\frac{2^{k+1}a^{3/2}c^{(k+2)/2}[a^{1/2} + c^{1/2}]^{k}}{(a-c)[2c^{1/2}(a^{1/2} + c^{1/2})-1][(2a+1)^{k}-1]}\notag\\
                              &\le \frac{a^{3/2}c[(ca)^{1/2}+c]^k}{(a-c)[c^{1/2}(a^{1/2}+c^{1/2})-1/2]a^k}\\
  &=\frac{(c/a)}{[1-(c/a)][c^{1/2}(1+(c/a)^{1/2})-(1/2a^{1/2})]}[(c/a^{1/2})+(c/a)]^k.
\end{align*}
%%%%%%%%%%%%%%%%%%%%%%
For $n=5$ and $n\ge 6$, $c/a^{1/2}$ tends to $0$ as $d\maps \infty$ and so,  
%
%Setting $x=1/2a$, $y=c/a$ and $z=1/(ac)$ the latter can be rewritten as
%\[ \frac{p^{(ii)}(d,k)}{l(d,k)}= \frac{y^{(k+2)/2}[1+y^{1/2}]^k}{[1-y][1+y^{1/2}-(z^{1/2}/2)][(1+x^{1/2)})^k-1]},\]
%where $x=1/l_H^U(d)$, $y=2d/l_H^U(d)$ and $z=1/(2d l^U_H(d))$. When $n=5$ we have $l_H^U(d)=3^{d-1}(3+2d)$ and, when $n\ge 6$ we have $l_H^U(d)\ge \g^d-1$,
%with $\g\ge 5$. In
 in both cases, it can be seen that $\lim_{d,k\maps \infty}\frac{p^{(ii)}(d,k)}{l(d,k)}=0$. Therefore the set of $t$-powers in $\Lii$
has asymptotic density $0$ and we conclude that $\rho(Z_4)=1$. 

The conclusion is that  the set of elements for which all the hypotheses of Theorem \ref{thm:main} hold has asymptotic density $1$. That is 
 Theorem \ref{thm:main} holds on a generic subset of $L$. 

\end{document}